\documentclass[12pt,a4paper]{amsart}
\usepackage[english]{babel}
\usepackage[T1]{fontenc}
\usepackage[latin1]{inputenc}
\usepackage{graphicx}
\usepackage{amsmath,amssymb,amsthm,mathrsfs,txfonts}
\usepackage{soul}
\usepackage{array, multirow}
\usepackage{stmaryrd}
\usepackage{mathrsfs}
\usepackage{color}
\usepackage{hyperref}
\usepackage{enumerate}
\usepackage[all]{xy}
\usepackage[left=2.5cm, right=2.5cm, top=3cm, bottom=3cm]{geometry}

\theoremstyle{plain}
\newtheorem{theo}{Theorem}[section]
\newtheorem{lemme}[theo]{Lemma}
\newtheorem{prop}[theo]{Proposition}
\newtheorem{coro}[theo]{Corollary}

\theoremstyle{definition}
\newtheorem{defn}[theo]{Definition}

\newtheorem{nota}[theo]{Notation}

\theoremstyle{remark}
\newtheorem{rema}[theo]{Remark}

\def\Sp{\mathop{\rm Sp}\nolimits}
\def\J{\mathop{\rm J}\nolimits}
\def\U{\mathop{\rm U}\nolimits}
\def\M{\mathop{\rm M}\nolimits}
\def\det{\mathop{\rm det}\nolimits}
\def\O{\mathop{\rm O}\nolimits}
\def\GL{\mathop{\rm GL}\nolimits}	
\def\End{\mathop{\rm End}\nolimits}

\def\tr{\mathop{\rm tr}\nolimits}
\def\G{\mathop{\rm G}\nolimits}

\def\S{\mathop{\rm S}\nolimits}
\def\exp{\mathop{\rm exp}\nolimits}
\def\dim{\mathop{\rm dim}\nolimits}
\def\diag{\mathop{\rm diag}\nolimits}

\def\Im{\mathop{\rm Im}\nolimits}
\def\Op{\mathop{\rm Op}\nolimits}
\def\Hom{\mathop{\rm Hom}\nolimits}
\def\HS{\mathop{\rm H.S}\nolimits}
\def\supp{\mathop{\rm supp}\nolimits}
\def\Ad{\mathop{\rm Ad}\nolimits}
\def\rk{\mathop{\rm rk}\nolimits}
\def\H{\mathop{\rm H}\nolimits}
\def\Re{\mathop{\rm Re}\nolimits}
\def\reg{\mathop{\rm reg}\nolimits}
\def\Z{\mathop{\rm Z}\nolimits}
\def\ch{\mathop{\rm ch}\nolimits}
\def\sgn{\mathop{\rm sgn}\nolimits}
\def\K{\mathop{\rm K}\nolimits}
\def\pr{\mathop{\rm pr}\nolimits}
\def\P{\mathop{\rm P}\nolimits}
\def\Id{\mathop{\rm Id}\nolimits}
\def\Mat{\mathop{\rm Mat}\nolimits}
\def\E{\mathop{\rm E}\nolimits}
\def\A{\mathop{\rm A}\nolimits}
\def\T{\mathop{\rm T}\nolimits}

\title{Characters of some unitary highest weight representations via the theta correspondence}

\author{ALLAN MERINO \\ NATIONAL UNIVERSITY OF SINGAPORE \\ matafm@nus.edu.sg}

\date{}

\allowdisplaybreaks
	
\begin{document}

\maketitle

\begin{abstract}

\noindent In this article, we consider a dual pair $(\G, \G')$ in the symplectic group $\Sp(W)$ with $\G$ compact and let $(\tilde{\G}, \tilde{\G}')$ be the preimages of $\G$ and $\G'$ in the metaplectic group $\widetilde{\Sp(W)}$. For every irreducible representation $\Pi$ of $\tilde{\G}$ appearing in Howe correspondence, we compute explicitly the restriction of the character $\Theta_{\Pi'}$ of the associated representation $\Pi'$ of $\tilde{\G}'$ on the set of regular points on the compact Cartan subgroup $\tilde{\H}'$ of $\tilde{\G}'$.

\end{abstract}

\tableofcontents

\section{Introduction}

\noindent One of the first results concerning characters of representations was established by Weyl around 1925. Let $\G$ be a compact connected real Lie group and $\H$ a maximal torus in $\G$. We denote by $\mathfrak{h}_{0}$ and $\mathfrak{g}_{0}$ the Lie algebras of $\H$ and $\G$ respectively, and by $\mathfrak{h}$ and $\mathfrak{g}$ their complexifications. Let $\mathscr{W}(\mathfrak{g},\mathfrak{h})$ and $\Phi^+(\mathfrak{g},\mathfrak{h})$ be respectively the Weyl group  and a choice of positive roots of $(\mathfrak{g},\mathfrak{h})$. We denote by $\varepsilon(\sigma)$ the signature of $\sigma \in \mathscr{W}(\mathfrak{g},\mathfrak{h})$. Weyl proved that, for an irreducible representation $(\Pi, V_{\Pi})$ of $\G$, the character $\Theta_{\Pi}$ is given by 
\begin{equation}
\Theta_{\Pi}(\exp(x)) = \sum\limits_{\sigma \in \mathscr{W}(\mathfrak{g}, \mathfrak{h})} \varepsilon(\sigma) \cfrac{e^{\sigma(\lambda + \rho)}}{\prod\limits_{\alpha \in \Phi^{+}(\mathfrak{g}, \mathfrak{h})} (e^{\frac{\alpha(x)}{2}} - e^{-\frac{\alpha(x)}{2}})} \qquad (x \in \mathfrak{h}_{0}),
\label{WeylIntroduction}
\end{equation}
where $\lambda \in \mathfrak{h}^{*}$ is the highest weight of $\Pi$ and $\rho = \frac{1}{2} \sum\limits_{\alpha \in \Phi^{+}(\mathfrak{g}, \mathfrak{h})} \alpha$. 

\noindent He also proved that the character determines the representation up to equivalence, i.e. two representations are equivalent if and only if their characters are equal. 

\noindent In the early 1950's, Harish-Chandra developed the notion of character for a certain class of representations of a real reductive group $\G$, called quasi-simple \cite[Section~1]{HAR1}. The character $\Theta_\Pi$ of a quasi-simple representation $(\Pi,\mathscr{H})$ of $\G$ is a distribution on $\G$, i.e. a well defined and continuous map
\begin{equation}
\Theta_{\Pi}: \mathscr{C}^{\infty}_{c}(\G) \ni \Psi \mapsto \tr \Pi(\Psi) \in \mathbb{C}.
\end{equation}
He also proved that this distribution was given by a function, also denoted by $\Theta_{\Pi}$, on the open and dense subset of $\G$ consisting of the so-called regular elements. More precisely, for all function $\Psi \in \mathscr{C}^{\infty}_{c}(\G)$, we get:
\begin{equation}
\Theta_{\Pi}(\Psi) = \displaystyle\int_{\G} \Theta_{\Pi}(g) \Psi(g) d\mu_{\G}(g).
\end{equation}
Harish-Chandra also determined a character formula for the discrete series representations. This formula turns out to be formally very similar to Weyl's formula \eqref{WeylIntroduction}; see \cite{HARISH1} and \cite{HARISH2}.

\noindent Many other mathematicians have been interested in characters formulas. We can for example mention Kirillov \cite{KIR}, who proposed an integral formula for a general Lie group and Enright \cite{ENR}, who, using cohomological methods, established a formula for the character of a unitary highest weight module of a simple Lie group. Let us say some words on Enright's work. Let $(\G, \K)$ be a hermitian symmetric pair with $\G$ simple and such that $\rk(\K) = \rk(\G)$. Let $\H$ be a maximal torus in $\K$ and $(\Pi, \mathscr{H}_{\Pi})$ a representation of $\G$ of highest weight $\lambda - \rho$. Then, the character $\Theta_{\Pi}$ is given by:
\begin{equation}
\prod\limits_{\alpha \in \Phi^{+}(\mathfrak{g}, \mathfrak{h})}(e^{\frac{\alpha(x)}{2}} - e^{-\frac{\alpha(x)}{2}}) \Theta_{\Pi}(\exp(x)) = \sum\limits_{\omega \in \mathscr{W}^{\mathfrak{k}}_{\lambda}} (-1)^{l_{\lambda}(\omega)} \Theta(\K, \Lambda(\omega, \lambda))(\exp(x)) \qquad (x \in \mathfrak{h}^{\reg}),
\end{equation}
where $\mathscr{W}^{\mathfrak{k}}_{\lambda}$ is defined in \cite[Definition~2.1]{ENR}, and $\Theta(\K, \Lambda(\omega, \lambda))(\exp(x))$ is the character of a $\K$-representation of highest weight $\Lambda(\omega, \lambda)$, with $\Lambda(\omega, \lambda)$ as in \cite[Corollary~2.3]{ENR}.

\noindent In this article, we determine character formulas of unitary highest weight representations of non-compact reductive Lie groups, by using the theta correspondence. More precisely, we fix a dual pair $(\G, \G')$ in the symplectic group $\Sp(W)$ and we assume that $\G$ is compact. Let $\tilde{\G}$ and $\tilde{\G}'$ be the preimages of $\G$ and $\G'$ in the metaplectic group $\Sp(W)$ and let $\Pi \leftrightarrow \Pi'$ be two representations appearing in the correspondence. By a result of Przebinda, we know that the pullback of the character $\Theta_{\Pi'}$ via the Cayley transform is given by 
\begin{equation}
c^{*}_{-} \Theta_{\Pi'}(\varphi) = T(\overline{\Theta_{\Pi}})(\phi) \qquad (\varphi \in \mathscr{C}^{\infty}_{c}(\mathfrak{g}')),
\end{equation}
where $T: \widetilde{\Sp(W)} \to \S^{*}(W)$ is the embedding of $\widetilde{\Sp(W)}$ in $\S^{*}(W)$ given in section \ref{MetaplecticRepresentation} and 
\begin{equation}
\phi(w) = \mathscr{F}(\varphi \Theta \circ \tilde{c}) \circ \tau_{\mathfrak{g}'}
\label{FameuseFonctionphi}
\end{equation} 
where $\mathscr{F}: \S(\mathfrak{g}') \mapsto \S(\mathfrak{g}'^{*})$ is the Fourier transform defined in section \ref{Intertwining}.

\noindent In \cite{TOM4}, the authors computed precisely the distribution $\T(\overline{\Theta_{\Pi}})$ for the three compact dual pairs. With this distribution, they could compute the wave front set of the representation $\Pi'$. Using these results, we determine in this paper the function $\Theta_{\Pi'}$ on $\tilde{\G}'^{\reg}$, or more precisely on $\tilde{\H}'^{\reg}$, where $\H'$ is the compact Cartan subgroup of $\G'$. 

\noindent This paper is organized as follows. In section \ref{MetaplecticRepresentation}, we recall the construction of the metaplectic representation as recently given by Aubert and Przebinda in \cite{TOM1}, and quickly review Howe's duality theorem. In section \ref{Intertwining}, we explain how the character of $\Pi'$ can be obtained from the one of $\Pi$ using \cite{TOM3}. We then recall in section \ref{SectionFourier} the notion of Fourier transform of a coadjoint orbit, notion established by Rossmann for a regular parameter \cite{ROS} and by Duflo and Vergne in the general case \cite{DUFLO}. In section \ref{SectionTitreU(n)}, we consider first the case $(\G = \U(n, \mathbb{C}), \G' = \U(p, q, \mathbb{C}))$. The main result of this section is Theorem \ref{TheoremU(p,q)}. In particular, the character $\Theta_{\Pi'}$ is given by:
\begin{equation}
\Theta_{\Pi'}(h) = \sum\limits_{\eta \in \mathscr{W}(\G, \mathfrak{h}, m)/ \mathscr{W}(\G, \mathfrak{h})_{m}} \sum\limits_{\sigma \in \mathscr{W}(\K', \mathfrak{h}')} \cfrac{\sgn(\eta) \pr_{m}(\sigma h)^{-\eta^{-1}\mu'}}{\prod\limits_{\underset{\beta_{|_{\mathfrak{h}'(m)} \neq 0}}{\beta \in \Phi^{+}(\mathfrak{g}', \mathfrak{h}')}}((\sigma h)^{\frac{\beta}{2}} - (\sigma h)^{-\frac{\beta}{2}})}
\end{equation}

\noindent We obtain similar formulas for the pair $(\G = \O(2m, \mathbb{R}), \G' = \Sp(2n, \mathbb{R}))$ in section \ref{O(2n,R)}, $(\G = \O(2m+1, \mathbb{R}), \G' = \Sp(2n, \mathbb{R}))$ in section \ref{SectionO(2n+1)} and for $(\G = \U(n, \mathbb{H}), \G' = \O^{*}(m, \mathbb{H}))$ in section \ref{SectionU(n,H)}. 

\noindent In this paper, we always work under the condition $\rk(\G) \leq \rk(\G')$. Moreover, we assume that the character $\Theta_{\Pi'}$ is supported in $\tilde{\G}'_{1}$, where $\G'_{1}$ is the Zariski identity component of $\G'$. In particular, this eliminates some representations of the even orthogonal group.


\bigskip

\noindent \textbf{Acknowledgements:} A part of this paper was done during my thesis at the University of Lorraine under the supervision of Angela Pasquale (University of Lorraine) and Tomasz Przebinda (University of Oklahoma). I would like to thank them for the ideas and time they shared with me. I also would like to thank the two referees of my thesis Hadi Salmasian (University of Ottawa) and Hung Yean Loke (National University of Singapore) for their suggestions. I finished this paper during my first months at the National University of Singapore, as a Research Fellow under the supervision of Hung Yean Loke.

\section{Weil representation and Howe's duality theorem}

\label{MetaplecticRepresentation}

\noindent We start this section by recalling some main ideas of the construction of the metaplectic representation given in \cite{TOM1} (and by introducing some notation from \cite{TOM2}). 

\noindent Fix a finite dimensional real vector space $W$ endowed with a symplectic non-degenerate form $\langle \cdot,\cdot\rangle$. We denote by $\Sp(W)$ the corresponding symplectic group and by $\mathfrak{sp}(W)$ the Lie algebra, respectively defined as
\begin{equation}
\Sp(W) = \left\{g \in \GL(W), \left<g(w), g(w')\right> = \left<w, w'\right>, \forall w, w' \in W \right\}
\end{equation} 
\noindent and
\begin{equation}
\mathfrak{sp}(W) = \left\{X \in \End(W), \left<X(w), w'\right> + \left<w, X(w')\right> = 0, \forall w, w' \in W \right\}.
\end{equation}
Let $J \in \mathfrak{sp}(W)$ be an endomorphism satisfying $J^{2} = -\Id_{W}$ and such that the symmetric bilinear form $\beta = \left<J\cdot,\cdot\right>$ is positive-definite (such an element is called a compatible positive complex structure on $W$). For all $g \in \Sp(W)$, we denote by $J_{g}$  the automorphism of $W$ given by $J_{g}(w) = J^{-1}(g-1)$. For all $u, v \in W$, we have:
\begin{eqnarray*}
\beta(J_{g}u, v) & = & \left<J(J_{g}u), v\right> = \left<(g-1)u, v\right> = \left<u, (g^{-1} - 1)v\right> \\
                    & = & \left<u, JJ^{-1}(g^{-1}-1)v\right> = -\left<Ju, J^{-1}(g^{-1}-1)v\right>  = \left<Ju, J(g^{-1}-1)v\right> \\
                    & = & \beta(u, Jg^{-1}(1-g)v).
\end{eqnarray*}
Hence, the adjoint of $J_{g}$ with respect to the form $\beta$ is given by $J^{*}_{g} = Jg^{-1}(1-g)$. One can easily show that the restriction of $J_{g}$ to the space $\Im(J_{g}) = J_{g}(W)$ is well defined and invertible. 

\noindent As in \cite{TOM1}, let 
\begin{equation}
\widetilde{\Sp(W)} = \left\{\tilde{g} = (g, \xi) \in \Sp(W) \times \mathbb{C}, \xi^{2} = i^{\dim_{\mathbb{R}}(g-1)W} \det(J_{g})^{-1}_{J_{g}(W)}\right\}.
\end{equation}
The multiplication on $\widetilde{\Sp(W)}$ is given by 
\begin{equation}
(g_{1}, \xi_{1}).(g_{2}, \xi_{2}) = (g_{1}g_{2}, \xi_{1}\xi_{2}C(g_{1}, g_{2})).
\end{equation}
where $C$ is a cocycle on $\Sp(W)$ defined in \cite[Proposition~4.13]{TOM1}.

\noindent According to \cite[equation~(3)]{TOM2}, the absolute value of $C$ is given by
\begin{equation}
|C(g_{1}, g_{2})| = \sqrt{\left|\cfrac{\det(J_{g_{1}})_{J_{g_{1}}(W)}\det(J_{g_{2}})_{J_{g_{2}}(W)}}{\det(J_{g_{1}g_{2}})_{J_{g_{1}g_{2}}(W)}}\right|} \qquad (g_{1}, g_{2} \in \Sp(W)).
\end{equation}
Fix a non trivial unitary character $\chi$ of $\mathbb{R}$ and let $W = X \oplus Y$ be a complete polarization of the symplectic space $W$. We denote by $\S(X)$ and $\S(W)$ the Schwartz space on $X$ and $W$ respectively and let $\S^{*}(X)$ and $\S^{*}(W)$ the corresponding spaces of tempered distributions. 

\noindent For a subset $A$ of $\End(W)$, we set $A^{c} = \{X \in A, \det(X-1) \neq 0\}$. The Cayley transform is well defined on $\mathfrak{sp}^{c}(W)$ and $c(\mathfrak{sp}^{c}(W)) = \Sp^{c}(W)$. For all $g \in \Sp(W)$, we define the function $\chi_{c(g)}: (g-1)W \mapsto \mathbb{C}$ by
\begin{equation}
\chi_{c(g)}(u) = \chi\left(\frac{1}{4} \left<c(g)u, u\right>\right) \qquad (u \in (g-1)W),
\end{equation}
and let $t: \Sp(W) \mapsto \S^{*}(W)$, $\Theta: \widetilde{\Sp(W)} \mapsto \mathbb{C}$ and $T: \widetilde{\Sp(W)} \mapsto \S^{*}(W)$ be the maps defined respectively by:
\begin{equation}
t(g) = \chi_{c(g)} \mu_{(g-1)W}, \qquad \Theta(\tilde{g}) = \xi, \qquad T(\tilde{g}) = \Theta(\tilde{g})t(g).
\end{equation}
where $\tilde{g} = (g, \xi) \in \widetilde{\Sp(W)}$ and $\mu_{(g-1)W} \in \S^{*}(W)$ is the Lebesgue measure on the space $(g-1)W$ normalized with respect to the form $\beta$ such that the volume of the associated unit cube is $1$.

\noindent The Weyl Transform $\mathscr{K}: \S(W) \mapsto \S(X \times X)$ defined by
\begin{equation}
\mathscr{K}(f)(x, x') = \displaystyle\int_{Y} f(x - x' + y)\chi\left(\frac{1}{2} \left<y, x + x'\right>\right) d\mu_{Y}(y)
\end{equation}
is an isomorphism of linear topological spaces. Its natural extension, also denoted by $\mathscr{K}$, to the space of tempered distributions is still an isomorphism. Recall the map $\Op: \S(X \times X) \mapsto \Hom(\S(X), \S(X))$ given by
\begin{equation}
\Op(K)v(x) = \displaystyle\int_{X} K(x, x') v(x') d\mu_{X}(x').
\end{equation}
By the Schwartz Kernel Theorem, its extension $\Op : \S^{*}(X \times X) \mapsto \Hom(\S(X), \S^{*}(X))$ is an isomorphism of topological vector spaces.

\noindent Finally, we get a map $\omega: \widetilde{\Sp(W)} \mapsto \Hom(\S(X), \S^{*}(X))$ defined by
\begin{equation}
\omega(\tilde{g}) = \Op \circ \mathscr{K} \circ T (\tilde{g}).
\end{equation}
As explained in \cite[Section~4.8]{TOM1}, the restriction $\Op \circ \mathscr{K}: L^{2}(W) \to \HS(L^{2}(X))$, where $\HS(L^{2}(X))$ is the space of Hilbert Schmidt operators on $L^{2}(X)$ is an isometry and the map
\begin{equation}
\omega: \widetilde{\Sp(W)} \mapsto \U(L^{2}(X))
\end{equation}
is a unitary representation of the metaplectic group, often called Weil representation, or metaplectic representation (\cite[Section~4.8,~Theorem~4.27]{TOM1}). Moreover, the function $\Theta$ defined as above coincides with the character of $\omega$; it means that for all $\Psi \in \mathscr{C}^{\infty}_{c}(\widetilde{\Sp(W)})$, the following equality holds
\begin{equation}
\displaystyle\int_{\widetilde{\Sp(W)}} \Theta(\tilde{g}) \Psi(\tilde{g}) d\mu_{\widetilde{\Sp(W)}}(\tilde{g}) = \tr \displaystyle\int_{\widetilde{\Sp(W)}} \omega(\tilde{g}) \Psi(\tilde{g}) d\mu_{\widetilde{\Sp(W)}}(\tilde{g}).
\end{equation}

\noindent One can show that the Garding space of this representation is exactly the space of Schwartz functions on $X$.

\bigskip

\noindent By a reductive dual pair in $\Sp(W)$, we mean a pair of subgroups $(\G, \G')$ of $\Sp(W)$ which act reductively on the symplectic space $W$ and  are mutually centralizers in $\Sp(W)$. The pair $(\G, \G')$ is said irreducible if there is no orthogonal decomposition $W = W_{1} \oplus W_{2}$ with respect to the symplectic form $\langle \cdot, \cdot \rangle$ such that $W_{1}$ and $W_{2}$ are both $\G.\G'$-invariants. A classification of such pairs was done by Howe. In the next sections of this article, we will be interested in irreducible reductive dual pairs with one member compact. We recall here the three cases in which we will be interested in:
\begin{itemize}
\item $(\O(n, \mathbb{R}), \Sp(2m, \mathbb{R})) \subseteq \Sp(2nm, \mathbb{R})$,
\item $(\U(n, \mathbb{C}), \U(p, q, \mathbb{C})) \subseteq \Sp(2n(p+q), \mathbb{R})$,
\item $(\U(n, \mathbb{H}), \O^{*}(m, \mathbb{H})) \subseteq \Sp(4nm, \mathbb{R})$.
\end{itemize}
For a reductive dual pair $(\G, \G') \in \Sp(W)$, we denote by $\tilde{\G}$ and $\tilde{\G}'$ the preimages of $\G$ and $\G'$ respectively in $\widetilde{\Sp(W)}$. 

\noindent For any subgroup $\tilde{\H}$ of $\widetilde{\Sp(W)}$, we denote by $\mathscr{R}(\tilde{\H})$ the set of infinitesimal equivalence classes of continuous irreducible admissible representations of $\tilde{\H}$ on locally convex topological vector spaces. Let $\mathscr{R}(\tilde{\H}, \omega^{\infty})$ be the subset of $\mathscr{R}(\tilde{\H})$ of representations which can be realized as a quotient of $\S(X)$ by a closed $\omega^{\infty}(\tilde{\H})$-invariant subspace (here, $\omega^{\infty}$ is the $\mathscr{C}^{\infty}$-representation of $(\omega, L^{2}(X))$ realized on the space $\S(X)$).

\noindent According to \cite[Theorem~1]{HOW1}, for any reductive dual pair $(\G, \G') \in \Sp(W)$, there is a bijective correspondence between $\mathscr{R}(\tilde{\G}, \omega^{\infty})$ and $\mathscr{R}(\tilde{\G}', \omega^{\infty})$ having graph $\mathscr{R}(\tilde{\G}.\tilde{\G}', \omega^{\infty})$.

\begin{rema}

Assume that $\G$ is compact. Let $(\Pi, V_{\Pi})$ be an irreducible (unitary) representation of $\mathscr{R}(\tilde{\G}, \omega^{\infty})$. Then, $(\Pi, V_{\Pi})$ is a subrepresentation of $(\omega^{\infty}, \S(X))$ and in particular, its isotypic component $V(\Pi)$ is a closed subspace of $\S(X)$. Let $(\Pi', V_{\Pi'})$ be the associated representation in $\mathscr{R}(\tilde{\G}', \omega^{\infty})$. Using Howe's duality theorem, we get that $V(\Pi) = V(\Pi') = V(\Pi \otimes \Pi') \subseteq \S(X)$, where $V(\Pi')$ is the $\Pi'$-isotypic component. In particular,
\begin{equation}
\omega^{\infty} = \bigoplus\limits_{\Pi \in \mathscr{R}(\tilde{\G}, \omega^{\infty})} \Pi \otimes \Pi'.
\end{equation}
where here the sum is not an algebraic sum but the closure of the algebraic sum with respect to the topology of $\S(X)$.
\label{RemarkGCompact}
\end{rema}

\section{The intertwining distribution}

\label{Intertwining}

\noindent We start this section by recalling the notion of intertwining distribution introduced by Przebinda in \cite{TOM3}.  

\noindent Let $(\G, \G')$ be an irreducible reductive dual pair in $\Sp(W)$, $\Pi \in \mathscr{R}(\tilde{\G}, \omega^{\infty})$ and $\Pi' \in \mathscr{R}(\tilde{\G}', \omega^{\infty})$ such that $\Pi \otimes \Pi' \in \mathscr{R}(\tilde{\G}.\tilde{\G}', \omega^{\infty})$. There exists a closed $\omega^{\infty}(\tilde{\G}.\tilde{\G}')$-invariant subspace $\mathscr{N}$ of $\S(X)$ such that 
\begin{equation}
\Pi \otimes \Pi' \approx \S(X) / \mathscr{N}.
\end{equation}
We denote by $(\Pi \otimes \Pi')^{*}$ the contragredient representation of $\Pi \otimes \Pi'$. We get:
\begin{equation}
(\Pi \otimes \Pi')^{*} \approx (\S(X) / \mathscr{N})^{*} \approx Ann(\mathscr{N}) = \{f \in \S^{*}(X), f(v) = 0, (\forall v \in \mathscr{N})\} \subseteq \S^{*}(X).
\end{equation}
Choosing $\Pi^{*}$ and $\Pi'^{*}$ instead of $\Pi$ and $\Pi'$, we obtain that the representation $\Pi \otimes \Pi'$ can be realized in $\S^{*}(X)$. Let $\S^{*}(X)_{\Pi \otimes \Pi'}$ be the $\Pi \otimes \Pi'$-isotypic component. By Howe's duality theorem, we get a unique, up to a constant, element $\Gamma_{\Pi \otimes \Pi'} \in \Hom_{\tilde{\G}.\tilde{\G}'}(\S(X), \S^{*}(X))$ such that the following equation holds
\begin{equation}
\Gamma_{\Pi \otimes \Pi'} \circ \omega^{\infty}(\tilde{g}\tilde{g}') = \Pi(\tilde{g}) \otimes \Pi'(\tilde{g}') \circ \Gamma_{\Pi \otimes \Pi'} \qquad (\tilde{g} \in \tilde{\G}, \tilde{g}' \in \tilde{\G}').
\end{equation}
Using the isomorphisms $\mathscr{K}$ and $\Op$ defined in the section \ref{MetaplecticRepresentation}, there exists a distribution $f_{\Pi \otimes \Pi'} \in \S^{*}(W)$ verifying $\Gamma_{\Pi \otimes \Pi'} = \Op \circ \mathscr{K}(f_{\Pi \otimes \Pi'})$. The element $f_{\Pi \otimes \Pi'} \in \S^{*}(W)$ is called the intertwining distribution in \cite[Section~5]{TOM3}.

\begin{rema}

If $\G$ is compact, according to remark \ref{RemarkGCompact}, we get that the $\Pi \otimes \Pi'$-isotypic component is a closed subspace of $\S(X)$. We denote by $\mathscr{P}_{\Pi \otimes \Pi'}$ the projection onto the isotypic component, i.e. the map $\mathscr{P}_{\Pi \otimes \Pi'}: \S(X) \mapsto \S(X)_{\Pi \otimes \Pi'}$ given by:
\begin{equation}
\mathscr{P}_{\Pi \otimes \Pi'} = \frac{1}{\dim(\Pi)} \displaystyle\int_{\tilde{\G}} \overline{\Theta_{\Pi}(\tilde{g})} \omega(\tilde{g}) d\mu_{\tilde{\G}}(\tilde{g}).
\end{equation}
Then, $\Gamma_{\Pi \otimes \Pi'} = \mathscr{P}_{\Pi \otimes \Pi'}$.

\end{rema}

\noindent For now on, we assume that the group $\G$ is compact. Using \cite[Lemma~5.4]{TOM3}, we get that the distribution $f_{\Pi \otimes \Pi'}$ is given by the formula
\begin{equation}
f_{\Pi \otimes \Pi'} = \displaystyle\int_{\tilde{\G}} \overline{\Theta_{\Pi}(\tilde{g})}T(\tilde{g}) d\mu_{\tilde{\G}}(\tilde{g}) = T(\overline{\Theta_{\Pi}}),
\end{equation}
where the function $\Theta_{\Pi}: \tilde{\G} \to \mathbb{C}$ is the character of $\Pi$. Denote by $\Theta_{\Pi'}$ the distribution character of $\Pi'$ and the function defined on the regular points on $\tilde{\G}'$. This distribution can be expressed via $f_{\Pi \otimes \Pi'}$ (\cite[Theorem~6.7]{TOM3}). To explain this link, we'll use the notations of \cite{TOM3}. Denote by $\mathscr{F}: \S(\mathfrak{g}') \to \S(\mathfrak{g}'^{*})$ the Fourier transform defined as
\begin{equation}
\mathscr{F}(\Psi)(\eta) = \displaystyle\int_{\mathfrak{g}'} \Psi(\xi) \chi(\eta(\xi))d\xi \qquad (\Psi \in \S(\mathfrak{g}'), \eta \in \mathfrak{g}'^{*}),
\end{equation} 
and let $\mathscr{F}^{*} : \S^{*}(\mathfrak{g}'^{*}) \to \S^{*}(\mathfrak{g}')$ be the dual map.

\noindent As explained in \cite[Section~3]{TOM3}, for a fixed element $\widetilde{(-1)} \in \widetilde{\Sp(W)}$ in the preimage of $(-1) \in \Sp(W)$, there exists a unique map $\tilde{c}: \mathfrak{sp}(W) \to \widetilde{\Sp^{c}(W)}$ such that $\tilde{c}(0) = \widetilde{(-1)}$ and $\tilde{c} \circ \pi = c$.

\noindent We denote by $j_{\mathfrak{sp}}$ the function defined of the domain of $\tilde{c}$ which satisfies
\begin{equation}
\displaystyle\int_{\widetilde{\Sp(W)}} \Psi(\tilde{g}) d\mu_{\widetilde{\Sp(W)}}(\tilde{g}) = \displaystyle\int_{\mathfrak{sp}(W)} \Psi(\tilde{c}(x))j_{\mathfrak{sp}}(x) dx.
\end{equation}
for all functions $\Psi \in \mathscr{C}^{\infty}_{c}(\Sp^{c}(W))$ such that $\supp(\Psi) \subseteq \Im(\tilde{c})$. Denote by $j_{\mathfrak{g}}$ and $j_{\mathfrak{g}'}$ the associated maps on $\mathfrak{g}$ and $\mathfrak{g}'$ respectively, and let $\tau_{\mathfrak{g}'}: W \to \mathfrak{g}'^{*}$ be the map defined by
\begin{equation}
\tau_{\mathfrak{g}'}(w)(x') = \frac{1}{4}\left<x'(w), w\right> \qquad (w \in W, x' \in \mathfrak{g}'^{*}).
\end{equation}
As proved in \cite[Lemma~6.1]{TOM3}, the pullback of $\tau_{\mathfrak{g}'}$, from $\S(\mathfrak{g}'^{*})$ onto $\S(W)$, given by $\psi \to \psi \circ \tau_{\mathfrak{g}'}$, is well defined and continuous. By dualization, we get a map $(\tau_{\mathfrak{g}'})_{*}: \S^{*}(W) \to \S^{*}(\mathfrak{g}'^{*})$, defined by
\begin{equation}
(\tau_{\mathfrak{g}'})_{*}(f)(\psi) = f(\psi \circ \tau_{\mathfrak{g}'}) \qquad (f \in \S^{*}(W), \psi \in \S(\mathfrak{g}'^{*})).
\end{equation}
In \cite[Theorem~6.7]{TOM3}, Przebinda proved the following result:
\begin{equation}
\cfrac{1}{\Theta \circ \tilde{c}} \tilde{c}^{*}_{-}\Theta_{\Pi'} = K_{\Pi} \mathscr{F}^{*}((\tau')_{*}(T(\overline{\Theta_{\Pi}}))),
\end{equation}
where $c_{-}(x) = \tilde{c}(x)\tilde{c}(0)^{-1}$, $\tilde{c}^{*}_{-}\Theta_{\Pi'}$ is the pullback of $\Theta_{\Pi'}$ via $\tilde{c}_{-}$ and $K_{\Pi} \in \mathbb{C}$ is a constant.

\begin{rema}

Here we explain formally where this formula comes from. For every function $\Psi \in \mathscr{C}^{\infty}_{c}(\tilde{\G}')$, the distribution character is given by the following formula
\begin{equation}
\Theta_{\Pi'}(\Psi) = \tr(\mathscr{P}_{\Pi \otimes \Pi'} \omega(\Psi)) = \tr(\omega(d^{-1}_{\Pi}\overline{\Theta_{\Pi}}) \omega(\Psi)).
\end{equation}
where $d_{\Pi} = \dim(V_{\Pi})$. As shown in \cite[Section~4.8]{TOM1}, we have:
\begin{equation}
\omega(\Theta_{\Pi}) \circ \omega(\Psi) = \Op \circ \mathscr{K}(T(\Theta_{\Pi}) \natural T(\Psi)),
\end{equation}
where $\natural$ is the twisted convolution of distributions defined in \cite[Section~4.5~and~Lemma~4.24]{TOM1}. Using \cite[Section~4.8]{TOM1}, we get:
\begin{eqnarray*}
\tr \Op \circ \mathscr{K}(T(\Theta_{\Pi}) \natural T(\Psi)) & = & T(\Theta_{\Pi}) \natural T(\Psi)(0) = \displaystyle\int_{W} T(\Theta_{\Pi})(-w)T(\Psi)(w) dw \\
& = & \displaystyle\int_{W} T(\Theta_{\Pi})(w)T(\Psi)(w) dw.
\end{eqnarray*}
Finally, for all $\Psi \in \mathscr{C}^{\infty}_{c}(\tilde{\G}')$, we get:
\begin{equation}
\Theta_{\Pi'}(\Psi) = T(\overline{\Theta_{\Pi}})(T(\Psi))
\end{equation}
Even though $\Im(T) \subseteq \S^{*}(W)$, $T(\Psi)$ is a Schwartz function on $W$. In \cite{TOM4}, the authors computed $T(\overline{\Theta_{\Pi}})(\phi)$ for all compact dual pairs in $\Sp(W)$, where $\phi$ is a Schwartz function on $W$. We will use their results for a particular function $\phi \in \S(W)$ to get an explicit formula for the function $\Theta_{\Pi'}$ on $\tilde{\H}'^{\reg}$, where $\H'$ is a compact Cartan subgroup of $\G'$.

\end{rema}

\section{Fourier transform of a co-adjoint orbit}

\label{SectionFourier}

\noindent To recall the concept of Fourier transform of a co-adjoint orbit, we use \cite[Chapter~7,~Section~5]{VER}. For a compact group, this result is essentially due to Harish-Chandra and Kirillov. For a general semi-simple Lie group, Rossmann \cite{ROS} established a formula for a regular semi-simple element in $\mathfrak{g}^{*}$. This result had been generalized by M. Duflo and M. Vergne in \cite{DUFLO} for every $\lambda \in \mathfrak{g}^{*}$.

\noindent For a semi-simple Lie group $\G$ with maximal compact subgroup $\K$ satisfying $\rk(\K) = \rk(\G)$, denote by $\Ad: \G \to \GL(\mathfrak{g})$ the adjoint action and by $\Ad^{*}$ the associated action of $\G$ on $\mathfrak{g}^{*}$. For $\lambda \in \mathfrak{g}^{*}$, we note by $\G_{\lambda} = \Ad^{*}(\G)(\lambda)$ the $\G$-orbit associated to $\lambda$. Let $\H$ be a maximal torus in $\K$ (which is a maximal torus in $\G$ because of the rank equality), and fix $\Phi(\mathfrak{k}, \mathfrak{h})$ and $\Phi(\mathfrak{g}, \mathfrak{h})$ the roots corresponding to $(\mathfrak{k}, \mathfrak{h})$ and $(\mathfrak{g}, \mathfrak{h})$ respectively. Let $(\cdot, \cdot)$ be a $\G$-invariant symmetric non-degenerate form on $\mathfrak{g}$. More precisely, for each $\alpha \in \mathfrak{h}^{*}$, there exists a unique element $H_{\alpha} \in \mathfrak{h}$ such that $\alpha(h) = (H_{\alpha}, h)$ for all $h \in \mathfrak{h}$.

\noindent For $\lambda \in \mathfrak{h}^{*}$, we denote by $\P_{\lambda}$ the subset of $\Phi(\mathfrak{g}, \mathfrak{h})$ defined by
\begin{equation}
\P_{\lambda} = \left\{\alpha \in \Phi(\mathfrak{g}, \mathfrak{h}), \lambda(iH_{\alpha}) > 0\right\}.
\end{equation}

\noindent We denote by $d\beta_{\lambda}$ the Liouville measure on $\G_{\lambda}$. According to \cite[page~170]{BOU}, this measure is given by the following formula
\begin{equation}
\displaystyle\int_{\G_{\lambda}} \phi d\beta_{\lambda} = \prod\limits_{\alpha \in \P_{\lambda}}\lambda(iH_{\alpha}) \displaystyle\int_{\G/\G^{\lambda}} \phi(g\lambda) d(g\G^{\lambda}),
\label{Liouville}
\end{equation}
\noindent where $\G^{\lambda}$ is the stabilizer of $\lambda$. Here, $g\lambda$ means $\Ad^{*}(g)(\lambda)$. The Fourier transform of $\G_{\lambda}$, is the generalized function on $\mathfrak{g}$
\begin{equation}
F_{\G_{\lambda}}(X) = \displaystyle\int_{\G_{\lambda}} e^{if(X)} d\beta_{\lambda}(f) \qquad (X \in \mathfrak{h}^{\reg}).
\end{equation}
Using \cite[Theorem~7.29]{VER}, for all $X \in \mathfrak{h}^{\reg}$, we get:
\begin{equation}
F_{\G_{\lambda}}(X) = (-1)^{n(\lambda)} \sum\limits_{w \in \mathscr{W}(\mathfrak{k}, \mathfrak{h}) / \mathscr{W}(\mathfrak{k}, \mathfrak{h})^{\lambda}} \cfrac{e^{iw\lambda(X)}}{\prod\limits_{\alpha \in \P_{\lambda}} w\alpha(X)},
\label{equation1}
\end{equation}
where $n(\lambda)$ is the number of non-compact roots of $\P_{\lambda}$, $\mathscr{W}(\mathfrak{k}, \mathfrak{h}) = \left<s_{\alpha}, \alpha \in \Phi(\mathfrak{k}, \mathfrak{h})\right>$ is the compact Weyl group and $\mathscr{W}(\mathfrak{k}, \mathfrak{h})^{\lambda}$ is the stabilizer of the $\lambda$ under the action of $\mathscr{W}(\mathfrak{k}, \mathfrak{h})$:
\begin{equation}
\mathscr{W}(\mathfrak{k}, \mathfrak{h})^{\lambda} = \{\sigma \in \mathscr{W}(\mathfrak{k}, \mathfrak{h}), \sigma(\lambda) = \lambda\}.
\end{equation}
Using equations \eqref{Liouville} and \eqref{equation1}, we get:
\begin{equation}
\prod\limits_{\alpha \in \P_{\lambda}}\lambda(iH_{\alpha}) \displaystyle\int_{\G/\G^{\lambda}} e^{ig\lambda(X)} d(g\G^{\lambda}) = (-1)^{n(\lambda)} \sum\limits_{w \in \mathscr{W}(\mathfrak{k}, \mathfrak{h}) / \mathscr{W}(\mathfrak{k}, \mathfrak{h})^{\lambda}} \cfrac{e^{iw\lambda(X)}}{\prod\limits_{\alpha \in \P_{\lambda}} w\alpha(X)}.
\label{equation3}
\end{equation}
The equality \eqref{equation3} is what we call the Rossmann-Duflo-Vergne formula.

\begin{rema}

\begin{enumerate}
\item Assume that $\G$ is compact and connected, and let $(\Pi, V_{\Pi})$ be an irreducible representation of Harish-Chandra parameter $\lambda \in \mathfrak{h}^{*}$. Then, $\lambda$ is regular and using equation \eqref{equation1}, we get the classical Weyl character formula. 
\item We assume that $\G$ is semi-simple. In particular, the Killing form on $\mathfrak{g}$ is non-degenerate, which allows us to identify $\mathfrak{g}$ with its dual $\mathfrak{g}^{*}$. We denote by $(\cdot, \cdot)$ the Killing form on $\mathfrak{g}$. Then, the equation \eqref{equation3} can be written as:
 \begin{equation}
 \prod\limits_{\alpha \in \P_{t}} i\alpha(t) \displaystyle\int_{\G/\G^{t}} e^{i(gt, s)} d\bar{g} = (-1)^{n(\lambda)} \sum\limits_{w \in \mathscr{W}(\mathfrak{k}, \mathfrak{h}) / \mathscr{W}(\mathfrak{k}, \mathfrak{h})^{t}} \cfrac{e^{(s, wt)}}{\prod\limits_{\alpha \in \P_{t}} w\alpha(s)},
\label{Vergne2}
\end{equation}
where $s \in \mathfrak{h}^{\reg}$, $t \in \mathfrak{h}$ and $\P_{t} = \left\{\alpha \in \Phi(\mathfrak{g}, \mathfrak{h}) \thinspace | \thinspace i\alpha(t) > 0\right\}$.
\end{enumerate}

\end{rema}

\section{Character formula for the pair $(\G = \U(n, \mathbb{C}), \G' = \U(p, q, \mathbb{C}))$, $n \leq p+q$}

\label{SectionTitreU(n)}

\noindent Let $V_{\overline{0}}$ be a complex vector space of dimension $n$ over $\mathbb{C}$ endowed with a definite-positive hermitian form $b_{0}$. We fix a basis $\mathscr{B} = \{v_{1}, \ldots, v_{n}\}$ of $V_{\overline{0}}$ such that the matrix of $b_{0}$ with respect to $\mathscr{B}$ is $\Mat(b_{0}, \mathscr{B}) = \Id_{n}$, and let $\U(V_{\overline{0}}, b_{0})$ be the group of isometries of the form $b_{0}$, i.e.
\begin{equation}
\U(V_{\overline{0}}, b_{0}) = \left\{g \in \GL(V_{\overline{0}}), b_{0}(gu, gv) = b_{0}(u, v), (\forall u, v \in V_{\overline{0}})\right\}.
\end{equation}
We denote by $\mathfrak{g}_{0} = \mathfrak{u}(V_{\overline{0}}, b_{0})$ the Lie algebra of $\U(V_{\overline{0}}, b_{0})$ given by
\begin{equation}
\mathfrak{u}(V_{\overline{0}}, b_{0}) = \left\{X \in \End(V_{\overline{0}}), b_{0}(Xu, v) + b_{0}(u, Xv) = 0, (\forall u, v \in V_{\overline{0}})\right\}.
\end{equation}
Writing the endomorphism $X$ in the basis $\mathscr{B}$, the Lie algebra can be realized as:
\begin{equation}
\mathfrak{u}(V_{\overline{0}}, b_{0}) = \left\{X \in \M(n, \mathbb{C}), X^{*} + X = 0\right\}.
\end{equation}
Let $\mathfrak{h}_{0}$ the diagonal subalgebra of $\mathfrak{u}(V_{\overline{0}}, b_{0})$ given by:
\begin{equation}
\mathfrak{h}_{0} = \bigoplus\limits_{k = 1}^n \mathbb{R} i\E_{k, k},
\end{equation}
and we denote by $\mathfrak{h}$ and $\mathfrak{g}$ their complexifications. The roots of $\mathfrak{g}$ with respect to $\mathfrak{h}$ are given by:
\begin{equation}
\Phi(\mathfrak{g}, \mathfrak{h}) = \left\{\pm (e_{i} - e_{j}), 1 \leq i < j \leq n\right\},
\end{equation}
where the linear form $e_{k}$ is given by
\begin{equation}
e_{k}\left(\sum\limits_{k =  1}^n ih_{k}\E_{k, k}\right) = ih_{k}.
\end{equation}
We denote by $\pi_{\mathfrak{g}/\mathfrak{h}}$ the product of all positive roots:
\begin{equation}
\pi_{\mathfrak{g}/\mathfrak{h}} = \prod\limits_{\alpha \in \Phi^{+}(\mathfrak{g}, \mathfrak{h})} \alpha.
\end{equation}
For all $h = \sum\limits_{k = 1}^n ih_{k}\E_{k, k} \in \mathfrak{h}$, we get:
\begin{equation}
\pi_{\mathfrak{g}/\mathfrak{h}}(h) = \prod\limits_{1 \leq i < j \leq n}(e_{i} - e_{j})(h) = \prod\limits_{1 \leq i < j \leq n}(ih_{i} - ih_{j}).
\end{equation}

\noindent Now, we consider a complex vector space $V_{\overline{1}}$ endowed with a non-degenerate hermitian form $b_{1}$ of signature $(p, q)$. Fix a basis $\mathscr{B}' = \{w_{1}, \ldots, w_{p+q}\}$ of $V_{\overline{1}}$ such that $\Mat(b_{1}, \mathscr{B}') = i\Id_{p, q}$, where $\Id_{p, q} = \begin{pmatrix} \Id_{p} & 0 \\ 0 & -\Id_{q} \end{pmatrix}$ and we denote by $\U(V_{\overline{1}}, b_{1})$ the group of isometries of the form $b_{1}$, i.e.
\begin{equation}
\U(V_{\overline{1}}, b_{1}) = \left\{g \in \GL(V_{\overline{1}}), b_{1}(gu, gv) = b_{1}(u, v), (\forall u, v \in V_{\overline{1}})\right\}.
\end{equation}
By writing the endomorphisms $X$ in the basis $\mathscr{B}'$, we get the following realization for the Lie algebra
\begin{equation}
\mathfrak{u}(V_{\overline{1}}, b_{1}) = \left\{X \in \M(p+q, \mathbb{C}), \Id_{p, q}X + X^{t}\Id_{p, q} = 0\right\}.
\end{equation}
Again, we consider the diagonal subspace $\mathfrak{h}'_{0}$ of $\mathfrak{u}(V_{\overline{1}}, b_{1})$ 
\begin{equation}
\mathfrak{h}'_{0} = \bigoplus\limits_{k = 1}^{p+q} \mathbb{R} i\E_{k, k},
\end{equation}
and let $\Phi(\mathfrak{g}', \mathfrak{h}')$ be the set of roots of $\mathfrak{g}'$ given by:
\begin{equation}
\Phi(\mathfrak{g}', \mathfrak{h}') = \left\{\pm (e_{i} - e_{j}), 1 \leq i < j \leq p+q\right\}.
\end{equation}
Let $\K = \U(p, \mathbb{C}) \times \U(q, \mathbb{C})$ be the maximal compact subgroup of $\U(p, q, \mathbb{C})$, $\mathfrak{k}_{0}$ the Lie algebra of $\K$ and $\mathfrak{k}$ its complexification. The set of compact roots of $\mathfrak{g}'$ are given by:
\begin{equation}
\Phi(\mathfrak{k}', \mathfrak{h}') = \{\pm (e_{i} - e_{j}), 1 \leq i < j \leq p\} \cup \{\pm (e_{i} - e_{j}), p+1 \leq i < j \leq p+q\}.
\end{equation}
For all $h' = \sum\limits_{k = 1}^{p+q} ih_{k}\E_{k, k} \in \mathfrak{h}'$, we have
\begin{equation}
\pi_{\mathfrak{g}'/\mathfrak{h}'}(h') = \prod\limits_{\alpha \in \Phi^{+}(\mathfrak{g}', \mathfrak{h}')} \alpha(h') = \prod\limits_{1 \leq i < j \leq p+q} (ih_{i} - ih_{j})
\end{equation}
and
\begin{equation}
\pi_{\mathfrak{k}'/\mathfrak{h}'}(h') = \prod\limits_{\alpha \in \Phi^{+}(\mathfrak{k}', \mathfrak{h}')} \alpha(h') = \prod\limits_{1 \leq i < j \leq p} (ih_{i} - ih_{j}) \prod\limits_{p+1 \leq i < j \leq p+q} (ih_{i} - ih_{j}).
\end{equation}

\noindent We now use the correspondence between dual pairs and Lie supergroups. Let $V = V_{\overline{0}} \oplus V_{\overline{1}}$, $b = b_{0} \oplus b_{1}$ as defined in appendix \ref{AppendixA}, and let $(\S, \mathfrak{s} = \mathfrak{s}(V, b))$ be the associated supergroup.

\begin{lemme}

An element $X = \begin{pmatrix} 0 & X_{1} \\ X_{2} & 0 \end{pmatrix}$ is in $\mathfrak{s}(V, b)_{\overline{1}}$ if and only if $X_{2} = -i\Id_{p, q} \bar{X_{1}}^{t}$.

\end{lemme}

\begin{proof}

As explained in appendix \ref{AppendixA}, an element $X$ is in $\mathfrak{s}(V, b)_{\overline{1}}$ if and only if $b(X(u), v) - b(u, sX(v)) = 0$ for all $u = (u_{0}, u_{1}), v = (v_{0}, v_{1})$. We have:
\begin{equation}
Xu = \begin{pmatrix} 0 & X_{1} \\ X_{2} & 0 \end{pmatrix} \begin{pmatrix} u_{0} \\ u_{1} \end{pmatrix}  = \begin{pmatrix} X_{1}u_{1} \\ X_{2}u_{0} \end{pmatrix} \qquad sX(v) = \begin{pmatrix} 1 & 0 \\ 0 & -1 \end{pmatrix} \begin{pmatrix} 0 & X_{1} \\ X_{2} & 0 \end{pmatrix} \begin{pmatrix} v_{0} \\ v_{1} \end{pmatrix}  = \begin{pmatrix} X_{1}v_{1} \\ -X_{2}v_{0} \end{pmatrix}.
\end{equation}
In particular, 
\begin{eqnarray*}
 0 = b(X(u), v) - b(u, sX(v)) & = & (v^{*}_{0}, v^{*}_{1}) \begin{pmatrix} \Id_{n} & 0 \\ 0 & i\Id_{p, q}\end{pmatrix} \begin{pmatrix} X_{1}u_{1} \\ X_{2}u_{0} \end{pmatrix} - (v^{*}_{1}X^{*}_{1}, -v^{*}_{0}X^{*}_{2})\begin{pmatrix} \Id_{n} & 0 \\ 0 & i\Id_{p, q} \end{pmatrix} \begin{pmatrix} u_{0} \\ u_{1} \end{pmatrix} \\
                                   & = & (v^{*}_{0}, v^{*}_{1}i\Id_{p, q})\begin{pmatrix} X_{1}u_{1} \\ X_{2}u_{0} \end{pmatrix} - (v^{*}_{1}X^{*}_{1}, -v^{*}_{0}X^{*}_{2}i\Id_{p, q}) \begin{pmatrix} u_{0} \\ u_{1} \end{pmatrix} \\
                                   & = & v^{*}_{0}X_{1}u_{1} + v^{*}_{1}i\Id_{p, q}X_{2}u_{0} - v^{*}_{1}X^{*}_{1}u_{0} + v^{*}_{0}X^{*}_{2}i\Id_{p, q}u_{1} \\
                                   & = & v^{*}_{0}(X_{1} + X^{*}_{2}i\Id_{p, q})u_{1} + v^{*}_{1}(i\Id_{p, q}X_{2} - X^{*}_{1})u_{0}
\end{eqnarray*}
We get the result by taking $v_{0} = 0$ and remarking that $\Id^{-1}_{p,q} = \Id_{p, q}$.

\end{proof}

\noindent We define $V^{j}_{\overline{0}} = \mathbb{C}u_{j}$ and $V^{j}_{\overline{1}} = \mathbb{C}w_{j}$. Then, 
\begin{equation}
V_{\overline{0}} = \bigoplus\limits_{k=1}^n \mathbb{C}u_{k} = \bigoplus\limits_{k=1}^n V^{k}_{\overline{0}} \qquad V_{\overline{1}} = \bigoplus\limits_{k=1}^{p+q} \mathbb{C}w_{k} = \bigoplus\limits_{k=1}^{p+q} V^{k}_{\overline{1}},
\end{equation}
and $V = \left(\bigoplus\limits_{k=1}^n V^{k}_{\overline{0}} \oplus  V^{k}_{\overline{1}}\right) \oplus \bigoplus\limits_{k=n+1}^{p+q} V^{k}_{\overline{1}}$.

\noindent Now, we fix an integer $m \in [|\max(n-q, 0), \min(p, n)|]$. We define the endomorphisms $u_{j} \in \mathfrak{s}(V, b)_{\overline{1}}$ by:
\begin{equation}
u_{j}(v_{k}) = \delta_{j, k} e^{-\frac{i\pi}{4}}w_{k}, \qquad u_{j}(w_{k}) = \delta_{j, k} e^{-\frac{i\pi}{4}}v_{k}, \qquad (1 \leq j \leq m),
\end{equation}
\begin{equation}
u_{j}(v_{k}) = \delta_{j, k} e^{-\frac{i\pi}{4}}w_{p+q-n+k}, \qquad u_{j}(w_{p+q-n+k}) = v_{k}, \qquad (m+1 \leq j \leq n),
\end{equation}
with $u_{j}(w_{m+1}) = \ldots = u_{j}(w_{p+q-n+m}) = 0$. Let $\mathfrak{h}_{1, m}$ be the subspace of $\mathfrak{s}(V, b)_{\overline{1}}$ given by
\begin{equation}
\mathfrak{h}_{1, m} = \sum\limits_{j = 1}^n \mathbb{R} u_{j}.
\end{equation}
We define $\tau: \mathfrak{s}(V, b)_{\overline{1}} \mapsto \mathfrak{g}$ and $\tau': \mathfrak{s}(V, b)_{\overline{1}} \mapsto \mathfrak{g}'$ by
\begin{equation}
\tau(w) = w^{2}_{|_{V_{\overline{0}}}}, \qquad \tau'(w) = w^{2}_{|_{V_{\overline{1}}}} \qquad (w \in \mathfrak{s}(V, b)_{\overline{1}}).
\end{equation}
For all $w = \sum\limits_{k=1}^n h_{k}u_{k} \in \mathfrak{h}_{1, m}$, we have:
\begin{equation}
\tau(w) = \sum\limits_{j=1}^n -ih^{2}_{j}\E_{j, j}, \qquad \tau'(w) = \sum\limits_{j = 1}^m -ih^{2}_{j}\E_{j, j} + \sum\limits_{j = m+1}^n -ih^{2}_{j}\E_{p+q-n+j, p+q-n+j}.
\end{equation}
In particular, we get an injection of $\mathfrak{h}$ into $\mathfrak{h}'$, depending on $m$, given by:
\begin{equation}
\mathfrak{h} \ni \sum\limits_{j = 1}^n ih_{j}\E_{j, j} \mapsto \sum\limits_{j = 1}^m ih_{j}\E_{j, j} + \sum\limits_{j = m+1}^n ih_{j}\E_{p+q-n+j, p+q-n+j} \in \mathfrak{h}'.
\label{Embedding1}
\end{equation}

\noindent We denote by $\mathfrak{h}'(m)$ the image of $\mathfrak{h}$ into $\mathfrak{h}'$ via the map \eqref{Embedding1}. Let $\mathfrak{z}'(m) = C_{\mathfrak{g}'}\mathfrak{h}'(m)$ and
\begin{equation}
\pi_{\mathfrak{g}'/\mathfrak{z}'(m)} = \prod\limits_{\underset{\alpha_{|_{\mathfrak{h}'(m)} \neq 0}}{\alpha \in \Phi^{+}(\mathfrak{g}', \mathfrak{h}')}} \alpha, \qquad \pi_{\mathfrak{z}'(m)/\mathfrak{h}'} = \prod\limits_{\underset{\alpha_{|_{\mathfrak{h}'(m)} = 0}}{\alpha \in \Phi^{+}(\mathfrak{g}', \mathfrak{h}')}} \alpha.
\end{equation}


\noindent We denote by $\mathscr{W}(\G, \mathfrak{h})$ the Weyl group of $(\mathfrak{g}, \mathfrak{h})$. It acts on the elements $h = \sum\limits_{k=1}^n ih_{k}\E_{k, k} \in \mathfrak{h}$ by permutations of the components $\left\{h_{k}, 1 \leq i \leq n\right\}$. We will use the same symbol to indicate the Weyl group element and the corresponding permutation in $\mathscr{S}_{n}$, i.e.
\begin{equation}
\sigma.\left(\sum\limits_{k=1}^n ih_{k}\E_{k, k}\right) = \sum\limits_{k=1}^n ih_{\sigma(k)}\E_{k, k} \qquad (\sigma \in \mathscr{S}_{n}).
\end{equation}
Similarly, for $h' = \sum\limits_{k=1}^{p+q} ih'_{k}\E_{k, k} \in \mathfrak{h}'$, the Weyl group $\mathscr{W}(\G', \mathfrak{h}')$ of $(\mathfrak{g}', \mathfrak{h}')$ permutes the components $h'_{1}, \ldots, h'_{p+q}$. Moreover, we let $\mathscr{W}(\K', \mathfrak{h}')$ be the compact Weyl group, i.e. the subgroup of $\mathscr{W}(\G', \mathfrak{h}')$ which acts on the sets $\left\{h'_{i}, 1 \leq i \leq p\right\}$ and $\left\{h'_{i}, p+1 \leq i \leq p+q\right\}$. 

\noindent For every $m \in [|\max(n-q, 0), \min(p, n)|]$ and $h = \sum\limits_{i=1}^n ih_{k}\E_{k, k} \in \mathfrak{h}$, we denote by $\mathscr{W}(\G, \mathfrak{h})_{m}$ the subgroup of $\mathscr{W}(\G, \mathfrak{h})$ which acts on the sets $\left\{h_{i}, 1 \leq i \leq m\right\}$ and $\left\{h_{i}, m+1 \leq i \leq n\right\}$, and by $\mathscr{W}(\K', \mathfrak{h}')_{m}$ the subgroup of $\mathscr{W}(\K', \mathfrak{h}')$ which permutes the elements of $\left\{h'_{i}, 1 \leq i \leq m\right\}$ and $\left\{h'_{p+q+m-n+i}, 1 \leq i \leq n-m\right\}$ separately.

\noindent Finally, for all $x, y \in \mathfrak{g}$ or $\mathfrak{g}'$, we denote by $B$ the bilinear form defined by
\begin{equation}
B(x, y) = \Re\tr(xy).
\label{FormeB}
\end{equation}

\begin{rema}

The form $B$ is $\G$ (resp. $\G'$)-invariant and non-degenerate on $\mathfrak{g}$ (resp. $\mathfrak{g}'$). More precisely, for all $x = \sum\limits_{k=1}^n ix_{k}\E_{k, k}$ (resp. $x' = \sum\limits_{k=1}^{p+q} ix'_{k}\E_{k, k}$) and $y = \sum\limits_{k=1}^n iy_{k}\E_{k, k}$ (resp. $y' = \sum\limits_{k=1}^{p+q} iy'_{k}\E_{k, k}$), the form $B$ is given by
\begin{equation}
B(x, y) = \sum\limits_{k=1}^n -\pi x_{k} y_{k} \qquad \left(\text{resp. } B(x', y') = \sum\limits_{k=1}^{p+q} -\pi x'_{k} y'_{k}\right).
\end{equation}
See \cite[page~49-50]{TOM4}.

\end{rema}


\subsection{Harish-Chandra's semisimple orbital integral}

\label{Section5.1}

As introduced in \cite[Section~4,~Definition~10]{TOM4}, for every function $\phi \in \S(W)$, we define the function $f_{\phi}$ on $\tau(\mathfrak{h}^{\reg}_{1, m})$ by 
\begin{equation}
f_{\phi}(\tau(w)) := C_{\mathfrak{h}_{1, m}} \pi_{\mathfrak{g}'/\mathfrak{z}'(m)}(\tau'(w)) \displaystyle\int_{\G'/\Z'(m)} \phi(g'w)d(g'Z'(m)),
\end{equation}
where $C_{\mathfrak{h}_{1, m}}$ is the constant of modulus 1 given in \cite[Lemma~8,~page~17]{TOM4}, which, a priori, depends on $m$.

\begin{lemme}

The constant $C_{\mathfrak{h}_{1, m}}$ does not depend on $m$.

\end{lemme}

\begin{proof}

As explained in \cite{TOM4}, the constant $C_{\mathfrak{h}_{1, m}}$ satisfies, for all $w = \sum\limits_{k=1}^n w_{k}u_{k} \in \mathfrak{h}^{\reg}_{1, m}$, the following equality:
\begin{equation}
|\pi_{\mathfrak{g}/\mathfrak{h}}(\tau(w)) \pi_{\mathfrak{g}'/\mathfrak{z}'(m)}(\tau'(w))| = C_{\mathfrak{h}_{1, m}}\pi_{\mathfrak{g}/\mathfrak{h}}(\tau(w)) \pi_{\mathfrak{g}'/\mathfrak{z}'(m)}(\tau'(w)). 
\end{equation}
In particular, we have:
\begin{equation}
\tau(w) = \sum\limits_{k=1}^n -iw^{2}_{k}\E_{k, k}, \qquad \tau'(w) = \sum\limits_{k=1}^m -iw^{2}_{k}\E_{k, k} + \sum\limits_{k=p+q-n+m+1}^{p+q} -iw^{2}_{k}\E_{k, k},
\end{equation}
and
\begin{equation}
\pi_{\mathfrak{g}/\mathfrak{h}}(\tau(w)) = \prod\limits_{1 \leq i < j \leq n} (-iw^{2}_{i} + iw^{2}_{j}) = (-i)^{\frac{(n-1)n}{2}} |\pi_{\mathfrak{g}/\mathfrak{h}}(\tau(w))|.
\end{equation}
Similarly, we have:
\begin{equation*}
\pi_{\mathfrak{g}'/\mathfrak{z}'(m)}(\tau'(w)) = 
\end{equation*}
{\small
\begin{equation}
\prod\limits_{1 \leq i < j \leq m} (-iw^{2}_{i} + iw^{2}_{j}) \prod\limits_{p+q-n+m+1 \leq i < j \leq p+q} (-iw^{2}_{i} + iw^{2}_{j}) \prod\limits_{k=1}^m (-iw^{2}_{k})^{p+q-n} \prod\limits_{k=p+q-n+m+1}^{p+q}(iw^{2}_{k})^{p+q-n} \prod\limits_{i=1}^m \prod\limits_{j=p+q-n+m+1}^{p+q} (-iw^{2}_{i}+iw^{2}_{j}).
\end{equation}}
In particular,
\begin{equation}
\pi_{\mathfrak{g}'/\mathfrak{z}'(m)}(\tau'(w)) = (-i)^{\frac{n^{2}-n}{2} + n(p+q-n)} |\pi_{\mathfrak{g}'/\mathfrak{z}'(m)}(\tau'(w))|.
\end{equation}
\end{proof}

\noindent As we have seen in section \ref{Intertwining}, the distribution character is defined, up to a constant, by the distribution $T(\overline{\Theta_{\Pi}})$. Before giving a formula for $T(\overline{\Theta_{\Pi}})(\phi), \phi \in \S(W)$, we introduce some notations. Let $\H$ be a Cartan subgroup of $\G$ and let $\G^{\sharp}$ be a double cover of $\tilde{\G}$ as defined in \cite[Section~2]{TOM4}. Let $\H^{\sharp}$ be the preimage of $\H$ in $\G^{\sharp}$ and let $\H^{\sharp}_{0}$ be the connected identity component of $\H^{\sharp}$. For every roots $\alpha \in \Phi(\mathfrak{g}, \mathfrak{h})$, $\frac{\alpha}{2}$ is analytic integral on $\H^{\sharp}_{0}$ (see \cite[Proposition~4.58]{KNA}). It means that there exists a character $\xi^{\sharp}_{\frac{\alpha}{2}}: \H^{\sharp}_{0} \mapsto \S^{1}$ having the linear form $\frac{\alpha}{2}$ as derivative. More generally, for every $\mu \in \mathfrak{h}^{*}$ which is analytic integral on $\G^{\sharp}$, we denote by $\xi^{\sharp}_{\mu}: \H^{\sharp}_{0} \mapsto \S^{1}$ the corresponding character.

\begin{rema}

In particular, $\rho = \frac{1}{2} \sum\limits_{\alpha > 0} \alpha$ is analytic integral on $\H^{\sharp}_{0}$. Then, the Weyl denominator is well defined on $\H^{\sharp}_{0}$. More precisely, this function is analytic on $\H^{\sharp}$.

\end{rema}

\noindent We denote by $c^{\sharp}_{-}: \mathfrak{h} \to \H^{\sharp}_{0}$ the lift of the map $c_{-}: \mathfrak{h} \to \H$ to $\H^{\sharp}_{0}$. For every analytic integral form $\mu \in \mathfrak{h}^{*}$, we note by $c_{-}(\cdot)^{\mu}$ the function defined on $\H^{\sharp}_{0}$ by
\begin{equation}
c_{-}(x)^{\mu} = \xi^{\sharp}_{\mu}(c^{\sharp}_{-}(x)) \qquad (x \in \H^{\sharp}_{0}).
\end{equation}

\begin{nota}

For all $n \in \mathbb{N}$, we indicate by $\Psi^{n}$ the embedding of $\M(n, \mathbb{C})$ onto $\M(2n, \mathbb{R})$ given by:
\begin{equation}
\Psi^{n}(A+iB) = \begin{pmatrix} A & -B \\ B & A \end{pmatrix} \qquad (A, B \in \M(n, \mathbb{R})).
\label{EmbeddingComplex}
\end{equation}
For an endomorphism $X = A + iB$ of $\M(n, \mathbb{C})$, we denote by $\det_{\mathbb{R}}(X)$ the determinant of the matrix $\Psi^{n}(X) \in \M(2n, \mathbb{R})$.

\end{nota}

\begin{prop}

Fix $\Pi \in \mathscr{R}(\tilde{\G}, \omega^{\infty})$ of highest weight $\nu$ and let $\mu = \nu + \rho$. For every $\phi \in \S(W)$, the following formula holds:
\begin{equation}
T(\overline{\Theta_{\Pi}})(\phi) = \displaystyle\int_{\mathfrak{h}}c_{-}(x)^{\mu'}\ch(x)^{p+q-n-1} \displaystyle\int_{\mathfrak{h} \cap \tau(W)} e^{iB(x, y)}f_{\phi}(y)dydx,
\label{Proposition1908}
\end{equation}
where $\ch(x) = |\det_{\mathbb{R}}(x-1)|^{\frac{1}{2}}$ and $\mu' \in \mathfrak{h}^{*}$ is defined by
\begin{equation}
\mu'_{j} = \mu_{n+1-j} \qquad (1 \leq j \leq n).
\end{equation}

\end{prop}

\begin{proof}

The proof can be found in \cite[Corollary~38,~page~47]{TOM4}.

\end{proof}

\begin{rema}
\begin{enumerate}
\item By \cite{VERGNE2} (one can also consult the appendix of \cite{TOM7}), the weights of the representations of $\widetilde{\U(n,\mathbb{C})}$ which appear in the Howe correspondence are the $\nu$'s of the form:
\begin{equation}
\nu = \sum\limits_{a=1}^n \cfrac{q - p}{2} e_{a} - \sum\limits_{a = 1}^r \lambda_{a} e_{m+1-a} + \sum\limits_{a=1}^s \mu_{a} e_{a},
\end{equation}
where $0 \leqslant r \leqslant p$, $0 \leqslant s \leqslant q$, $r+s \leqslant n$, and where $\lambda_{1}, \ldots, \lambda_{r}, \mu_{1}, \ldots, \mu_{s}$ are integers which satisfy $\lambda_{1} \geqslant \ldots \geqslant \lambda_{r} > 0$ and $\mu_{1} \geqslant \ldots \geqslant \mu_{s} > 0$.
The weights $\nu$ can then be written as follows
\begin{equation}
\nu = \sum\limits_{a=1}^n \left(\cfrac{q - p}{2} + \nu_{a}\right) e_{a}
\end{equation}
where $\nu_{i} \in \mathbb{Z}$, $\nu_{1} \geq \ldots \geq \nu_{n}$ with at most $q$ positives $\nu_{i}$ and $p$ negatives.

\noindent The parameter $\rho$ for $\widetilde{\U(n, \mathbb{C})}$ is given by the following formula:
\begin{equation}
\rho = \frac{1}{2} \sum\limits_{\alpha \in \Phi^{+}(\mathfrak{g}, \mathfrak{h})} \alpha = \sum\limits_{a=1}^n \cfrac{n-2a+1}{2} e_{a} = \sum\limits_{a=1}^n \rho_{a} e_{a}.
\end{equation}
Then,
\begin{equation}
\mu = \nu + \rho = \sum\limits_{k=1}^n (\nu + \rho)_{k}e_{k} = \sum\limits_{k=1}^n \left(\nu_{k} - k + \cfrac{q-p+n+1}{2}\right)e_{k},
\end{equation}
and
\begin{equation}
\mu'_{k} = \mu_{n+1-k} = \nu_{n+1-k} - (n+1-k) + \cfrac{q-p+n+1}{2} = \nu_{n+1-k} + k - \cfrac{q-p-n-1}{2}.
\end{equation}
\item One can easily show that, for every regular element $y' \in \mathfrak{h}'(m) \subseteq \mathfrak{h}'$, we have:
\begin{equation}
\Z'(m) = \{g' \in \G', \Ad(g')y' = y'\}, \qquad \mathscr{W}(\K', \mathfrak{h}')^{y'} = \mathscr{W}(\K', \mathfrak{h}')^{\mathfrak{h}'(m)}.
\end{equation}
\end{enumerate}
\end{rema}

\noindent Now we want to simplify the equation \eqref{Proposition1908}. We first get the following theorem.

\begin{theo}

For all regular element $y \in \mathfrak{h}'(m)^{\reg}$ (where $\mathfrak{h}'(m)$ was introduced in \eqref{Embedding1}) and $y' \in \mathfrak{h}'^{\reg}$, we get:
\begin{equation}
\pi_{\mathfrak{g}'/\mathfrak{z}'(m)}(y) \displaystyle\int_{\G'/\Z'(m)} e^{iB(y', g'y)}d(g'\Z'(m)) = (-i)^{\frac{1}{2}\dim(\mathfrak{g}'/\mathfrak{z}'(m))}(-1)^{n(\mathfrak{h'}(m))} \sum\limits_{\sigma \in \mathscr{W}(\K', \mathfrak{h}')/\mathscr{W}(\K', \mathfrak{h}')^{\mathfrak{h}'(m)}} \cfrac{e^{iB(y', \sigma y)}}{\pi_{\mathfrak{g}'/\mathfrak{z}'(m)}(\sigma^{-1}y')},
\label{Theoreme11}
\end{equation}
where $n(\mathfrak{h}'(m))$ is the number of non-compact positive roots which do not vanish on $\mathfrak{h}'(m)$.

\end{theo}

\begin{proof}

\noindent Let $\P_{y} = \{\alpha \in \Phi(\mathfrak{g}', \mathfrak{t}') \thinspace ; \thinspace i\alpha(y) > 0 \}$ and $n(y)$ the number of non-compact roots in $\P_{y}$. Using the equation \eqref{Vergne2}, we get
\begin{equation}
\prod\limits_{\alpha \in \P_{y}} i\alpha(y) \displaystyle\int_{\G'/\Z'(m)}e^{iB(y', g'y)} d(g'\Z'(m)) = (-1)^{n(y')} \sum\limits_{\sigma \in \mathscr{W}(\K', \mathfrak{h}')/\mathscr{W}(\K', \mathfrak{h}')^{\mathfrak{h}'(m)}} \cfrac{e^{iB(y', \sigma y)}}{\prod\limits_{\alpha \in \P_{y}} \alpha(\sigma^{-1}(y'))}.
\end{equation}

\noindent For all $\eta \in \mathscr{W}(\G, \mathfrak{h})$, we still denote by $\eta$ the element of $\mathscr{W}(\G', \mathfrak{h}')$ obtained by identification of $\mathfrak{h}$ with $\mathfrak{h}'(m) \subseteq \mathfrak{h}'$ via the embedding \eqref{Embedding1}. So, the action on the orthogonal complement of $\mathfrak{h}'(m)$ in $\mathfrak{h}'$ is trivial. We fix $\eta$ such that
\begin{equation*}
(\eta y)_{1} > \ldots > (\eta y)_{n}.
\end{equation*}
Then, we get
\begin{equation*}
\prod\limits_{\alpha \in \P_{\eta y}} \alpha(y) = \pi_{\mathfrak{g}'/ \mathfrak{z}'(m)}(y).
\end{equation*}
Now, 
\begin{equation*}
\P_{\eta y} = \eta \P_{y}, \qquad \prod\limits_{\alpha \in \P_{\eta y}} \alpha(y) = \sgn(\eta) \prod\limits_{\alpha \in \P_{y}} \alpha(y),
\end{equation*}
and
\begin{equation*}
\prod\limits_{\alpha \in \P_{y}} i\alpha(y) = \sgn(\eta) \prod\limits_{\alpha \in \P_{\eta y}} i\alpha(y)  =   i^{\frac{1}{2} \dim(\mathfrak{g}'/\mathfrak{z}'(m))} \pi_{\mathfrak{g}'/\mathfrak{z}'(m)}(y).
\end{equation*}
We have
\begin{equation*}
\pi_{\mathfrak{g}'/\mathfrak{z}'(m)}(y) \displaystyle\int_{\G'/\Z'(m)} e^{iB(y', g'y)}d(g'\Z'(m)) = (-i)^{\frac{1}{2}\dim(\mathfrak{g}'/\mathfrak{z}'(m))}(-1)^{n(y)} \sum\limits_{\sigma \in \mathscr{W}(\K', \mathfrak{h}')/\mathscr{W}(K', \mathfrak{h}')^{\mathfrak{h}'(m)}} \cfrac{e^{iB(y', \sigma y)}}{\pi_{\mathfrak{g}'/\mathfrak{z}'(m)}(\sigma^{-1}y')}.
\end{equation*}
As $y \in \mathfrak{h}'(m)^{\reg}$, we obtain
\begin{equation*}
\left\{\alpha \in \Phi(\mathfrak{g}', \mathfrak{h}') \thinspace ; \thinspace \alpha(y) \neq 0 \right\} = \left\{\alpha \in \Phi(\mathfrak{g}', \mathfrak{h}') \thinspace ; \thinspace  \alpha_{|_{\mathfrak{h}'(m)}} \neq 0 \right\}.
\end{equation*}
Similarly, if $\alpha$ is such that $\alpha(y) \neq 0$, then $-\alpha(y) \neq 0$. Finally,
\begin{eqnarray*}
n(y) & = & \sharp \left(P_{y} \cap \left(\Phi(\mathfrak{g}', \mathfrak{h}') \setminus \Phi(\mathfrak{k}', \mathfrak{h}')\right)\right) = \frac{1}{2} \sharp \left( \{\alpha \in \Phi(\mathfrak{g}', \mathfrak{h}') \thinspace ; \thinspace \alpha(y') \neq 0 \} \cap \left(\Phi(\mathfrak{g}', \mathfrak{h}') \setminus \Phi(\mathfrak{k}', \mathfrak{h}')\right) \right) \\
        & = & \sharp \left( \{\alpha \in \Phi^{+}(\mathfrak{g}', \mathfrak{h}') \thinspace ; \thinspace \alpha_{|_{\mathfrak{h}'(m)}} \neq 0 \} \cap \left(\Phi(\mathfrak{g}', \mathfrak{h}') \setminus \Phi(\mathfrak{k}', \mathfrak{h}')\right) \right) = n(\mathfrak{h}'(m)).
\end{eqnarray*}
So, $(-1)^{n(y)} = (-1)^{n(\mathfrak{h}'(m))}$, and the theorem is proved.

\end{proof}

\noindent It is clear that for all $\psi \in \mathscr{C}^{\infty}_{c}(\G'.\mathfrak{h}'^{\reg})$, 
\begin{equation}
\phi: W \ni w \mapsto \displaystyle\int_{\mathfrak{g}'} \chi_{x'}(w) \psi(x')dx' \in \mathbb{C}
\end{equation}
is a Schwartz function on $W$.
Using the Weyl integration formula, we get for all $y \in \tau(\mathfrak{h}_{1, m})$
\begin{eqnarray*}
& & \pi_{\mathfrak{g}'/\mathfrak{z}'(m)}(y) \displaystyle\int_{\G'/\Z'(m)} \displaystyle\int_{\mathfrak{g}'} e^{i B(y', g'y)} \psi(y')dy'd(g'\Z'(m)) \\
& = & \cfrac{\pi_{\mathfrak{g}'/\mathfrak{z}'(m)}(y)}{|\mathscr{W}(\K', \mathfrak{h}')|} \displaystyle\int_{\G'/\Z'(m)} \displaystyle\int_{\mathfrak{h}'} \displaystyle\int_{\tilde{\G}'/\tilde{\H}'} e^{i B(\Ad(\tilde{g}h', g'y)} \psi(\Ad(\tilde{g})h') |D(h')|^{2} d(\tilde{g}\tilde{\H}')dh'd(g'\Z'(m)) \\
& = & \cfrac{\pi_{\mathfrak{g}'/\mathfrak{z}'(m)}(y)}{|\mathscr{W}(\K', \mathfrak{h}')|} \displaystyle\int_{\mathfrak{h}'} \displaystyle\int_{\tilde{\G}'/\tilde{\H}'} \displaystyle\int_{\G'/\Z'(m)} e^{i B(h', \Ad(g^{-1}g')y)} \psi(\Ad(\tilde{g})h') |\pi_{\mathfrak{g}'/\mathfrak{h}'}(h')|^{2} d(g'\Z'(m)) d(\tilde{g}\tilde{\H}') dh' \\
& = & \cfrac{1}{|\mathscr{W}(\K', \mathfrak{h}')|} \displaystyle\int_{\mathfrak{h}'} \left( \pi_{\mathfrak{g}'/\mathfrak{z}'(m)}(y) \displaystyle\int_{\G'/\Z'(m)} e^{i B(y', g'y)} d(g'\Z'(m)) \right) \psi^{\tilde{\G}'/\tilde{\H}'}(y') |\pi_{\mathfrak{g}'/\mathfrak{h}'}(y')|^{2}dy' 
\end{eqnarray*}
because the measure on $\G'/\Z'(m)$ is $\G'$-invariant and the function $\psi^{\tilde{\G}'/\tilde{\H}'} : \mathfrak{h}'^{\reg} \to \mathbb{C}$ is given by:
\begin{equation}
\psi^{\tilde{\G'}/\tilde{\H'}}(y') = \displaystyle\int_{\tilde{\G'}/\tilde{\H'}} \psi(\Ad(\tilde{g}')(y')) d(\tilde{g'}\tilde{\H'}).
\label{psiG'H'}
\end{equation}

\noindent Then, the equation \eqref{Theoreme11} can be written as
\begin{eqnarray*}
& &  \pi_{\mathfrak{g}'/\mathfrak{z}'(m)}(y) \displaystyle\int_{\G'/\Z'(m)} \displaystyle\int_{\mathfrak{g}'} e^{i B(y', g'y)} \psi(y')dy'd(g'\Z'(m)) \\
& = & \cfrac{(-i)^{\frac{1}{2}\dim(\mathfrak{g}'/\mathfrak{z}'(m))}(-1)^{n(\mathfrak{h}'(m))}}{|\mathscr{W}(\K', \mathfrak{h}')|} \sum\limits_{\sigma \in \mathscr{W}(\K', \mathfrak{h}')/\mathscr{W}(\K', \mathfrak{h}')^{\mathfrak{h}'(m)}} \displaystyle\int_{\mathfrak{h}'} \cfrac{e^{iB(y', \sigma y)}}{\pi_{\mathfrak{g}'/\mathfrak{z}'(m)}(\sigma^{-1}y')} \psi^{\tilde{\G}'/\tilde{\H}'}(y') |\pi_{\mathfrak{g}'/\mathfrak{h}'}(y')|^{2}dy'.
\end{eqnarray*}
Clearly, $\psi^{\tilde{\G}'/\tilde{\H}'}$ and the form $B$ are both $\mathscr{W}(\K', \mathfrak{h}')$-invariant. Hence,
\begin{center}
\resizebox\textwidth!{$
\begin{aligned}
\sum\limits_{\sigma \in \mathscr{W}(\K', \mathfrak{h}')/\mathscr{W}(\K', \mathfrak{h}')^{\mathfrak{h}'(m)}} \displaystyle\int_{\mathfrak{h}'} \cfrac{e^{iB(y', \sigma y)}}{\pi_{\mathfrak{g}'/\mathfrak{z}'(m)}(\sigma^{-1}y')} \psi^{\G'/\H'}(y') |\pi_{\mathfrak{g}'/\mathfrak{h}'}(y')|^{2}dy' = \cfrac{|\mathscr{W}(\K', \mathfrak{h}')|}{|\mathscr{W}(\K', \mathfrak{h}')^{\mathfrak{h}'(m)}|} \displaystyle\int_{\mathfrak{h}'} \cfrac{e^{iB(y', y)}}{\pi_{\mathfrak{g}'/\mathfrak{z}'(m)}(y')} \psi^{\tilde{\G}'/\tilde{\H}'}(y') |\pi_{\mathfrak{g}'/\mathfrak{h}'}(y')|^{2}dy'.
\end{aligned}$}
\end{center}
Using the following equality
\begin{equation}
|\pi_{\mathfrak{g}'/\mathfrak{h}'}(y')|^{2} = \pi_{\mathfrak{g}'/\mathfrak{h}'}(y') \overline{\pi_{\mathfrak{g}'/\mathfrak{h}'}(y')} = \pi_{\mathfrak{g}'/\mathfrak{z}'(m)}(y')\pi_{\mathfrak{z}'(m)/\mathfrak{h}'}(y')\overline{\pi_{\mathfrak{g}'/\mathfrak{h}'}(y')},
\end{equation}
we finally get
\begin{equation}
f_{\phi}(y) = C_{m} (-1)^{n(\mathfrak{h}'(m))} \displaystyle\int_{\mathfrak{h}'}e^{i B(y', y)} \pi_{\mathfrak{z}'(m)/\mathfrak{h}'}(y') \overline{\pi_{\mathfrak{g}'/\mathfrak{h}'}(y')} \psi^{\tilde{\G}'/\tilde{\H}'}(y')dy',
\label{fphi}
\end{equation}
where the constant $C_{m}$ is given by 
\begin{equation*}
C_{m} = \cfrac{(-1)^{n(n-1)} i^{n(p+q-n)} (-i)^{(p+q)^{2} - (p+q-n)^{2} - n}}{(p-m)!(q-n+m)!}.
\end{equation*}

\begin{rema}

We denote by $\mathscr{A}_{\psi}$ the function on $\mathfrak{h}'^{\reg}$ defined by
\begin{equation}
\mathscr{A}_{\psi}(y) = \pi_{\mathfrak{g}'/\mathfrak{h}'}(y) \psi^{\tilde{\G}'/\tilde{\H}'}(y), \qquad y \in \mathfrak{h}'^{\reg}.
\end{equation}

\noindent Using the equation \eqref{fphi}, we get:
\begin{equation*}
 f_{\phi} = \mathscr{F}_{\mathfrak{h}'}\left(C_{m} (-1)^{n(\mathfrak{h}'(m))} \pi_{\mathfrak{z}'(m)/\mathfrak{h}'} \mathscr{A}_{\psi}\right)_{|_{\mathfrak{h}'(m)}},
\end{equation*}
where $(\mathscr{F}_{\mathfrak{h}'})_{|_{\mathfrak{h}'(m)}}$ is the restriction of the Fourier transform on $\mathfrak{h}'$ to $\mathfrak{h}'(m)$.
\label{RemaFourier}

\end{rema} 

\subsection{The character $\Theta_{\Pi'}$ on the compact diagonal Cartan subgroup}

Before giving the main theorem of this section, we prove a result concerning the right-hand side of \eqref{Proposition1908}. More particularly, we are interested in the support of the following distribution:
\begin{equation}
\displaystyle\int_{\mathfrak{h}}c_{-}(x)^{-\mu'}\ch(x)^{p+q-n-1} e^{i B(x, y)}dx
\end{equation}
We get the proposition.

\begin{prop}

Let $m_{\min}$ and $m_{\max}$ be the two positive integers defined by
\begin{equation}
m_{\min} = \begin{cases} \max\{j \geq 1 \thinspace ; \thinspace \mu'_{j} \leq -\frac{1}{2}(p+q-n+1)\} & \text{ if } \mu'_{1} \leq -\frac{1}{2}(p+q-n+1) \\ 0 & \text{ otherwise } \end{cases},
\end{equation}
\begin{equation}
m_{\max} = \begin{cases} \min\{j \geq 1 \thinspace ; \thinspace \mu'_{j} \geq \frac{1}{2}(p+q-n+1)\} - 1 & \text{ if } \mu'_{n} \geq \frac{1}{2}(p+q-n+1) \\ n & \text{ otherwise } \end{cases}.
\end{equation}
Then,
\begin{equation}
\supp \displaystyle\int_{\mathfrak{h}} c_{-}(x)^{-\mu'} \ch(x)^{p+q-n-1} e^{iB(x, y)} dx = \bigcap_{m^{'} = m_{\min}}^{m_{\max}} \tau(\mathfrak{h}_{1, m^{'}}).
\label{Support distribution}
\end{equation}

\end{prop}

\begin{proof}

According to \cite[Lemma~33,~page~42]{TOM4}, we have:
\begin{equation}
c_{-}(x)^{-\mu'}\ch(x)^{p+q-n-1} = \prod\limits_{k=1}^n (1+ix_{k})^{-a_{k}}(1-ix_{k})^{-b_{k}}
\end{equation}
where $a_{k} = \mu'_{k} - \frac{p+q-n-1}{2}$ and $b_{k} = -\mu'_{k} - \frac{p+q-n-1}{2}$. Hence,
\begin{equation}
\displaystyle\int_{\mathfrak{h}} c_{-}(x)^{-\mu'}\ch(x)^{p+q-n-1}e^{iB(x, y)} dx = \prod\limits_{k=1}^n \displaystyle\int_{\mathbb{R}} (1+ix_{k})^{-a_{k}}(1-ix_{k})^{-b_{k}}e^{-i\pi x_{k}y_{k}} dx_{k}.
\end{equation}
Using \cite[Appendix~C]{TOM4}, we get:
\begin{equation}
\displaystyle\int_{\mathbb{R}} (1+ix_{k})^{-a_{k}}(1-ix_{k})^{-b_{k}}e^{-i\pi x_{k}y_{k}} = P_{a_{k}, b_{k}}(\pi y_{k})e^{-|\pi y_{k}|} + Q_{a_{k}, b_{k}}\left(\cfrac{\partial}{\partial y_{k}}\right) \delta_{0}(\pi y_{k}).
\end{equation}
where $P_{a_{k}, b_{k}}$ and $Q_{a_{k}, b_{k}}$ are the polynomials defined in \cite[Appendix~C,~equations~C.1-5]{TOM4}. In particular, using \cite[Appendix~C]{TOM4}, one can easily verify that the support of the polynomials $P_{a_{k}, b_{k}}$ are
\begin{equation}
\supp P_{a_{k}, b_{k}} = \begin{cases} ]-\infty, 0] & \text{ if } a_{k} \geq 1 \text{ and } b_{k} \leq 0  \\ [0, +\infty[ & \text{ if }a_{k} \leq 0 \text{ and } b_{k} \geq 1 \\ \{\emptyset\} & \text{ if } b_{k} \leq 0 \text{ and } a_{k} \leq 0 \end{cases}
\label{Support1}
\end{equation} 
We have $a_{k} = -p+k+\nu_{n+1-k}$ and $b_{k} = -\nu_{n+1-k} - k - q + n +1$. In particular, $a_{k} \geq 1$ (resp. $b_{k} \geq 1$) implies $b_{k} \leq 0$ (resp. $a_{k} \leq 0$). Then, \eqref{Support1} can be rewritten as follow:
\begin{equation}
\supp P_{a_{k}, b_{k}} = \begin{cases} ]-\infty, 0] & \text{ if } a_{k} \geq 1 \\ [0, +\infty[ & \text{ if } b_{k} \geq 1 \\ \{\emptyset\} & \text{ otherwise } \end{cases} = \begin{cases} ]-\infty, 0] & \text{ if } \mu'_{k} \geq \frac{p+q-n+1}{2} \\ [0, +\infty[ & \text{ if } \mu'_{k} \leq -\frac{p+q-n+1}{2} \\ \{\emptyset\} & \text{ otherwise } \end{cases}
\label{Support2}
\end{equation} 
The sequence $(\mu'_{i})_{i}$ is increasing. So,
\begin{equation}
\supp \displaystyle\int_{\mathbb{R}} (1+ix_{k})^{-a_{k}}(1-ix_{k})^{-b_{k}}e^{-i\pi x_{k}y_{k}} = \begin{cases} [0, +\infty[ & \text{ if } k \leq m_{\min} \\ \{0\} & \text{ if } k \in [|m_{\min} + 1, m_{\max}|] \\ ]-\infty, 0] & \text{ if } k \geq m_{\max} + 1 \end{cases} 
\end{equation}
Therefore, using \cite[Lemma~9]{TOM4}, we conclude:
\begin{equation}
\bigcap\limits_{m' = m_{\min}}^{m_{\max}} \tau(\mathfrak{h}_{1, m'}) = \left\{\sum\limits_{k = 1}^n -ih_{k}E_{k, k}, h_{1} \geq \ldots \geq h_{m_{\min}} > 0 = h_{m_{\min}+1} = \ldots = h_{m_{\max}} > h_{m_{\max}+1} \geq \ldots \geq h_{n}\right\}.
\end{equation}
Thus, the equality \eqref{Support distribution} is satisfied.

\end{proof}

\noindent From now on, we fix $m_{\min} \leq m \leq m_{\max}$ and we consider the decomposition of $\mathfrak{h}'$ as:
\begin{equation}
\mathfrak{h}' = \mathfrak{h}'(m) \oplus \mathfrak{h}'(m)^{\perp}
\end{equation}
where $\mathfrak{h}'(m)^{\perp}$ is the orthogonal complement of $\mathfrak{h}'(m)$ with respect to the bilinear form $B$. Let $\pr_{m}$ be the orthogonal projection of $\mathfrak{h}'$ on $\mathfrak{h}'(m)$ that is:
\begin{equation}
\pr_{m}(y' + y'') = y' \qquad (y' \in \mathfrak{h}'(m), y'' \in \mathfrak{h}'(m)^{\perp}).
\label{label projection}
\end{equation}
Moreover, let $\rho_{\mathfrak{z}'(m)}$ be the linear form of $\mathfrak{h}'^{*}$ defined by
\begin{equation}
\rho_{\mathfrak{z}'(m)} = \frac{1}{2} \sum\limits_{\underset{\alpha_{|_{\mathfrak{h}'(m)}} = 0}{\alpha \in \Phi^{+}(\mathfrak{g}', \mathfrak{h}')}} \alpha.
\end{equation}

\noindent Finally, let us consider the subgroup $\mathscr{W}(\G, \mathfrak{h}, m)$ of $\mathscr{W}(\G, \mathfrak{h})$ defined by:
\begin{equation}
\mathscr{W}(\G, \mathfrak{h}, m) = \{\eta \in \mathscr{W}(\G, \mathfrak{h}) \thinspace ; \thinspace \bigcap_{m' = m_{\min}}^{m_{\max}} \tau(\mathfrak{h}_{1, m'}) \subseteq \eta\tau(\mathfrak{h}_{1, m})\}
\end{equation}

\begin{theo}
\label{TheoremU(p,q)}
The character $\Theta_{\Pi'}$ of $\Pi' \in \mathscr{R}(\tilde{\G}', \omega^{\infty})$ is given, up to a constant, by the following formula:
\begin{equation}
\Theta_{\Pi'}(h) = \sum\limits_{\eta \in \mathscr{W}(\G, \mathfrak{h}, m)/ \mathscr{W}(\G, \mathfrak{h})_{m}} \sum\limits_{\sigma \in \mathscr{W}(\K', \mathfrak{h}')} \cfrac{\sgn(\eta) \pr_{m}(\sigma h)^{-\eta^{-1}\mu'}}{\prod\limits_{\underset{\beta_{|_{\mathfrak{h}'(m)} \neq 0}}{\beta \in \Phi^{+}(\mathfrak{g}', \mathfrak{h}')}}((\sigma h)^{\frac{\beta}{2}} - (\sigma h)^{-\frac{\beta}{2}})}.
\label{FormuleFinaleU(p,q)}
\end{equation}

\end{theo}

\begin{nota}

To simplify the notation, we note $d = p=q-n-1$.

\end{nota}

\begin{proof}

By \eqref{Proposition1908} and \eqref{Support distribution} , for every $m \in [|m_{\min}, m_{\max}|]$, we get
\begin{equation}
T(\overline{\Theta_{\Pi}})(\phi) = \displaystyle\int_{\mathfrak{h}} c_{-}(x)^{-\mu'}\ch(x)^{d} \displaystyle\int_{\mathscr{W}(\G, \mathfrak{h})\tau(\mathfrak{h}_{1, m})} e^{i B(x, y)} f_{\phi}(y)dy dx 
\end{equation}
Indeed, for all $m \in [|m_{\min}, m_{\max}|]$:
\begin{equation*}
\bigcap\limits_{m' = m_{\min}}^{m_{\max}} \tau(\mathfrak{h}_{1, m'}) \subseteq \tau(\mathfrak{h}_{1, m}) \subseteq \mathscr{W}(\G, \mathfrak{h}) \tau(\mathfrak{h}_{1, m})
\end{equation*}
Then,
\begin{eqnarray}
& & \displaystyle\int_{\mathfrak{h}} c_{-}(x)^{-\mu'}\ch(x)^{d} \displaystyle\int_{\mathscr{W}(\G, \mathfrak{h})\tau(\mathfrak{h}_{1, m})} e^{i B(x, y)} f_{\phi}(y)dy dx \nonumber \\
& = & \sum\limits_{\eta_{0} \in \mathscr{W}(\G, \mathfrak{h},m)/\mathscr{W}(\G, \mathfrak{h})_{m}} \displaystyle\int_{\mathfrak{h}} c_{-}(x)^{-\mu'}\ch(x)^{d} \displaystyle\int_{\tau(\mathfrak{h}_{1, m})} e^{i B(x, \eta_{0}(y))} f_{\phi}(\eta_{0}(y)) dydx \nonumber \\
& = & \sum\limits_{\eta_{0} \in \mathscr{W}(\G, \mathfrak{h}, m)/\mathscr{W}(\G, \mathfrak{h})_{m}} \sgn(\eta_{0})\displaystyle\int_{\mathfrak{h}} c_{-}(x)^{-\mu'}\ch(x)^{d} \displaystyle\int_{\tau(\mathfrak{h}_{1, m})} e^{i B(\eta^{-1}_{0}(x), y)} f_{\phi}(y) dydx \nonumber\\
& = & \sum\limits_{\eta_{0} \in \mathscr{W}(\G, \mathfrak{h}, m)/\mathscr{W}(\G, \mathfrak{h})_{m}} \sgn(\eta_{0})\displaystyle\int_{\mathfrak{h}} c_{-}(\eta_{0}(x))^{-\mu'}\ch(\eta_{0}(x))^{d} \displaystyle\int_{\tau(\mathfrak{h}_{1, m})} e^{i B(x, y)} f_{\phi}(y) dydx \nonumber \\
& = & \sum\limits_{\eta_{0} \in \mathscr{W}(\G, \mathfrak{h}, m)/\mathscr{W}(\G, \mathfrak{h})_{m}} \sgn(\eta_{0})\displaystyle\int_{\mathfrak{h}} c_{-}(x)^{-\eta^{-1}_{0}\mu'}\ch(x)^{d} \displaystyle\int_{\tau(\mathfrak{h}_{1, m})} e^{i B(x, y)} f_{\phi}(y) dydx \label{int1}
\end{eqnarray}
Again, playing with the support of $f_{\phi}$, we get:
\begin{equation*}
\displaystyle\int_{\tau(\mathfrak{h}_{1, m})} e^{i B(x, y)} f_{\phi}(y) dy = \displaystyle\int_{\mathfrak{h}} e^{i B(x, y)} f_{\phi}(y) dy
\end{equation*}
We identify $\mathfrak{h}$ with its image $\mathfrak{h}'(m)$ in $\mathfrak{h}'$. Hence, \eqref{int1} can be written as:
\begin{equation}
\sum\limits_{\eta_{0} \in \mathscr{W}(\G, \mathfrak{h}, m)/\mathscr{W}(\G, \mathfrak{h})_{m}} \sgn(\eta_{0})\displaystyle\int_{\mathfrak{h}'(m)} c_{-}(x)^{-\eta^{-1}_{0}\mu'}\ch(x)^{d} \displaystyle\int_{\mathfrak{h}'(m)} e^{i B(x, y)} f_{\phi}(y) dydx
\label{int2}
\end{equation}
By Remark \ref{RemaFourier}, \eqref{int2} can be written as:
\begin{eqnarray}
& & \sum\limits_{\eta_{0} \in \mathscr{W}(\G, \mathfrak{h}, m)/\mathscr{W}(\G, \mathfrak{h})_{m}} \sgn(\eta_{0})\displaystyle\int_{\mathfrak{h}} c_{-}(x)^{-\eta^{-1}_{0}\mu'}\ch(x)^{d} \displaystyle\int_{\tau(\mathfrak{h}_{1, m})} e^{i B(x, y)} f_{\phi}(y) dydx \nonumber \\
& = & \sum\limits_{\eta_{0} \in \mathscr{W}(\G, \mathfrak{h}, m)/\mathscr{W}(\G, \mathfrak{h})_{m}} \sgn(\eta_{0})\displaystyle\int_{\mathfrak{h}'(m)} c_{-}(x)^{-\eta^{-1}_{0}\mu'}\ch(x)^{d} \displaystyle\int_{\mathfrak{h}'(m)} e^{i B(x, y)} \left. \mathscr{F}_{\mathfrak{h}'}(\pi_{\mathfrak{z}'(m)/\mathfrak{h}'} \mathscr{A}_{\psi})\right|_{\mathfrak{h}'(m)}(y) dydx \nonumber \\
& = & \sum\limits_{\eta_{0} \in \mathscr{W}(\G, \mathfrak{h}, m)/\mathscr{W}(\G, \mathfrak{h})_{m}} \sgn(\eta_{0})\displaystyle\int_{\mathfrak{h}'(m)} c_{-}(x)^{-\eta^{-1}_{0}\mu'}\ch(x)^{d} \mathscr{F}^{-1}_{\mathfrak{h}'} \left. \mathscr{F}_{\mathfrak{h}'}(\pi_{\mathfrak{z}'(m)/\mathfrak{h}'} \mathscr{A}_{\psi})\right|_{\mathfrak{h}'(m)}(y) dydx \nonumber \\
& = & \sum\limits_{\eta_{0} \in \mathscr{W}(\G, \mathfrak{h}, m)/\mathscr{W}(\G, \mathfrak{h})_{m}} \sgn(\eta_{0})\displaystyle\int_{\mathfrak{h}'(m)} c_{-}(x')^{-\eta^{-1}_{0}\mu'}\ch(x')^{d} \displaystyle\int_{\mathfrak{h}'(m)^{\perp}} (\pi_{\mathfrak{z}'(m)/\mathfrak{h}'} \mathscr{A}_{\psi})(x' + y') dy'dx' \nonumber \\
& = & \sum\limits_{\eta_{0} \in \mathscr{W}(\G, \mathfrak{h}, m)/\mathscr{W}(\G, \mathfrak{h})_{m}} \sgn(\eta_{0})\displaystyle\int_{\mathfrak{h}'(m)} c_{-}(x')^{-\eta^{-1}_{0}\mu'}\ch(x')^{d} \displaystyle\int_{\mathfrak{h}'(m)^{\perp}} \pi_{\mathfrak{z}'(m)/\mathfrak{h}'}(y') \mathscr{A}_{\psi}(x' + y') dy'dx'. \label{int3}
\end{eqnarray}
Using \eqref{FameuseFonctionphi}, we now choose a particular function $\psi \in \mathscr{C}^{\infty}_{c}(\G'.\mathfrak{h}'^{\reg})$. We denote by $\G'^{c} \subseteq \G'$ the domain of the Cayley transform in $\G'$. Fix $\Psi \in \mathscr{C}^{\infty}_{c}(\widetilde{\G}')$ such that $\supp(\Psi) \in \widetilde{\G'^{c}} \subseteq \widetilde{\G}'$ and let $j_{\mathfrak{g}'}$ be the Jacobian of $c: \mathfrak{g}'^{c} \to \G'^{c}$. We recall that $\tilde{c}_{-}$ is defined by $\tilde{c}_{-}(x') = \tilde{c}(x') \tilde{c}(0)^{-1}$ and let $\psi$ be the function given by
\begin{equation}
\psi(x') = \Theta(\tilde{c}(x'))j_{\mathfrak{g}'}(x')\Psi(\tilde{c}_{-}(x')) \qquad (x' \in \mathfrak{g}'^{c}).
\label{choixpsi}
\end{equation}
With such a function $\psi$, the integral \eqref{int3} becomes:
\begin{equation}
\sum\limits_{\eta_{0} \in \mathscr{W}(\G, \mathfrak{h}, m)/\mathscr{W}(\G, \mathfrak{h})_{m}} \sgn(\eta_{0})\displaystyle\int_{\mathfrak{h}'} c_{-}(\pr_{m}(x))^{-\eta^{-1}_{0}\mu'}\ch(\pr_{m}(x))^{d}\pi_{\mathfrak{z}'(m)/\mathfrak{h}'}(x) \pi_{\mathfrak{g}'/\mathfrak{h}'}(x) \Theta(\tilde{c}(x)) j_{\mathfrak{g}'}(x)\Psi^{\tilde{\G}'/\tilde{\H}'}(\tilde{c}_{-}(x)) dx,
\label{int4}
\end{equation}
where $\pr_{m}$ is the projection given in the equation \eqref{label projection}.

\noindent To get the character $\Theta_{\Pi'}$ of $\Pi'$, we would like to write the integral defined in \eqref{int4} as an integral over $\widetilde{\H}'$. For all $h' = \sum\limits_{k=1}^{p+q} ih_{k}\E_{k, k} \in \mathfrak{h}'$, we have:
\begin{equation}
\det_{\mathbb{R}}(h'-1) = \prod\limits_{k=1}^{p+q} (1+h^{2}_{k}), \qquad \Theta(\tilde{c}(h')) = \ch(h')^{\dim_{\mathbb{C}}(V_{\overline{0}})} = \prod\limits_{k=1}^{p+q} (1+h^{2}_{k})^{\frac{n}{2}}.
\end{equation}

\noindent Moreover, for every $\mu = \sum\limits_{k=1}^{p+q} \mu_{k}e_{k}$ an analytic form on $\H'^{\sharp}_{0}$, we have:
\begin{equation}
\xi^{\sharp}_{\mu}(\exp^{\sharp}(h')) = e^{\mu(h')} = e^{\sum\limits_{k=1}^n \mu_{k}e_{k}(h')} = \prod\limits_{k=1}^n (e^{ih_{k}})^{\mu_{k}}.
\end{equation}
The Cayley transform $c(h')$ is given by
\begin{equation}
c\begin{pmatrix} ih_{1} & & \\ & \ddots & \\ & & ih_{p+q} \end{pmatrix} = \begin{pmatrix} c(ih_{1}) & & \\ & \ddots & \\ & & c(ih_{p+q}) \end{pmatrix}
\end{equation}
where $c(ih_{k}) = \frac{ih_{k}+1}{ih_{k}-1} =  \frac{h^{2}_{k} - 1}{h^{2}_{k}+1} + i \frac{-2}{h^{2}_{k}+1} = e^{im_{k}}$, with $m_{k} = \arccos\left(\frac{h^{2}_{k}-1}{h^{2}_{k}+1}\right) = \arcsin\left(\frac{-2}{h^{2}_{k}+1}\right)$. We denote by $M = \diag(im_{1}, \ldots, im_{p+q})$. Then,
\begin{equation}
\xi^{\sharp}_{\mu}(c^{\sharp}_{-}(h')) = \xi^{\sharp}_{\mu}(\exp^{\sharp}(M)) = \prod\limits_{k = 1}^{p+q} (e^{im_{k}})^{\mu_{k}} = \prod\limits_{k=1}^{p+q} c_{-}(ih_{k})^{\mu_{k}}.
\end{equation}

\noindent Then, we have:
\begin{eqnarray*}
& & \prod\limits_{\alpha \in \Phi^{+}(\mathfrak{g}', \mathfrak{h}')} \left(\xi^{\sharp}_{\frac{\alpha}{2}}(c^{\sharp}_{-}(h')) - \xi^{\sharp}_{-\frac{\alpha}{2}}(c^{\sharp}_{-}(h'))\right) = \prod\limits_{1 \leq i < j \leq p+q} \left(\xi^{\sharp}_{\frac{e_{i} - e_{j}}{2}}(c^{\sharp}_{-}(h')) - \xi^{\sharp}_{\frac{-e_{i} + e_{j}}{2}}(c^{\sharp}_{-}(h'))\right) \\
& = & \prod\limits_{1 \leq i < j \leq p+q} (c_{-}(ih_{i})^{\frac{1}{2}}c_{-}(ih_{j})^{-\frac{1}{2}} - c_{-}(ih_{i})^{-\frac{1}{2}}c_{-}(ih_{j})^{\frac{1}{2}}) = \prod\limits_{1 \leq i < j \leq p+q} c_{-}(ih_{i})^{-\frac{1}{2}}c_{-}(ih_{j})^{-\frac{1}{2}}(c_{-}(ih_{i}) - c_{-}(ih_{j})) \\
& = & \prod\limits_{1 \leq i < j \leq p+q} c_{-}(ih_{i})^{-\frac{1}{2}}c_{-}(ih_{j})^{-\frac{1}{2}}\left(\cfrac{ih_{i}+1}{ih_{i}-1} - \cfrac{ih_{j}+1}{ih_{j}-1}\right) = \prod\limits_{1 \leq i < j \leq p+q} c_{-}(ih_{i})^{-\frac{1}{2}}c_{-}(ih_{j})^{-\frac{1}{2}} \cfrac{2i(h_{j} - h_{i})}{(ih_{i}-1)(ih_{j}-1)} \\
& = & \prod\limits_{1 \leq i < j \leq p+q} \cfrac{(ih_{i}-1)^{\frac{1}{2}}(ih_{j}-1)^{\frac{1}{2}}}{(ih_{i}+1)^{\frac{1}{2}}(ih_{j}+1)^{\frac{1}{2}}}\cfrac{ih_{i}+1}{ih_{i}-1} - \cfrac{ih_{j}+1}{ih_{j}-1} = \prod\limits_{1 \leq i < j \leq p+q} \cfrac{2i(h_{j} - h_{i})}{(1+h^{2}_{i})^{\frac{1}{2}} (1+h^{2}_{j})^{\frac{1}{2}}} = \cfrac{\prod\limits_{1 \leq i < j \leq p+q} 2i(h_{j} - h_{i})}{\prod\limits_{k=1}^{p+q} (1+h^{2}_{k})^{\frac{p+q-1}{2}}} \\
& = & C \cfrac{\pi_{\mathfrak{g}'/\mathfrak{h}'}(h')}{\prod\limits_{k=1}^{p+q} (1+h^{2}_{k})^{\frac{p+q-1}{2}}}
\end{eqnarray*}
and the quantity $\prod\limits_{\underset{\alpha_{|_{\mathfrak{h}'(m)} \neq 0}}{\alpha \in \Phi^{+}(\mathfrak{g}', \mathfrak{h}')}} \left(\xi^{\sharp}_{\frac{\alpha}{2}}(c^{\sharp}_{-}(h')) - \xi^{\sharp}_{-\frac{\alpha}{2}}(c^{\sharp}_{-}(h'))\right)$ is given by:
{\small
\begin{eqnarray*}
& & \cfrac{\prod\limits_{1 \leq i < j \leq m} 2i(h_{j} - h_{i}) \prod\limits_{p+q-n+m+1 \leq i < j \leq p+q} 2i(h_{j} - h_{i})\prod\limits_{i=1}^m \prod\limits_{j=m+1}^{p+q}2i(h_{j} - h_{i})\prod\limits_{i=m+1}^{p+q-n+m} \prod\limits_{j=p+q-n+m+1}^{p+q} 2i(h_{j} - h_{i})}{\prod\limits_{k=1}^{m} (1+h^{2}_{k})^{\frac{m-1}{2}}\prod\limits_{k=p+q-n+m+1}^{p+q} (1+h^{2}_{k})^{\frac{n-m}{2}}\prod\limits_{i=1}^m \prod\limits_{j=m+1}^{p+q} (1+h^{2}_{i})^{\frac{1}{2}}(1+h^{2}_{j})^{\frac{1}{2}}\prod\limits_{i=m+1}^{p+q-n+m} \prod\limits_{j=p+q-n+m+1}^{p+q} (1+h^{2}_{i})^{\frac{1}{2}}(1+h^{2}_{j})^{\frac{1}{2}}} \\
& = & C' \cfrac{\pi_{\mathfrak{g}'/\mathfrak{z}'(m)}(h')}{\prod\limits_{k=1}^{m} (1+h^{2}_{k})^{\frac{m-1}{2}} \prod\limits_{k=p+q-n+m+1}^{p+q} (1+h^{2}_{k})^{\frac{n-m}{2}}\prod\limits_{i=1}^m \prod\limits_{j=m+1}^{p+q} (1+h^{2}_{i})^{\frac{1}{2}}(1+h^{2}_{j})^{\frac{1}{2}} \prod\limits_{i=m+1}^{p+q-n+m} \prod\limits_{j=p+q-n+m+1}^{p+q}(1+h^{2}_{i})^{\frac{1}{2}}(1+h^{2}_{j})^{\frac{1}{2}}} \\
& = & C'\cfrac{\pi_{\mathfrak{g}'/\mathfrak{z}'(m)}(h')}{\prod\limits_{k=1}^{m} (1+h^{2}_{k})^{\frac{m-1}{2}} \prod\limits_{k=p+q-n+m+1}^{p+q} (1+h^{2}_{k})^{\frac{n-m}{2}} \prod\limits_{i=1}^m (1+h^{2}_{i})^{\frac{p+q-m}{2}} \prod\limits_{i = m+1}^{p+q-n+m} (1 + h^{2}_{i})^{\frac{n-m-1}{2}} \prod\limits_{j=p+q-n+m+1}^{p+q} (1+h^{2}_{j})^{\frac{p+q-n}{2}}} \\
& = & C' \cfrac{\pi_{\mathfrak{g}'/\mathfrak{z}'(m)}(h')}{\prod\limits_{k=1}^{m} (1+h^{2}_{k})^{\frac{p+q-1}{2}} \prod\limits_{k = p+q-n+m+1}^{p+q}(1+h^{2}_{k})^{\frac{p+q-1}{2}} \prod\limits_{k=m+1}^{p+q-n+m} (1+h^{2}_{k})^{\frac{n}{2}}}
\end{eqnarray*}}
\noindent Finally, we get:
\begin{equation}
\cfrac{\left|\prod\limits_{\alpha \in \Phi^{+}(\mathfrak{g}', \mathfrak{h}')} \left(\xi^{\sharp}_{\frac{\alpha}{2}}(c^{\sharp}_{-}(h')) - \xi^{\sharp}_{-\frac{\alpha}{2}}(c^{\sharp}_{-}(h'))\right)\right|^{2}}{\prod\limits_{\underset{\alpha_{|_{\mathfrak{h}'(m)} \neq 0}}{\alpha \in \Phi^{+}(\mathfrak{g}', \mathfrak{h}')}} \left(\xi^{\sharp}_{\frac{\alpha}{2}}(c^{\sharp}_{-}(h')) - \xi^{\sharp}_{-\frac{\alpha}{2}}(c^{\sharp}_{-}(h'))\right)} = C'' \cfrac{\pi_{\mathfrak{g}'/\mathfrak{h}'}(h') \pi_{\mathfrak{z}'(m)/\mathfrak{h}'}(h')}{\prod\limits_{k=1}^{m} (1+h^{2}_{k})^{\frac{p+q-1}{2}} \prod\limits_{k = p+q-n+m+1}^{p+q}(1+h^{2}_{k})^{\frac{p+q-1}{2}} \prod\limits_{k = m+1}^{p+q-n+m} (1+h^{2}_{k})^{p+q-1-\frac{n}{2}}}
\end{equation}
Using the fact that $j_{\mathfrak{g}'}(h') = \prod\limits_{k=1}^{p+q} (1+h^{2}_{k})^{-(p+q)}$ and $j_{\mathfrak{h}'}(h') = \prod\limits_{k=1}^{p+q} (1+h^{2}_{k})^{-1}$, we obtain, up to a constant, the following equality:
\begin{equation}
\ch(\pr_{m}(h'))^{p+q-n-1}\pi_{\mathfrak{z}'(m)/\mathfrak{h}'}(h') \pi_{\mathfrak{g}'/\mathfrak{h}'}(h') \Theta(\tilde{c}(h')) = \cfrac{\left|\prod\limits_{\alpha \in \Phi^{+}(\mathfrak{g}', \mathfrak{h}')} \left(\xi_{\frac{\alpha}{2}}(c^{\sharp}_{-}(h')) - \xi_{-\frac{\alpha}{2}}(c^{\sharp}_{-}(h'))\right)\right|^{2}}{\prod\limits_{\underset{\alpha_{|_{\mathfrak{h}'(m)} \neq 0}}{\alpha \in \Phi^{+}(\mathfrak{g}', \mathfrak{h}')}} \left(\xi_{\frac{\alpha}{2}}(c^{\sharp}_{-}(h')) - \xi_{-\frac{\alpha}{2}}(c^{\sharp}_{-}(h'))\right)} (j^{-1}_{\mathfrak{g}'}j_{\mathfrak{h}'})(h').
\end{equation}

\noindent Then, the equation \eqref{int4} becomes:{\small
\begin{eqnarray*}
& & T(\overline{\Theta_{\Pi}})(\phi) \\
&=&  C_{m} \sum\limits_{\eta_{0} \in \mathscr{W}(\G, \mathfrak{h}, m)/\mathscr{W}(\G, \mathfrak{h})_{m}} \sgn(\eta_{0})\displaystyle\int_{\mathfrak{h}'} c_{-}(\pr_{m}(x))^{-\eta^{-1}_{0}\mu'}\ch(\pr_{m}(x))^{d}\pi_{\mathfrak{z}'(m)/\mathfrak{h}'}(x) \pi_{\mathfrak{g}'/\mathfrak{h}'}(x) \Theta(\tilde{c}(x)) j_{\mathfrak{g}'}(x)\Psi^{\tilde{\G}'/\tilde{\H}'}(\tilde{c}_{-}(x)) dx \\
& = & C'(m) \sum\limits_{\eta_{0} \in \mathscr{W}(\G, \mathfrak{h}, m)/\mathscr{W}(\G, \mathfrak{h})_{m}} \sgn(\eta_{0}) \displaystyle\int_{\mathfrak{h}'} c_{-}(\pr_{m}(x))^{-\eta^{-1}_{0}\mu'} \cfrac{\left| \prod\limits_{\alpha \in \Phi^{+}(\mathfrak{g}', \mathfrak{h}')} (c_{-}(x)^{\frac{\alpha}{2}} - c_{-}(x)^{-\frac{\alpha}{2}})\right|^{2}}{\prod\limits_{\underset{\alpha_{|_{\mathfrak{h}'(m)}} \neq 0}{\alpha \in \Phi^{+}(\mathfrak{g}', \mathfrak{h}')}} (c_{-}(x)^{\frac{\alpha}{2}} - c_{-}(x)^{-\frac{\alpha}{2}})} j_{\mathfrak{h}}(x) \Psi^{\tilde{\G}'/\tilde{\H}'}(\tilde{c}_{-}(x)) dx
\end{eqnarray*}}

\noindent Finally, 
\begin{equation}
T(\overline{\Theta_{\Pi}})(\phi) = \displaystyle\int_{\widetilde{\H}'} C'(m) \sum\limits_{\eta_{0} \in \mathscr{W}(\G, \mathfrak{h}, m)/\mathscr{W}(\G, \mathfrak{h})_{m}} \cfrac{\sgn(\eta_{0}) \pr_{m}(h)^{-\eta^{-1}_{0}\mu'}}{\prod\limits_{\underset{\alpha_{|_{\mathfrak{h}'(m)}} \neq 0}{\alpha \in \Phi^{+}(\mathfrak{g}', \mathfrak{h}')}} (h^{\frac{\alpha}{2}} - h^{-\frac{\alpha}{2}})} \left|\prod\limits_{\alpha \in \Phi^{+}(\mathfrak{g}', \mathfrak{h}')} (h^{\frac{\alpha}{2}} - h^{-\frac{\alpha}{2}})\right|^{2} \Psi^{\tilde{\G}'/\tilde{\H}'}(h) dh
\label{intH'}
\end{equation}
Using the invariance of $\left|\prod\limits_{\alpha \in \Phi^{+}(\mathfrak{g}', \mathfrak{h}')} (h^{\frac{\alpha}{2}} - h^{-\frac{\alpha}{2}})\right|^{2} \Psi^{\tilde{\G}'/\tilde{\H}'}(h)$ under the action of $\mathscr{W}(\K', \mathfrak{h}')$, we get:
\begin{equation*}
\left|\prod\limits_{\alpha \in \Phi^{+}(\mathfrak{g}', \mathfrak{h}')}(h^{\frac{\alpha}{2}} - h^{-\frac{\alpha}{2}})\right|^{2} \Psi^{\tilde{\G}'/\tilde{\H}'}(h) = \cfrac{1}{|\mathscr{W}(\K', \mathfrak{h}')|} \sum\limits_{\sigma \in \mathscr{W}(\K', \mathfrak{h}')} \left|\prod\limits_{\alpha \in \Phi^{+}(\mathfrak{g}', \mathfrak{h}')}((\sigma h)^{\frac{\alpha}{2}} - (\sigma h)^{-\frac{\alpha}{2}})\right|^{2} \Psi^{\tilde{\G}'/\tilde{\H}'}(\sigma h)    
\end{equation*}   
and {\small
\begin{eqnarray*}
& &T(\overline{\Theta_{\Pi}})(\phi) \\
& = &  \cfrac{C'(m)}{|\mathscr{W}(\K', \mathfrak{h}')|}\displaystyle\int_{\widetilde{\H}'} \sum\limits_{\eta_{0} \in \mathscr{W}(\G, \mathfrak{h}, m)/\mathscr{W}(\G, \mathfrak{h})_{m}} \sum\limits_{\sigma \in \mathscr{W}(\K', \mathfrak{h}')} \cfrac{\sgn(\eta_{0}) \pr_{m}(h)^{-\eta^{-1}_{0}\mu'}}{\prod\limits_{\underset{\alpha_{|_{\mathfrak{h}'(m)}} \neq 0}{\alpha \in \Phi^{+}(\mathfrak{g}', \mathfrak{h}')}} (h^{\frac{\alpha}{2}} - h^{-\frac{\alpha}{2}})} \left|\prod\limits_{\alpha \in \Phi^{+}(\mathfrak{g}', \mathfrak{h}')} ((\sigma h)^{\frac{\alpha}{2}} - (\sigma h)^{-\frac{\alpha}{2}})\right|^{2} \Psi^{\tilde{\G}'/\tilde{\H}'}(\sigma h) dh \\
                  & = & \cfrac{C'(m)}{|\mathscr{W}(\K', \mathfrak{h}')|} \displaystyle\int_{\widetilde{\H}'} \sum\limits_{\eta_{0} \in \mathscr{W}(\G, \mathfrak{h}, m)/\mathscr{W}(\G, \mathfrak{h})_{m}} \sum\limits_{\sigma \in \mathscr{W}(\K', \mathfrak{h}')} \cfrac{\sgn(\eta_{0}) \pr_{m}(\sigma^{-1}h)^{-\eta^{-1}_{0}\mu'}}{\prod\limits_{\underset{\alpha_{|_{\mathfrak{h}'(m)}} \neq 0}{\alpha \in \Phi^{+}(\mathfrak{g}', \mathfrak{h}')}} ((\sigma^{-1}h)^{\frac{\alpha}{2}} - (\sigma^{-1}h^{-\frac{\alpha}{2}})} \left|\prod\limits_{\alpha \in \Phi^{+}(\mathfrak{g}', \mathfrak{h}')} (h^{\frac{\alpha}{2}} - h^{-\frac{\alpha}{2}})\right|^{2} \Psi^{\tilde{\G}'/\tilde{\H}'}(h) dh 
\end{eqnarray*}}
\noindent Finally,  using the Weyl integration formula, we get that the character $\Theta_{\Pi'}(h)$ is given, up to a constant, by the formula:
\begin{equation*}
\sum\limits_{\eta \in \mathscr{W}(\G, \mathfrak{h}, m)/ \mathscr{W}(\G, \mathfrak{h})_{m}} \sum\limits_{\sigma \in \mathscr{W}(\K', \mathfrak{h}')} \cfrac{\sgn(\eta) \pr_{m}(\sigma h)^{-\eta^{-1}\mu'}}{\prod\limits_{\underset{\alpha_{|_{\mathfrak{h}'(m)} \neq 0}}{\alpha \in \Phi^{+}(\mathfrak{g}', \mathfrak{h}')}}((\sigma h)^{\frac{\alpha}{2}} - (\sigma h)^{-\frac{\alpha}{2}})}
\end{equation*}

                        
\end{proof}

\begin{coro}

For all $m \in [|m_{\min}, m_{\max}|]$, the character $\Theta_{\Pi'}$ of $\Pi'$ is given by the following formula:
\begin{equation}
\prod\limits_{\alpha \in \Phi^{+}(\mathfrak{g}', \mathfrak{h}')}(h^{\frac{\alpha}{2}} - h^{-\frac{\alpha}{2}}) \Theta_{\Pi'}(h) = C \sum\limits_{\eta \in \mathscr{W}(\G, \mathfrak{h}, m)/\mathscr{W}(\G, \mathfrak{h})_{m}} \sum\limits_{\sigma \in \mathscr{W}(\Z'(m), \mathfrak{h}')} \sum\limits_{\tau \in \mathscr{W}(\K', \mathfrak{h}')} \sgn(\eta\sigma\tau) \pr_{m}(\tau h)^{-\eta^{-1}\mu'}(\tau h)^{\sigma \rho_{\mathfrak{z}'(m)}},
\label{Formule finale caractere U(n, C)}
\end{equation}
where $C$ is a constant and $\rho_{\mathfrak{z}'(m)} = \frac{1}{2} \sum\limits_{\underset{\alpha_{|_{\mathfrak{h}'(m)} = 0}}{\alpha \in \Phi^{+}(\mathfrak{g}', \mathfrak{h}')}} \alpha$.

\end{coro}

\begin{proof}

According to \cite[Chapter~V,~Section~6,~page~319]{KNA}, we get:
\begin{equation}
\prod\limits_{\underset{\alpha_{|_{\mathfrak{h}'(m)}} = 0}{\alpha \in \Phi^{+}(\mathfrak{g}', \mathfrak{h}')}} (h^{\frac{\alpha}{2}} - h^{-\frac{\alpha}{2}}) = \sum\limits_{\sigma \in \mathscr{W}(\Z'(m), \mathfrak{h}')} \sgn(\sigma) h^{\sigma(\rho_{\mathfrak{z}'(m)})}.
\end{equation}

\end{proof}

\section{The dual pair $(\G = \O(2n, \mathbb{R}), \G' = \Sp(2m, \mathbb{R}))$, $n \leq m$}

\label{O(2n,R)}

\noindent Let $V_{\overline{0}}$ be a real vector space of dimension $2n$ endowed with a positive-definite symmetric bilinear form $b_{0}$. We fix $\mathscr{B} = \{v_{1}, v'_{1}, \ldots, v_{n}, v'_{n}\}$ a basis of $V_{\overline{0}}$ such that $\Mat(b_{0}, \mathscr{B}) = \Id_{2n}$ and let $\O(V_{\overline{0}}, b_{0})$ be the group of isometries of the form $b_{0}$, i.e.
\begin{equation}
\O(V_{\overline{0}}, b_{0}) = \left\{g \in \GL(V_{\overline{0}}), b_{0}(gu, gv) = b_{0}(u, v), (\forall u, v \in V_{\overline{0}})\right\}.
\end{equation}
We denote by $\mathfrak{o}(V_{\overline{0}}, b_{0})$ the Lie algebra of $\O(V_{\overline{0}}, b_{0})$ given by
\begin{equation}
\mathfrak{o}(V_{\overline{0}}, b_{0}) = \left\{X \in \End(V_{\overline{0}}), b_{0}(Xu, v) + b_{0}(u, Xv) = 0, (\forall u, v \in V_{\overline{0}})\right\}.
\end{equation}
Writing the endomorphisms $X$ in the basis $\mathscr{B}$, the Lie algebra can be realized as:
\begin{equation}
\mathfrak{o}(V_{\overline{0}}, b_{0}) = \left\{X \in \M(2n, \mathbb{R}), X = -X^{t}\right\}.
\end{equation}
In particular, $\mathfrak{o}(V_{\overline{0}}, b_{0})$ corresponds to the set of skew-symmetric matrices of $\M(2n, \mathbb{R})$ and we get the following decomposition:
\begin{equation}
\mathfrak{o}(V_{\overline{0}}, b_{0}) = \bigoplus\limits_{1 \leq i < j \leq 2n} \mathbb{R}(\E_{i, j} - \E_{j, i}).
\end{equation}
Let $\mathfrak{h}_{0}$ be the subalgebra of $\mathfrak{g}_{0}$ defined by
\begin{equation}
\mathfrak{h}_{0} = \bigoplus\limits_{k = 1}^n \mathbb{R}(\E_{-1+2k, 2k} - \E_{2k, -1+2k}).
\end{equation}
To simplify the notation, let $\H_{k}$ be the matrix defined by $H_{k} = \E_{-1+2k, 2k} - \E_{2k, -1+2k}, 1 \leq k \leq n$. The complexifications of $\mathfrak{g}_{0}$ and $\mathfrak{h}_{0}$ are respectively denoted by $\mathfrak{g}$ and $\mathfrak{h}$. The roots of the Lie algebra $\mathfrak{g}$ with respect to $\mathfrak{h}$, denoted by $\Phi(\mathfrak{g}, \mathfrak{h})$, are given by:
\begin{equation}
\Phi(\mathfrak{g}, \mathfrak{h}) = \left\{\pm e_{i} \pm e_{j}, 1 \leq i \neq j \leq n\right\},
\end{equation}
where the linear form $e_{a}$, $1 \leq a \leq n$, is given by
\begin{equation}
e_{a}\left(\sum\limits_{k = 1}^n h_{k}\H_{k}\right) = -ih_{a}.
\label{Racines}
\end{equation}
We denote by $\pi_{\mathfrak{g}/\mathfrak{h}}$ the product of positive roots. For all $h = \sum\limits_{k = 1}h_{k}\H_{k} \in \mathfrak{h}_{0}$, we get:
\begin{eqnarray*}
\pi_{\mathfrak{g}/\mathfrak{h}}(h) & = & \prod\limits_{\alpha \in \Phi^{+}(\mathfrak{g}, \mathfrak{h})} \alpha(h) = \prod\limits_{1 \leq i < j \leq n} (e_{i} - e_{j})(h) \prod\limits_{1 \leq i < j \leq n} (e_{i} + e_{j})(h) \\
                                                     & = & \prod\limits_{1 \leq i < j \leq n} (-ih_{i} + ih_{j}) \prod\limits_{1 \leq i < j \leq n}(-ih_{i} - ih_{j}) = \prod\limits_{1 \leq i < j \leq n} (-h^{2}_{i} + h^{2}_{j}).
\end{eqnarray*}

\noindent We consider $(V_{\overline{1}}, b_{1})$ a symplectic vector space of dimension $2m$ over $\mathbb{R}$. Let $\mathscr{B}' = \{w_{1}, w'_{1}, \ldots, w_{m}, w'_{m}\}$ be a basis of $V_{\overline{1}}$ such that $\Mat(b_{1}, \mathscr{B}') = \diag\left(\begin{pmatrix} 0 & 1 \\ -1 & 0 \end{pmatrix}, \ldots, \begin{pmatrix} 0 & 1 \\ -1 & 0 \end{pmatrix}\right) = \J_{m, m}$, and we denote by $\Sp(V_{\overline{1}}, b_{1})$ the group of isometries of $b_{1}$, i.e.
\begin{equation}
\Sp(V_{\overline{1}}, b_{1}) = \left\{g \in \GL(V_{\overline{1}}), b_{1}(gu, gv) = b_{1}(u, v), (\forall u, v \in V_{\overline{1}})\right\}.
\end{equation}  
The Lie algebra $\mathfrak{sp}(V_{\overline{1}}, b_{1})$ of $\Sp(V_{\overline{1}}, b_{1})$ is given by
\begin{equation}
\mathfrak{sp}(V_{\overline{1}}, b_{1}) = \left\{X \in \End(V_{\overline{1}}), b_{1}(Xu, v) + b_{1}(u, Xv) = 0, (\forall u, v \in V_{\overline{1}})\right\}.
\end{equation}
Writing the endomorphisms $X$ in the basis $\mathscr{B}'$, we get:
\begin{equation}
\mathfrak{sp}(V_{\overline{1}}, b_{1}) = \left\{X \in \M(2m, \mathbb{R}), \J_{m, m}X + X^{t}\J_{m, m} = 0\right\}.
\end{equation}
Let $\mathfrak{h}_{0}$ be the subalgebra of $\mathfrak{sp}(V_{\overline{1}}, b_{1})$ given by:
\begin{equation}
\mathfrak{h}_{0} = \sum\limits_{k = 1}^m \mathbb{R}(\E_{-1+2k, 2k} - \E_{2k, -1+2k}).
\end{equation}
Similarly, we denote by $\mathfrak{h}$ and $\mathfrak{g}$ the complexifications of $\mathfrak{h}_{0}$ and $\mathfrak{g}_{0}$ respectively. The roots of $\mathfrak{g}$ with respect to $\mathfrak{h}$ are given by:
\begin{equation}
\Phi(\mathfrak{g}, \mathfrak{h}) = \left\{\pm e_{i} \pm e_{j}, 1 \leq i \neq j \leq m\right\} \cup \left\{\pm 2e_{k}, 1 \leq k \leq m\right\},
\end{equation}
where the linear form $e_k$ is defined by 
\begin{equation}
e_{k}\left(\sum\limits_{a=1}^m h_{a}H_{a}\right) = -ih_{k}.
\end{equation}
Let $\K$ be a maximal compact subgroup of $\Sp(V_{\overline{1}}, b_{1})$, $\mathfrak{k}_{0}$ the Lie algebra of $\K$ and $\mathfrak{k}$ its complexification. In particular, we have $\K \approx \U(m, \mathbb{C})$ and
\begin{equation}
\Phi(\mathfrak{k}, \mathfrak{h}) = \left\{\pm (e_{i} - e_{j}), 1 \leq i < j \leq m\right\}.
\end{equation}
For all $h' = \sum\limits_{k = 1}^m h_{k}\H_{k} \in \mathfrak{h}'_{0}$, we get:
\begin{eqnarray*}
\pi_{\mathfrak{g}'/\mathfrak{h}'}(h') & = & \prod\limits_{\alpha \in \Phi^{+}(\mathfrak{g}', \mathfrak{h}')} \alpha(h') = \prod\limits_{1 \leq i < j \leq m} (e_{i} - e_{j})(h') \prod\limits_{1 \leq i < j \leq m} (e_{i} + e_{j})(h') \prod\limits_{k=1}^m 2e_{k}(h') \\
                           & = & \prod\limits_{1 \leq i < j \leq m} (-ih_{i} + ih_{j}) \prod\limits_{1 \leq i < j \leq m} (-ih_{i} - ih_{j}) \prod\limits_{k=1}^m (-2ih_{k}) \\
                           & = & \prod\limits_{1 \leq i < j \leq m} (-h^{2}_{i} + h^{2}_{j}) \prod\limits_{k=1}^m (-2ih_{k})
\end{eqnarray*}   

\noindent Let $V = V_{\overline{0}} \oplus V_{\overline{1}}$, $b = b_{0} \oplus b_{1}$ as defined in Appendix \ref{AppendixA} and let $(\S, \mathfrak{s}(V, b) = \mathfrak{s}(V, b)_{\overline{0}} \oplus \mathfrak{s}(V, b)_{\overline{1}})$ be the corresponding Lie supergroup.

\begin{lemme}

An element $X = \begin{pmatrix} 0 & X_{1} \\ X_{2} & 0 \end{pmatrix}$ is in $\mathfrak{s}(V, \tau)_{\overline{1}}$ if and only if $X_{2} = -J_{m, m}X^{t}_{1}$.

\end{lemme}




\noindent For all $j$, we denote by $V^{j}_{\overline{0}}$ and $V^{j}_{\overline{1}}$ the subspaces of $V_{\overline{0}}$ and $V_{\overline{1}}$ respectively given by:
\begin{equation}
V^{j}_{\overline{0}} = \mathbb{R} v_{j} \oplus \mathbb{R} v'_{j} \qquad V^{j}_{\overline{1}} = \mathbb{R} w_{j} \oplus \mathbb{R} w'_{j}.
\end{equation}     

\noindent We now consider the endomorphisms $u_{j} \in \mathfrak{s}(V, b)_{\overline{1}}, 1 \leq j \leq n$, defined on $V^{j}_{\overline{0}}$ (resp. $V^{j}_{\overline{1}}$) by
\begin{equation}
u_{j}(v_{j}) = \frac{1}{\sqrt{2}}(w_{j} - w'_{j}), \qquad u_{j}(v'_{j}) = \frac{1}{\sqrt{2}}(w_{j} + w'_{j}), \qquad u_{j}(v_{k}) = u_{j}(v'_{k}) = 0, k \neq j,
\end{equation}
\begin{equation}
u_{j}(w_{j}) = \frac{1}{\sqrt{2}}(v_{j} - v'_{j}), \qquad u_{j}(w'_{j}) = \frac{1}{\sqrt{2}}(v_{j} + v'_{j}), \qquad u_{j}(w_{k}) = u_{j}(w'_{k}) = 0, k \neq j,
\end{equation}
and let $\mathfrak{h}_{1}$ be the subspace of $\mathfrak{s}(V, \tau)_{\overline{1}}$ given by
\begin{equation}
\mathfrak{h}_{1} = \sum\limits_{j = 1}^{n} \mathbb{R}u_{j}.
\end{equation}
We define the moment maps $\tau: \mathfrak{s}(V, b)_{\overline{1}}  \to \mathfrak{o}(2n, \mathbb{R})$ and $\tau': \mathfrak{s}(V, b)_{\overline{1}} \to \mathfrak{sp}(2m, \mathbb{R})$ by:
\begin{equation}
\tau(w) = w^{2}|_{V_{\overline{0}}}, \qquad \tau'(w) = w^{2}|_{V_{\overline{1}}}.
\end{equation}
More precisely, for all $w = \begin{pmatrix} 0 & X_{1} \\ \J^{-1}_{m, m}X^{t}_{1} & 0 \end{pmatrix} \in \mathfrak{s}(V, b)_{\overline{1}}$, we have
\begin{equation}
\tau\begin{pmatrix} 0 & X_{1} \\ \J^{-1}_{m, m}X^{t}_{1} & 0 \end{pmatrix} = X_{1}\J^{-1}_{m, m}X^{t}_{1}, \qquad \tau\begin{pmatrix} 0 & X_{1} \\ \J^{-1}_{m, m}X^{t}_{1} & 0 \end{pmatrix} = \J^{-1}_{m, m}X^{t}_{1}X_{1}.
\end{equation}
We consider the injection of $\mathfrak{h}$ into $\mathfrak{h}'$ given by:
\begin{equation}
\mathfrak{h} \ni \sum\limits_{k=1}^n h_{k}\H_{k} \mapsto \sum\limits_{k=1}^n h_{k}\H_{k} \in \mathfrak{h}'.
\label{PreviousEquation}
\end{equation}
We denote by $\tilde{\mathfrak{h}}$ the image of $\mathfrak{h}$ into $\mathfrak{h}'$ via the map \eqref{PreviousEquation}. Let $\mathfrak{z}'$ (resp. $Z'$) be the subalgebra (resp. subgroup) of $\mathfrak{g}'$ (resp. $\G'$) defined by
\begin{equation}
\mathfrak{z}' = \{x \in \mathfrak{g}', [x, y] = 0, (\forall y \in \tilde{\mathfrak{h}})\},
\end{equation}
\begin{equation}
Z' = \{g' \in \G', \Ad(g')y = y, (\forall y' \in \tilde{\mathfrak{h}})\}.
\end{equation}
We define $\Phi(\mathfrak{g}', \mathfrak{z}') = \{\alpha \in \Phi(\mathfrak{g}', \mathfrak{h}'), \alpha_{|_{\tilde{\mathfrak{h}}}} \neq 0\}$ and $\Phi(\mathfrak{z}', \mathfrak{h}') = \Phi(\mathfrak{g}', \mathfrak{h}') \setminus \Phi(\mathfrak{g}', \mathfrak{z}')$, and let:
\begin{equation}
\pi_{\mathfrak{g}'/\mathfrak{z}'} = \prod\limits_{\alpha \in \Phi^{+}(\mathfrak{g}', \mathfrak{z}')} \alpha, \qquad \pi_{\mathfrak{g}'/\mathfrak{z}'} = \prod\limits_{\alpha \in \Phi^{+}(\mathfrak{z}', \mathfrak{h}')} \alpha.
\end{equation}

\begin{lemme}

For all $h' = \sum\limits_{k=1}^m h_{k}\H_{k} \in \mathfrak{h}'_{0}$, we have:
\begin{equation}
\pi_{\mathfrak{g}'/\mathfrak{z}'}(h') = \prod\limits_{1 \leq i < j \leq n} (-h^{2}_{i} + h^{2}_{j}) \prod\limits_{i = 1}^n \prod\limits_{j = n+1}^m (-h^{2}_{i} + h^{2}_{j}) \prod\limits_{k = 1}^n (-2ih_{k}) \qquad \pi_{\mathfrak{z}'/\mathfrak{h}'}(h') = \prod\limits_{n+1 \leq i < j \leq m} (-h^{2}_{i} + h^{2}_{j}) \prod\limits_{k=n+1}^{m} (-2ih_{k}).
\end{equation}
In particular, we have:
\begin{equation}
\pi_{\mathfrak{g}'/\mathfrak{z}'}(\pr(h')) =  \prod\limits_{1 \leq i < j \leq n} (-h^{2}_{i} + h^{2}_{j}) \prod\limits_{i = 1}^n (-h^{2}_{i})^{m - n} \prod\limits_{k = 1}^n (-2ih_{k}).
\end{equation}

\end{lemme}

\noindent Let $\mathscr{W}(\G, \mathfrak{h})$, $\mathscr{W}(\G', \mathfrak{h}')$ and $\mathscr{W}(\K', \mathfrak{h}')$ be the Weyl groups of $\G$, $\G'$ and $\K'$ respectively.

\noindent For every function $\phi \in \S(W)$, we define the function $f_{\phi}$ on $\tau(\mathfrak{h}^{\reg}_{1})$ by 
\begin{equation}
f_{\phi}(\tau(w)) := C_{\mathfrak{h}_{1}} \pi_{\mathfrak{g}'/\mathfrak{z}'}(\tau'(w)) \displaystyle\int_{\G'/\Z'} \phi(g'w)d(g'\Z')
\end{equation}
where $C_{\mathfrak{h}_{1}}$ is a constant of modulus 1 defined in \cite[Lemma~8,~page~17]{TOM4}.

\noindent As in the previous section, we denote by $\H^{\sharp}$ the two fold cover of $\widetilde{\H}$ such that the linear forms $\frac{\alpha}{2}$ are analytic integral for every roots $\alpha \in \Phi(\mathfrak{g}, \mathfrak{h})$ and let $\H^{\sharp}_{0}$ be its connected identity component. Let $\xi^{\sharp}_{\frac{\alpha}{2}}: \H^{\sharp}_{0} \to \S^{1}$ the multiplicative character having the linear form $\frac{\alpha}{2}$ as derivative and let $c^{\sharp}_{-}: \mathfrak{h} \to \H^{\sharp}_{0}$ be the extension of $c_{-}$ on $ \H^{\sharp}_{0}$ (section \ref{Section5.1}).

\begin{prop}

Fix $\Pi \in \mathscr{R}(\tilde{\G}, \omega^{\infty})$ of highest weight $\nu$ and let $\mu = \nu + \rho$. For every $\phi \in \S(W)$, the following formula holds:
\begin{equation}
T(\overline{\Theta_{\Pi}})(\phi) = \displaystyle\int_{\mathfrak{h}}c_{-}(x)^{\mu'}\ch(x)^{2m-2n} \displaystyle\int_{\mathfrak{h} \cap \tau(W)} e^{iB(x, y)}f_{\phi}(y)dydx,
\label{Proposition2107}
\end{equation}
where $\ch(x) = |\det(x-1)|^{\frac{1}{2}}$ and $\mu' \in \mathfrak{h}^{*}$ is defined by
\begin{equation}
\mu'_{j} = \mu_{n+1-j} \qquad (1 \leq j \leq n).
\end{equation}

\end{prop}

\begin{proof}

For a proof, we refer the reader to \cite[Corollary~38,~page~47]{TOM4}.

\end{proof}

\begin{rema}

\begin{enumerate}
\item By \cite{VERGNE2}, the weights $\nu$ and $\nu'$ of the representations $\Pi$ and $\Pi'$ respectively are given by:
\begin{equation}
\nu = \sum\limits_{a=1}^k \nu_{a}e_{a} \mapsto \nu' = -\sum\limits_{a=1}^m ne_{a} - \sum\limits_{a = 1}^k \nu_{m+1-a}e_{a},
\end{equation}
where $0 \leq k \leq n$ and $\nu_{1} \geq \ldots \geq \nu_{k} > 0$. By fixing $k = m$ and considering a decreasing sequence $\nu_{1} \geq \ldots \geq \nu_{m} \geq 0$ with at most $n$ non zero $\nu_{i}$, we get
\begin{equation}
\nu = \sum\limits_{a=1}^n \nu_{a}e_{a} \mapsto \nu' = -\sum\limits_{a=1}^m ne_{a} - \sum\limits_{a = 1}^m \nu_{m+1-a}e_{a} = \sum\limits_{a = 1}^m \tau_{a}e_{a},
\end{equation}
with $\tau_{a} = -n - \nu_{m+1-a}$, where $(\tau_{a})_{a \in [|1, m|]}$ is a decreasing sequence of negative numbers.

\noindent The linear form $\rho$ is given by:
\begin{eqnarray*}
\rho & = & \frac{1}{2} \sum\limits_{\alpha \in \Phi^{+}(\mathfrak{g}, \mathfrak{h})} \alpha = \frac{1}{2} \sum\limits_{1 \leq i < j \leq n} (e_{i} - e_{j}) + \frac{1}{2} \sum\limits_{1 \leq i < j \leq n} (e_{i} + e_{j}) \\
       & = & \frac{1}{2} \sum\limits_{1 \leq i < j \leq n} 2e_{i} = \sum\limits_{k=1}^n (n-k)e_{k}
\end{eqnarray*}
Then $\mu = \sum\limits_{k = 1}^n (\nu+\rho)_{a}e_{a}$ with $(\nu+\rho)_{a} = \nu_{a} + n -a$ and 
\begin{equation}
\mu'_{a} = \mu_{n+1-a} = \nu_{n+1-a} + n -(n+1-a) = \nu_{n+1-a} + a -1.
\end{equation}
\item For every $x, y \in \mathfrak{g}$ or $\mathfrak{g}'$, we denote by $B$ the bilinear form defined by
\begin{equation}
B(x, y) = \Re\tr(xy)
\label{FormeB1}
\end{equation}
The form $B$ is $\G$ (resp. $\G'$)-invariant and non-degenerate on $\mathfrak{g}$ and $\mathfrak{g}'$. More precisely, for all $x = \sum\limits_{k=1}^n x_{k}\H_{k, k}$ (resp. $x' = \sum\limits_{k=1}^{m} x'_{k}\H_{k, k}$) and $y = \sum\limits_{k=1}^n y_{k}\H_{k, k}$ (resp. $y' = \sum\limits_{k=1}^{m} y'_{k}\H_{k, k}$), the form $B$ is given by
\begin{equation}
B(x, y) = \sum\limits_{k=1}^n -\pi x_{k} y_{k} \qquad \left(\text{resp. } B(x', y') = \sum\limits_{k=1}^m -\pi x'_{k} y'_{k}\right).
\end{equation}
See \cite[page~49-50]{TOM4}.
\end{enumerate}

\end{rema}

\begin{theo}

For all regular element $y \in \tilde{\mathfrak{h}}$ and $y' \in \mathfrak{h}'^{\reg}$, we get:
\begin{equation}
\pi_{\mathfrak{g}'/\mathfrak{z}'}(y) \displaystyle\int_{\G'/\Z'} e^{iB(y', g'y)}d(g'\Z') = (-i)^{\frac{1}{2}\dim(\mathfrak{g}'/\mathfrak{z}')}(-1)^{n(\mathfrak{h'})} \sum\limits_{\sigma \in \mathscr{W}(\K', \mathfrak{h}')/\mathscr{W}(\K', \mathfrak{h}')^{\tilde{\mathfrak{h}}}} \cfrac{e^{iB(y', \sigma y)}}{\pi_{\mathfrak{g}'/\mathfrak{z}'}(\sigma^{-1}y')}.
\label{Theoreme1}
\end{equation}
where $n(\tilde{\mathfrak{h}})$ is the number of non-compacts positives roots which do not vanish on $\tilde{\mathfrak{h}}$.

\end{theo}

\noindent It is clear that for all $\psi \in \mathscr{C}^{\infty}_{c}(\G'.\mathfrak{h}'^{\reg})$, 
\begin{equation}
\phi: W \ni w \mapsto \displaystyle\int_{\mathfrak{g}'} \chi_{x'}(w) \psi(x')dx' \in \mathbb{C}
\end{equation}
is a Schwartz function on $W$.
Using the Weyl integration formula, we get for all $y \in \tau(\mathfrak{h}^{\reg}_{1})$
\begin{eqnarray*}
& & \pi_{\mathfrak{g}'/\mathfrak{z}'}(y) \displaystyle\int_{\G'/\Z'} \displaystyle\int_{\mathfrak{g}'} e^{i B(y', g'y)} \psi(y')dy'd(g'\Z') \\
& = & \cfrac{1}{|\mathscr{W}(\K', \mathfrak{h}')|} \displaystyle\int_{\mathfrak{h}'} \left( \pi_{\mathfrak{g}'/\mathfrak{z}'}(y) \displaystyle\int_{\G'/\Z'} e^{i B(y', g'y)} d(g'\Z'(m)) \right) \psi^{\tilde{\G}'/\tilde{\H}'}(y') |\pi_{\mathfrak{g}'/\mathfrak{h}'}(y')|^{2}dy' 
\end{eqnarray*}
where $\psi^{\tilde{\G}'/\tilde{\H}'} : \mathfrak{h}'^{\reg} \to \mathbb{C}$ is given by:
\begin{equation}
\psi^{\tilde{\G'}/\tilde{\H'}}(y') = \displaystyle\int_{\tilde{\G'}/\tilde{\H'}} \psi(\Ad(\tilde{g}')(y')) d(\tilde{g'}\tilde{\H'}).
\label{psiG'H'}
\end{equation}
Using the same arguments from the previous section, we get the following equality:
\begin{equation}
f_{\phi}(y) = C \displaystyle\int_{\mathfrak{h}'}e^{i B(y', y)} \pi_{\mathfrak{z}'/\mathfrak{h}'}(y') \overline{\pi_{\mathfrak{g}'/\mathfrak{h}'}(y')} \psi^{\tilde{\G}'/\tilde{\H}'}(y')dy'.
\label{Allan157}
\end{equation}
We denote by $\mathscr{A}_{\psi}$ the function on $\mathfrak{h}'^{\reg}$ defined by
\begin{equation}
\mathscr{A}_{\psi}(y') = \pi_{\mathfrak{g}'/\mathfrak{h}'}(y') \psi^{\tilde{\G}'/\tilde{\H}'}(y') \qquad (y' \in \mathfrak{h}'^{\reg}).
\end{equation}
Using \eqref{Allan157}, we get:
\begin{equation}
f_{\phi} = \mathscr{F}_{\mathfrak{h}'}(C\pi_{\mathfrak{z}'/\mathfrak{h}'} \mathscr{A}_{\psi})_{|_{\tilde{\mathfrak{h}}}}.
\label{RemaFourier1908}
\end{equation}

\noindent Before giving the main theorem, we prove a result concerning the right hand-side of \eqref{Proposition2107}. More particularly, we are interested in the support of the following distribution:
\begin{equation}
\displaystyle\int_{\mathfrak{h}} c_{-}(x)^{-\mu'} \ch(x)^{2m-2n}e^{i B(x, y)} dx.
\label{DistributionSupport1908}
\end{equation}

\begin{prop}

Let $\beta_{1}$ and $\beta_{2}$ be the two integers defined by:
\begin{equation}
\beta_{1} = \begin{cases} \max\left\{k \geq 1, \mu'_{k} \leq -(m-n+1)\right\} & \text{ if } \mu'_{1} \leq -(m-n+1) \\ 0 & \text{ otherwise } \end{cases}
\end{equation}
\begin{equation}
\beta_{2} = \begin{cases} \min\left\{ k \geq 1, \mu'_{k} \geq m-n+1\right\} - 1 & \text{ if } \mu'_{n} \geq m - n +1 \\ n & \text{ otherwise } \end{cases}
\end{equation}
Then,
\begin{equation}
\supp \displaystyle\int_{\mathfrak{h}} c_{-}(x)^{-\mu'} \ch(x)^{2m-2n}e^{i B(x, y)} dx = \left\{\sum\limits_{k=1}^{\beta_{1}} h_{k}\H_{k} + \sum\limits_{k=\beta_{2}+1}^{n} y_{k}\H_{k}, h_{i} > 0, y_{i} < 0 \right\}.
\label{Support2107}
\end{equation}

\end{prop}

\begin{proof}

According to \cite[Lemma~33,~page~43]{TOM4}, we have:
\begin{equation}
\displaystyle\int_{\mathfrak{h}} c_{-}(x)^{-\mu'} \ch(x)^{2m-2n}e^{iB(x, y)} dx = \prod\limits_{k=1}^n \displaystyle\int_{\mathbb{R}} (1+ix_{k})^{-a_{k}}(1-ix_{k})^{-b_{k}} e^{-i\pi x_{k}y_{k}} dx_{k},
\end{equation}
where $a_{k} = \mu'_{k} - m + n = \nu_{n+1-k} + k - m + n -1$ and $b_{k} = -\mu'_{k} - m + n = -\nu_{n+1-k} - k - m + n +1$. In particular, $a_{k} \geq 1$ (resp. $b_{k} \geq 1$) implies $b_{k} \leq 0$ (resp. $a_{k} \leq 0$).

\noindent Using \cite[Appendix~C]{TOM4}, we get
\begin{equation}
\displaystyle\int_{\mathbb{R}} (1+ix_{k})^{-a_{k}}(1-ix_{k})^{-b_{k}} e^{-i\pi x_{k}y_{k}} dx_{k} = P_{a_{k}, b_{k}}(\pi y_{k})e^{-|\pi y_{k}|} + Q_{a_{k}, b_{k}}\left(\cfrac{\partial}{\partial y_{k}}\right)\delta_{0}(\pi y_{k}),
\end{equation}
where $P_{a_{k}, b_{k}}$ and $Q_{a_{k}, b_{k}}$ are the polynomials defined in \cite[Appendix~C]{TOM4}. In particular, using \cite[Appendix~C]{TOM4}, the support of the polynomials $P_{a_{k}, b_{k}}$ are
\begin{equation}
\supp P_{a_{k}, b_{k}} = \begin{cases} ]-\infty, 0] & \text{ if } a_{k} \geq 1 \text{ and } b_{k} \leq 0 \\ [0, +\infty[ & \text{ if } a_{k} \leq 0 \text{ and } b_{k} \geq 1 \\ \{\emptyset\} & \text{ otherwise } \end{cases} = \begin{cases} ]-\infty, 0] & \text{ if } \mu'_{k} \geq m-n+1 \\ [0, +\infty[ & \text{ if } \mu'_{k} \leq -(m-n+1) \\ \{\emptyset\} & \text{ otherwise } \end{cases}.
\end{equation}
Then,
\begin{equation}
\supp \displaystyle\int_{\mathbb{R}} (1+ix_{k})^{-a_{k}}(1-ix_{k})^{-b_{k}} e^{-i\pi x_{k}y_{k}} dx_{k} = \begin{cases} [0, +\infty[ & \text{ if } k \leq \beta_{1} \\ \{0\} & \text{ if } k \in [|\beta_{1}+1, \beta_{2}|] \\ ]-\infty, 0] & \text{ if } k \geq \beta_{2}+1 \end{cases}
\end{equation}
and finally
\begin{equation}
\supp \displaystyle\int_{\mathfrak{h}} c_{-}(x)^{-\mu'} \ch(x)^{2m-2n}e^{iB(x, y)} dx = \left\{\sum\limits_{k = 1}^{\beta_{1}} ih_{k}H_{k} + \sum \limits_{k = \beta_{2} +1}^{n} iy_{k}H_{k}, h_{k} > 0, y_{k} < 0\right\} \subseteq \tau(\mathfrak{h}_{1}).
\end{equation}

\end{proof}

\noindent From now on, we denote by $\supp(\beta_{1}, \beta_{2})$ the support of the distribution given in \eqref{DistributionSupport1908}. We denote by $\tilde{\mathfrak{h}}^{\perp}$ the orthogonal complement of $\tilde{\mathfrak{h}}$ with respect to the bilinear form $B$, and by $\pr: \mathfrak{h} \to \tilde{\mathfrak{h}}$ the associated projection, i.e.
\begin{equation}
\pr(y_{1} + y_{2}) = y_{1} \qquad (y_{1} \in \tilde{\mathfrak{h}}, y_{2} \in \tilde{\mathfrak{h}}^{\perp}).
\label{label projection 1908}
\end{equation}
Finally, let us consider the subgroup $\mathscr{W}(\G, \mathfrak{h}, \beta)$ of $\mathscr{W}(\G, \mathfrak{h})$ defined by:
\begin{equation}
\mathscr{W}(\G, \mathfrak{h}, \beta) = \{\eta \in \mathscr{W}(\G, \mathfrak{h}), \supp(\beta_{1}, \beta_{2}) \subseteq \eta \tau(\mathfrak{h}_{1})\}.
\end{equation}

\begin{theo}
The character $\Theta_{\Pi'}$ of $\Pi' \in \mathscr{R}(\tilde{\G}', \omega^{\infty})$ is given, up to a constant, by the following formula:
\begin{equation}
\Theta_{\Pi'}(h) = \sum\limits_{\eta \in \mathscr{W}(\G, \mathfrak{h}, \beta)} \sum\limits_{\sigma \in \mathscr{W}(\K', \mathfrak{h}')} \cfrac{\sgn(\eta) \pr(\sigma h)^{-\eta^{-1}\mu'}}{\prod\limits_{\underset{\alpha_{|_{\tilde{\mathfrak{h}}} \neq 0}}{\alpha \in \Phi^{+}(\mathfrak{g}', \mathfrak{h}')}}((\sigma h)^{\frac{\alpha}{2}} - (\sigma h)^{-\frac{\alpha}{2}})}.
\label{Formule finale caractere 1}
\end{equation}

\end{theo}

\begin{proof}

By \eqref{Proposition2107} and \eqref{Support2107}, we get
\begin{equation}
T(\overline{\Theta_{\Pi}})(\phi) = \displaystyle\int_{\mathfrak{h}} c_{-}(x)^{-\mu'}\ch(x)^{2m-2n} \displaystyle\int_{\mathscr{W}(\G, \mathfrak{h})\tau(\mathfrak{h}_{1})} e^{i B(x, y)} f_{\phi}(y)dy dx 
\end{equation}
Then,
\begin{eqnarray}
& & \displaystyle\int_{\mathfrak{h}} c_{-}(x)^{-\mu'}\ch(x)^{2m-2n} \displaystyle\int_{\mathscr{W}(\G, \mathfrak{h})\tau(\mathfrak{h}_{1, m})} e^{i B(x, y)} f_{\phi}(y)dy dx \nonumber \\
& = & \sum\limits_{\eta_{0} \in \mathscr{W}(\G, \mathfrak{h},\beta)} \displaystyle\int_{\mathfrak{h}} c_{-}(x)^{-\mu'}\ch(x)^{2m-2n} \displaystyle\int_{\tau(\mathfrak{h}_{1})} e^{i B(x, \eta_{0}(y))} f_{\phi}(\eta_{0}(y)) dydx \nonumber \\
& = & \sum\limits_{\eta_{0} \in \mathscr{W}(\G, \mathfrak{h}, \beta)} \sgn(\eta_{0})\displaystyle\int_{\mathfrak{h}} c_{-}(x)^{-\eta^{-1}_{0}\mu'}\ch(x)^{2m-2n} \displaystyle\int_{\tau(\mathfrak{h}_{1})} e^{i B(x, y)} f_{\phi}(y) dydx \label{int11}
\end{eqnarray}
Again, playing with the support of $f_{\phi}$, we get:
\begin{equation*}
\displaystyle\int_{\tau(\mathfrak{h}_{1})} e^{i B(x, y)} f_{\phi}(y) dy = \displaystyle\int_{\mathfrak{h}} e^{i B(x, y)} f_{\phi}(y) dy.
\end{equation*}
We identify $\mathfrak{h}$ with its image $\tilde{\mathfrak{h}}$ in $\mathfrak{h}'$. Hence, \eqref{int11} can be written as:
\begin{equation}
\sum\limits_{\eta_{0} \in \mathscr{W}(\G, \mathfrak{h}, \beta)} \sgn(\eta_{0})\displaystyle\int_{\tilde{\mathfrak{h}}} c_{-}(x)^{-\eta^{-1}_{0}\mu'}\ch(x)^{2m-2n} \displaystyle\int_{\tilde{\mathfrak{h}}} e^{i B(x, y)} f_{\phi}(y) dydx.
\label{int2208}
\end{equation}
By \eqref{RemaFourier1908}, this can be written as:
\begin{eqnarray}
& & \sum\limits_{\eta_{0} \in \mathscr{W}(\G, \mathfrak{h}, \beta)} \sgn(\eta_{0})\displaystyle\int_{\mathfrak{h}} c_{-}(x)^{-\eta^{-1}_{0}\mu'}\ch(x)^{2m-2n} \displaystyle\int_{\tau(\mathfrak{h}_{1})} e^{i\pi B(x, y)} f_{\phi}(y) dydx \nonumber \\
& = & \sum\limits_{\eta_{0} \in \mathscr{W}(\G, \mathfrak{h}, \beta)} \sgn(\eta_{0})\displaystyle\int_{\tilde{\mathfrak{h}}} c_{-}(x)^{-\eta^{-1}_{0}\mu'}\ch(x)^{2m-2n} \displaystyle\int_{\tilde{\mathfrak{h}}} e^{i\pi B(x, y)} \left. \mathscr{F}_{\mathfrak{h}'}(\pi_{\mathfrak{z}'/\mathfrak{h}'} \mathscr{A}_{\psi})\right|_{\tilde{\mathfrak{h}}}(y) dydx \nonumber \\
& = & \sum\limits_{\eta_{0} \in \mathscr{W}(\G, \mathfrak{h}, \beta)} \sgn(\eta_{0})\displaystyle\int_{\tilde{\mathfrak{h}}} c_{-}(x)^{-\eta^{-1}_{0}\mu'}\ch(x)^{2m-2n} \mathscr{F}^{-1}_{\mathfrak{h}'} \left. \mathscr{F}_{\mathfrak{h}'}(\pi_{\mathfrak{z}'/\mathfrak{h}'} \mathscr{A}_{\psi})\right|_{\tilde{\mathfrak{h}}}(y) dydx \nonumber \\
& = & \sum\limits_{\eta_{0} \in \mathscr{W}(\G, \mathfrak{h}, \beta)} \sgn(\eta_{0})\displaystyle\int_{\tilde{\mathfrak{h}}} c_{-}(x')^{-\eta^{-1}_{0}\mu'}\ch(x')^{2m-2n} \displaystyle\int_{\tilde{\mathfrak{h}}^{\perp}} (\pi_{\mathfrak{z}'/\mathfrak{h}'} \mathscr{A}_{\psi})(x' + y') dy'dx' \nonumber \\
& = & \sum\limits_{\eta_{0} \in \mathscr{W}(\G, \mathfrak{h}, \beta)} \sgn(\eta_{0})\displaystyle\int_{\tilde{\mathfrak{h}}} c_{-}(x')^{-\eta^{-1}_{0}\mu'}\ch(x')^{2m-2n} \displaystyle\int_{\tilde{\mathfrak{h}}^{\perp}} \pi_{\mathfrak{z}'/\mathfrak{h}'}(y') \mathscr{A}_{\psi}(x' + y') dy'dx' \label{int33}
\end{eqnarray}
We now choose a particular $\psi \in \mathscr{C}^{\infty}_{c}(\G'.\mathfrak{h}'^{\reg})$. We denote by $\G'^{c} \subseteq \G'$ the domain of the Cayley transform in $\G'$. Fix $\Psi \in \mathscr{C}^{\infty}_{c}(\widetilde{\G}')$ such that $\supp(\Psi) \in \widetilde{\G'^{c}} \subseteq \widetilde{\G}'$ and let $j_{\mathfrak{g}'}$ be the Jacobian of $c: \mathfrak{g}'^{c} \to \G'^{c}$. We recall that $\tilde{c}_{-}$ is defined by $\tilde{c}_{-}(x') = \tilde{c}(x') \tilde{c}(0)^{-1}$ and let $\psi$ be the function given by
\begin{equation}
\psi(x') = \Theta(\tilde{c}(x'))j_{\mathfrak{g}'}(x')\Psi(\tilde{c}_{-}(x')) \qquad (x' \in \mathfrak{g}'^{c}).
\label{choixpsi1}
\end{equation}
With such a function $\psi$, the integral \eqref{int33} becomes:
\begin{equation}
\sum\limits_{\eta_{0} \in \mathscr{W}(G, \mathfrak{h}, \beta)} \sgn(\eta_{0})\displaystyle\int_{\mathfrak{h}'} c_{-}(\pr(x))^{-\eta^{-1}_{0}\mu'}\ch(\pr_{m}(x))^{2m-2n}\pi_{\mathfrak{z}'/\mathfrak{h}'}(x) \pi_{\mathfrak{g}'/\mathfrak{h}'}(x) \Theta(\tilde{c}(x)) j_{\mathfrak{g}'}(x) \Psi^{\tilde{\G}'/\tilde{\H}'}(\tilde{c}_{-}(x)) dx,
\label{int444}
\end{equation}
where $\pr$ is the projection given in \eqref{label projection 1908}. To get the character $\Theta_{\Pi'}$ of $\Pi'$, we would like to write the integral defined in \eqref{int444} as an integral over $\widetilde{\H}'$. 

\noindent For all $h' = \sum\limits_{k=1}^m h_{k}\H_{k}$, we get:
\begin{equation}
\det(h' - 1) = \prod\limits_{k = 1}^m (1 + h^{2}_{k}).
\end{equation}
Indeed, for all $h \in \mathbb{R}$, we denote by $\A(h)$ the matrix defined by $\A(h) = \begin{pmatrix} -1 & h \\ -h & -1 \end{pmatrix}$. Then,
\begin{equation}
\det(h'-1) = \det(\diag(\A(h_{1}), \ldots, \A(h_{m}))) = \prod\limits_{k = 1}^m \det(\A(h_{k})) = \prod\limits_{k=1}^m (1+h^{2}_{k}).
\end{equation}
Moreover, we have $\Theta(\tilde{c}(h')) = \ch(h')^{2n} = \prod\limits_{k=1}^{m} (1+h^{2}_{k})^{n}$.

\noindent Now, we focus our attention on the term $\prod\limits_{\alpha \in \Phi^{+}(\mathfrak{g}', \mathfrak{h}')} (\xi_{\frac{\alpha}{2}}(c^{\sharp}_{-}(h')) - \xi_{-\frac{\alpha}{2}}(c^{\sharp}_{-}(h')))$. 

\noindent First, we remark that $c(h') = \diag\left(c\begin{pmatrix} 0 & h_{1} \\ -h_{1} & 0 \end{pmatrix}, \ldots, c\begin{pmatrix} 0 & h_{m} \\ -h_{m} & 0 \end{pmatrix}\right)$. Moreover, for all $x \in \mathbb{R}$, we have:
\begin{eqnarray*}
c\begin{pmatrix} 0 & x \\ -x & 0 \end{pmatrix} & = & \begin{pmatrix} 1 & x \\ -x & 1 \end{pmatrix}\begin{pmatrix} -1 & x \\ -x & -1 \end{pmatrix}^{-1} = \cfrac{1}{1+x^{2}}\begin{pmatrix} 1 & x \\ -x & 1 \end{pmatrix}\begin{pmatrix} -1 & -x \\ x & -1 \end{pmatrix} \\
                       & = & \begin{pmatrix} \frac{-1+x^{2}}{1+x^{2}} & \frac{-2x}{1+x^{2}} \\ \frac{2x}{1+x^{2}} & \frac{-1+x^{2}}{1+x^{2}} \end{pmatrix} = R\left(\arccos\left(\frac{-1+x^{2}}{1+x^{2}}\right)\right)
\end{eqnarray*}
where $R(\theta)$ is the rotation matrix of angle $\theta$. Then,
\begin{equation}
c(h') = \diag\left(R\left(Y_{1} = \arccos\left(\frac{-1+h^{2}_{1}}{1+h^{2}_{1}}\right)\right), \ldots, R\left(Y_{m} = \arccos\left(\frac{-1+h^{2}_{m}}{1+h^{2}_{m}}\right)\right)\right)
\end{equation}

\noindent and
{\small
\begin{eqnarray*}
& & \prod\limits_{\alpha \in \Phi^{+}(\mathfrak{g}', \mathfrak{h}')} (\xi^{\sharp}_{\frac{\alpha}{2}}(c^{\sharp}_{-}(h')) - \xi^{\sharp}_{-\frac{\alpha}{2}}(c^{\sharp}_{-}(h'))) \\
& = & \prod\limits_{1 \leq i < j \leq m} \left(\xi^{\sharp}_{\frac{e_{i} - e_{j}}{2}}(c^{\sharp}_{-}(h')) - \xi^{\sharp}_{-\frac{e_{j}-e_{i}}{2}}(c^{\sharp}_{-}(h'))\right) \prod\limits_{1 \leq i < j \leq m} \left(\xi^{\sharp}_{\frac{e_{i} + e_{j}}{2}}(c^{\sharp}_{-}(h')) - \xi^{\sharp}_{-\frac{e_{j} + e_{i}}{2}}(c^{\sharp}_{-}(h'))\right) \prod\limits_{k=1}^m \left(\xi^{\sharp}_{e_{k}}(c^{\sharp}_{-}(h')) - \xi^{\sharp}_{-e_{k}}(c^{\sharp}_{-}(h'))\right) \\    
                      & = & \prod\limits_{1 \leq i < j \leq m} \left(e^{\frac{i(Y_{i} - Y_{j})}{2}} - e^{\frac{i(-Y_{i} + Y_{j})}{2}}\right)\prod\limits_{1 \leq i < j \leq m} \left(e^{\frac{i(Y_{i} + Y_{j})}{2}} - e^{\frac{-i(Y_{i} + Y_{j})}{2}}\right) \prod\limits_{k=1}^m \left(e^{iY_{k}} - e^{-iY_{k}}\right)\\                  
                      & = & \prod\limits_{1 \leq i < j \leq m} (e^{iY_{i}} - e^{-iY_{j}} - e^{iY_{j}} + e^{iY_{j}}) = \prod\limits_{1 \leq i < j \leq m} (2\cos(Y_{i}) - 2\cos(Y_{j})) \prod\limits_{k=1}^m 2i\sin(Y_{k})\\
                      & = & \prod\limits_{1 \leq i < j \leq m} 2\left(\cfrac{1 - h^{2}_{i}}{1+h^{2}_{i}} - \cfrac{1 - h^{2}_{i}}{1+h^{2}_{i}}\right)\prod\limits_{k=1}^m 2i\sqrt{1 - \cos^{2}(Y_{k})} \\
                      & = & \prod\limits_{1 \leq i < j \leq m}2\left(\cfrac{(1 - h^{2}_{i})(1 + h^{2}_{j}) - (1 - h^{2}_{j})(1 + h^{2}_{i}}{(1 + h^{2}_{i})(1 + h^{2}_{j})}\right) \prod\limits_{k=1}^m 2i\sqrt{1 - \cos\left(\arccos\left(\cfrac{-1+h^{2}_{k}}{1+h^{2}_{k}}\right)\right)^{2}} \\
                      & = &  \prod\limits_{1 \leq i < j \leq m} 4 \left(\cfrac{h^{2}_{j} - h^{2}_{i}}{(1 + h^{2}_{i})(1 + h^{2}_{j})}\right) \prod\limits_{k=1}^m 2i \sqrt{\left(1 - \left(\cfrac{-1+h^{2}_{k}}{1+h^{2}_{k}}\right)^{2}\right)} =  \prod\limits_{1 \leq i < j \leq m} 4 \left(\cfrac{h^{2}_{j} - h^{2}_{i}}{(1 + h^{2}_{i})(1 + h^{2}_{j})}\right) \prod\limits_{k=1}^m 2i\sqrt{\cfrac{4h^{2}_{k}}{(1+h^{2}_{k})^{2}}} \\
                      & = & \prod\limits_{1 \leq i < j \leq m} 4 \left(\cfrac{h^{2}_{j} - h^{2}_{i}}{(1 + h^{2}_{i})(1 + h^{2}_{j})}\right) \prod\limits_{k=1}^m \cfrac{4ih_{k}}{(1+h^{2}_{k})} = \cfrac{\prod\limits_{1 \leq i < j \leq m} 4(h^{2}_{j} - h^{2}_{i})}{\prod\limits_{k=1}^m (1+h^{2}_{k})^{m-1}} \cfrac{\prod\limits_{k=1}^m 4ih_{k}}{\prod\limits_{k=1}^m (1+h^{2}_{k})} \\
                      & = & \cfrac{\prod\limits_{1 \leq i < j \leq m} 4(h^{2}_{j} - h^{2}_{i}) \prod\limits_{k=1}^m 2ih_{k}}{\prod\limits_{k=1}^m (1+h^{2}_{k})^{m}} = C \cfrac{\pi_{\mathfrak{g}'/\mathfrak{h}'}(h')}{\prod\limits_{k=1}^m (1+h^{2}_{k})^{m}}.
\end{eqnarray*}  }             
Similarly, we obtain:{\small
\begin{eqnarray*}
& & \prod\limits_{\underset{\alpha_{|_{\tilde{\mathfrak{h}}} \neq 0}}{\alpha \in \Phi^{+}(\mathfrak{g}', \mathfrak{h}')}} (\xi^{\sharp}_{\frac{\alpha}{2}}(c^{\sharp}(h')) - \xi^{\sharp}_{-\frac{\alpha}{2}}(c^{\sharp}(h'))) = \prod\limits_{1 \leq i < j \leq n} 4 \left(\cfrac{h^{2}_{j} - h^{2}_{i}}{(1 + h^{2}_{i})(1 + h^{2}_{j})}\right)\prod\limits_{i=1}^n \prod\limits_{j = n+1}^m 4 \left(\cfrac{h^{2}_{j} - h^{2}_{i}}{(1 + h^{2}_{i})(1 + h^{2}_{j})}\right)\prod\limits_{k=1}^n \cfrac{4ih_{k}}{(1+h^{2}_{k})} \\
& = & \cfrac{\prod\limits_{1 \leq i < j \leq n} 4(h^{2}_{j} - h^{2}_{i})}{\prod\limits_{k=1}^n (1+h^{2}_{k})^{n-1}} \cfrac{\prod\limits_{i=1}^n \prod\limits_{j=n+1}^{m} 4(h^{2}_{j} - h^{2}_{i})}{\prod\limits_{i = 1}^n(1+h^{2}_{i})^{m-n} \prod\limits_{j = n+1}^{m}(1+h^{2}_{j})^{n}} \prod\limits_{k=1}^n \cfrac{4ih_{k}}{(1+h^{2}_{k})} = \cfrac{\prod\limits_{1 \leq i < j \leq n} 4(h^{2}_{j} - h^{2}_{i}) \prod\limits_{i=1}^n \prod\limits_{j=n+1}^{m} 4(h^{2}_{j} - h^{2}_{i}) \prod\limits_{k=1}^n 4ih_{k}}{\prod\limits_{k=1}^n (1+h^{2}_{k})^{m} \prod\limits_{k=n+1}^m (1+h^{2}_{k})^{n}} \\
& = & C' \cfrac{\pi_{\mathfrak{g}'/\mathfrak{z}'}(h')}{\prod\limits_{k=1}^n (1+h^{2}_{k})^{m} \prod\limits_{k=n+1}^m (1+h^{2}_{k})^{n}},
\end{eqnarray*}}
and then, the following equality holds:
\begin{equation}
\cfrac{\left|\prod\limits_{\alpha \in \Phi^{+}(\mathfrak{g}', \mathfrak{h}')} (\xi_{\frac{\alpha}{2}}(c^{\sharp}(h')) - \xi_{-\frac{\alpha}{2}}(c^{\sharp}(h')))\right|^{2}}{\prod\limits_{\underset{\alpha_{|_{\tilde{\mathfrak{h}}} \neq 0}}{\alpha \in \Phi^{+}(\mathfrak{g}', \mathfrak{h}')}} (\xi_{\frac{\alpha}{2}}(c^{\sharp}(h')) - \xi_{-\frac{\alpha}{2}}(c^{\sharp}(h')))} = C'' \cfrac{\pi_{\mathfrak{g}'/\mathfrak{h}'}(h') \pi_{\mathfrak{z}'/\mathfrak{h}'}(h')}{\prod\limits_{k=1}^n (1+h^{2}_{k})^{m} \prod\limits_{k = n+1}^m (1+h^{2}_{k})^{2m-n}}.
\end{equation}
Using the fact that $j_{\mathfrak{g}'}(h') = \prod\limits_{k=1}^{p+q} (1+h^{2}_{k})^{-(2m+1)}$ and $j_{\mathfrak{h}'}(h') = \prod\limits_{k=1}^{p+q} (1+h^{2}_{k})^{-1}$, we get:
\begin{equation}
\ch(\pr(h'))^{2m-2n} \pi_{\mathfrak{g}'/\mathfrak{h}'}(h') \pi_{\mathfrak{z}'/\mathfrak{h}'}(h') \Theta(\tilde{c}(h'))j_{\mathfrak{g}}(h') = \cfrac{\left|\prod\limits_{\alpha \in \Phi^{+}(\mathfrak{g}', \mathfrak{h}')} (\xi_{\frac{\alpha}{2}}(c^{\sharp}_{-}(h')) - \xi_{-\frac{\alpha}{2}}(c^{\sharp}_{-}(h')))\right|^{2}}{\prod\limits_{\underset{\alpha_{|_{\tilde{\mathfrak{h}}} \neq 0}}{\alpha \in \Phi^{+}(\mathfrak{g}', \mathfrak{h}')}} (\xi_{\frac{\alpha}{2}}(c^{\sharp}_{-}(h')) - \xi_{-\frac{\alpha}{2}}(c^{\sharp}_{-}(h')))} j_{\mathfrak{h}}(h').
\end{equation}

\noindent Then, the equation \eqref{int444} becomes:
\begin{eqnarray*}
& & T(\overline{\Theta_{\Pi}})(\phi) \\
&=&  C'_{1} \sum\limits_{\eta_{0} \in \mathscr{W}(G, \mathfrak{h}, \beta} \sgn(\eta_{0})\displaystyle\int_{\mathfrak{h}'} c_{-}(\pr(x))^{-\eta^{-1}_{0}\mu'}\ch(\pr(x))^{2m-2n}\pi_{\mathfrak{z}'/\mathfrak{h}'}(x) \pi_{\mathfrak{g}'/\mathfrak{h}'}(x) \Theta(\tilde{c}(x)) j_{\mathfrak{g}'}(x) \Psi^{\tilde{\G}'/\tilde{\H}'}(\tilde{c}_{-}(x)) dx \\
& = & C''_{1} \sum\limits_{\eta_{0} \in \mathscr{W}(\G, \mathfrak{h}, \beta)} \sgn(\eta_{0}) \displaystyle\int_{\mathfrak{h}'} c_{-}(\pr(x))^{-\eta^{-1}_{0}\mu'} \cfrac{\left| \prod\limits_{\alpha \in \Phi^{+}(\mathfrak{g}', \mathfrak{h}')} (c_{-}(x)^{\frac{\alpha}{2}} - c_{-}(x)^{-\frac{\alpha}{2}})\right|^{2}}{\prod\limits_{\underset{\alpha_{|_{\tilde{\mathfrak{h}}}} \neq 0}{\alpha \in \Phi^{+}(\mathfrak{g}', \mathfrak{h}')}} (c_{-}(x)^{\frac{\alpha}{2}} - c_{-}(x)^{-\frac{\alpha}{2}})} j_{\mathfrak{h}}(x) \Psi^{\tilde{\G}'/\tilde{\H}'}(\tilde{c}_{-}(x)) dx
\end{eqnarray*}

\noindent Finally,
\begin{equation}
T(\overline{\Theta_{\Pi}})(\phi) = \displaystyle\int_{\widetilde{\H}'} C''_{1} \sum\limits_{\eta_{0} \in \mathscr{W}(\G, \mathfrak{h}, \beta)} \cfrac{\sgn(\eta_{0}) \pr(h)^{-\eta^{-1}_{0}\mu'}}{\prod\limits_{\underset{\alpha_{|_{\tilde{\mathfrak{h}}}} \neq 0}{\alpha \in \Phi^{+}(\mathfrak{g}', \mathfrak{h}')}} (h^{\frac{\alpha}{2}} - h^{-\frac{\alpha}{2}})} \left|\prod\limits_{\alpha \in \Phi^{+}(\mathfrak{g}', \mathfrak{h}')} (h^{\frac{\alpha}{2}} - h^{-\frac{\alpha}{2}})\right|^{2} \Psi^{\tilde{\G}'/\tilde{\H}'}(h) dh
\label{intH'1}
\end{equation}
Using the invariance of $\left|\prod\limits_{\alpha \in \Phi^{+}(\mathfrak{g}', \mathfrak{h}')} (h^{\frac{\alpha}{2}} - h^{-\frac{\alpha}{2}})\right|^{2} \Psi^{\tilde{\G}'/\tilde{\H}'}(h)$ under the action of $\mathscr{W}(\K', \mathfrak{h}')$, we have:
\begin{equation*}
\left|\prod\limits_{\alpha \in \Phi^{+}(\mathfrak{g}', \mathfrak{h}')}(h^{\frac{\alpha}{2}} - h^{-\frac{\alpha}{2}})\right|^{2} \Psi^{\tilde{\G}'/\tilde{\H}'}(h) = \cfrac{1}{|\mathscr{W}(\K', \mathfrak{h}')|} \sum\limits_{\sigma \in \mathscr{W}(\K', \mathfrak{h}')} \left|\prod\limits_{\alpha \in \Phi^{+}(\mathfrak{g}', \mathfrak{h}')}((\sigma h)^{\frac{\alpha}{2}} - (\sigma h)^{-\frac{\alpha}{2}})\right|^{2} \Psi^{\tilde{\G}'/\tilde{\H}'}(\sigma h)    
\end{equation*}   
and
\begin{eqnarray*}
& &T(\overline{\Theta_{\Pi}})(\phi) \\
& = &  \cfrac{C''_{1}}{|\mathscr{W}(\K', \mathfrak{h}')|}\displaystyle\int_{\widetilde{\H}'} \sum\limits_{\eta_{0} \in \mathscr{W}(\G, \mathfrak{h}, \beta)} \sum\limits_{\sigma \in \mathscr{W}(\K', \mathfrak{h}')} \cfrac{\sgn(\eta_{0}) \pr(h)^{-\eta^{-1}_{0}\mu'}}{\prod\limits_{\underset{\alpha_{|_{\tilde{\mathfrak{h}}}} \neq 0}{\alpha \in \Phi^{+}(\mathfrak{g}', \mathfrak{h}')}} (h^{\frac{\alpha}{2}} - h^{-\frac{\alpha}{2}})} \left|\prod\limits_{\alpha \in \Phi^{+}(\mathfrak{g}', \mathfrak{h}')} ((\sigma h)^{\frac{\alpha}{2}} - (\sigma h)^{-\frac{\alpha}{2}})\right|^{2} \Psi^{\tilde{\G}'/\tilde{\H}'}(\sigma h) dh \\
                  & = & \cfrac{C''_{1}}{|\mathscr{W}(\K', \mathfrak{h}')|} \displaystyle\int_{\tilde{\H}'} \sum\limits_{\eta_{0} \in \mathscr{W}(\G, \mathfrak{h}, \beta)} \sum\limits_{\sigma \in \mathscr{W}(\K', \mathfrak{h}')} \cfrac{\sgn(\eta_{0}) \pr(\sigma^{-1}h)^{-\eta^{-1}_{0}\mu'}}{\prod\limits_{\underset{\alpha_{|_{\tilde{\mathfrak{h}}}} \neq 0}{\alpha \in \Phi^{+}(\mathfrak{g}', \mathfrak{h}')}} ((\sigma^{-1}h)^{\frac{\alpha}{2}} - (\sigma^{-1}h^{-\frac{\alpha}{2}})} \left|\prod\limits_{\alpha \in \Phi^{+}(\mathfrak{g}', \mathfrak{h}')} (h^{\frac{\alpha}{2}} - h^{-\frac{\alpha}{2}})\right|^{2} \Psi^{\tilde{\G}'/\tilde{\H}'}(h) dh 
\end{eqnarray*}
\noindent Finally, we obtain that the character $\Theta_{\Pi'}(h)$ is given, up to a constant, by the formula:
\begin{equation}
\sum\limits_{\eta \in \mathscr{W}(\G, \mathfrak{h}, \beta)} \sum\limits_{\sigma \in \mathscr{W}(\K', \mathfrak{h}')} \cfrac{\sgn(\eta) \pr(\sigma h)^{-\eta^{-1}\mu'}}{\prod\limits_{\underset{\alpha_{|_{\tilde{\mathfrak{h}}} \neq 0}}{\alpha \in \Phi^{+}(\mathfrak{g}', \mathfrak{h}')}}((\sigma h)^{\frac{\alpha}{2}} - (\sigma h)^{-\frac{\alpha}{2}})}.
\end{equation}
                        
\end{proof}

\section{The dual pair $(\G = \O(2n+1, \mathbb{R}), \G' = \Sp(2m, \mathbb{R}))$, $n \leq m$}

\label{SectionO(2n+1)}

Let $V_{\overline{0}}$ be a real vector space of dimension $2n+1$ endowed with a positive-definite symmetric bilinear form $b_{0}$. Let $\mathscr{B} = \{v_{1}, v'_{1}, \ldots, v_{n}, v'_{n}, v_{n+1}\}$ be a basis of $V_{\overline{0}}$ such that $\Mat(b_{0}, \mathscr{B}) = \Id_{2n+1}$ and let $\O(V_{\overline{0}}, b_{0})$ the group of isometries of the form $b_{0}$, i.e.
\begin{equation}
\O(V_{\overline{0}}, b_{0}) = \left\{g \in \GL(V_{\overline{0}}), b_{0}(gu, gv) = b_{0}(u, v), (\forall u, v \in V_{\overline{0}})\right\}.
\end{equation}
We denote by $\mathfrak{o}(V_{\overline{0}}, b_{0})$ the Lie algebra of $\O(V_{\overline{0}}, b_{0})$ given by:
\begin{equation}
\mathfrak{o}(V_{\overline{0}}, b_{0}) = \left\{X \in \End(V_{\overline{0}}), b_{0}(Xu, v) + b_{0}(u, Xv) = 0, (\forall u, v \in V_{\overline{0}})\right\}.
\end{equation}
Writing the endomorphisms $X$ in the basis $\mathscr{B}$, the Lie algebra can be realized as:
\begin{equation}
\mathfrak{o}(V_{\overline{0}}, b_{0}) = \left\{X \in \M(2n+1, \mathbb{R}), X = -X^{t}\right\}.
\end{equation}
In particular, $\mathfrak{o}(V_{\overline{0}}, b_{0})$ corresponds to the set of skew-symmetric matrices of $\M(2n+1, \mathbb{R})$ and we get the following decomposition:
\begin{equation}
\mathfrak{o}(V_{\overline{0}}, b_{0}) = \bigoplus\limits_{1 \leq i < j \leq 2n+1} \mathbb{R}(\E_{i, j} - \E_{j, i}) = \bigoplus\limits_{1 \leq i < j \leq 2n} \mathbb{R}(\E_{i, j} - \E_{j, i}) \oplus \bigoplus\limits_{k=1}^{2n}(\E_{k, 2n+1} - \E_{2n+1, k}).
\end{equation}
Let $\mathfrak{h}_{0}$ be the subalgebra of $\mathfrak{g}_{0}$ defined by:
\begin{equation}
\mathfrak{h}_{0} = \bigoplus\limits_{k = 1}^n \mathbb{R}(\E_{-1+2k, 2k} - \E_{2k, -1+2k}).
\end{equation}
To simplify the notation, let $H_{k}$ be the matrix defined by $\H_{k} = \E_{-1+2k, 2k} - \E_{2k, -1+2k}, 1 \leq k \leq n$. The complexifications of $\mathfrak{g}_{0}$ and $\mathfrak{h}_{0}$ are respectively denoted by $\mathfrak{g}$ and $\mathfrak{h}$. The roots of $\mathfrak{g}$ with respect to $\mathfrak{h}$ are given by:
\begin{equation}
\Phi(\mathfrak{g}, \mathfrak{h}) = \left\{\pm e_{i} \pm e_{j}, 1 \leq i \neq j \leq n\right\} \cup \left\{\pm e_{i}, 1 \leq i \leq n\right\}.
\end{equation}
where the linear forms $e_{a}$, $1 \leq a \leq n$, are defined in equation \eqref{Racines}.

\noindent Let $\pi_{\mathfrak{g}/\mathfrak{h}}$ be the product of positive roots. For all $h = \sum\limits_{k = 1}^{n} h_{k}\H_{k} \in \mathfrak{h}_{0}$, we have:
\begin{eqnarray*}
\pi_{\mathfrak{g}/\mathfrak{h}}(h) & = & \prod\limits_{\alpha \in \Phi^{+}(\mathfrak{g}, \mathfrak{h})} \alpha(h) = \prod\limits_{1 \leq i < j \leq n} (e_{i} - e_{j})(h) \prod\limits_{1 \leq i < j \leq n} (e_{i} + e_{j})(h) \prod\limits_{k=1}^n e_{k}(h) \\
                                                     & = & \prod\limits_{1 \leq i < j \leq n} (-ih_{i} + ih_{j}) \prod\limits_{1 \leq i < j \leq n}(-ih_{i} - ih_{j}) \prod\limits_{k=1}^n (-ih_{k})= (-1)^n \prod\limits_{1 \leq i < j \leq n} (-h^{2}_{i} + h^{2}_{j}) \prod\limits_{k=1}^n ih_{k}
\end{eqnarray*}

\noindent Let $V = V_{\overline{0}} \oplus V_{\overline{1}}$, $b = b_{0} \oplus b_{1}$ as defined in Appendix  \ref{AppendixA} and let $(\S, \mathfrak{s}(V, b) = \mathfrak{s}(V, b)_{\overline{0}} \oplus \mathfrak{s}(V, b)_{\overline{1}})$ be the corresponding Lie supergroup. 

\begin{lemme}

An element $X = \begin{pmatrix} 0 & X_{1} \\ X_{2} & 0 \end{pmatrix}$ is in $\mathfrak{s}(V, \tau)_{\overline{1}}$ if and only if $X_{2} = -\J_{m, m}X^{t}_{1}$.

\end{lemme}

\noindent For all $j$, we denote by $V^{j}_{\overline{0}}$ and $V^{j}_{\overline{1}}$ the subspaces of $V_{\overline{0}}$ and $V_{\overline{1}}$ given by:
\begin{equation}
V^{j}_{\overline{0}} = \mathbb{R} v_{j} \oplus \mathbb{R} v'_{j}, \qquad V^{n+1}_{\overline{0}} = \mathbb{R}v_{n+1}, \qquad V^{j}_{\overline{1}} = \mathbb{R} w_{j} \oplus \mathbb{R} w'_{j}.
\end{equation}     

\noindent We now consider the elements $u_{j} \in \mathfrak{s}(V)_{\overline{1}}, 1 \leq j \leq n$, defined on $V_{\overline{0}}$ (resp. $V_{\overline{1}}$) by
\begin{equation}
u_{j}(v_{j}) = \frac{1}{\sqrt{2}}(w_{j} - w'_{j}), \qquad u_{j}(v'_{j}) = \frac{1}{\sqrt{2}}(w_{j} + w'_{j}), \qquad u_{j}(v_{k}) = u_{j}(v'_{k}) = 0, k \neq j,
\end{equation}
\begin{equation}
u_{j}(w_{j}) = \frac{1}{\sqrt{2}}(v_{j} - v'_{j}), \qquad u_{j}(w'_{j}) = \frac{1}{\sqrt{2}}(v_{j} + v'_{j}), \qquad u_{j}(w_{k}) = u_{j}(w'_{k}) = 0, k \neq j,
\end{equation}
and let $\mathfrak{h}_{1}$ be the subspace of $\mathfrak{s}(V, b)_{\overline{1}}$ given by
\begin{equation}
\mathfrak{h}_{1} = \sum\limits_{j = 1}^{n} \mathbb{R}u_{j}.
\end{equation}
We define the moment maps $\tau: \mathfrak{s}(V, b)_{\overline{1}}  \mapsto \mathfrak{o}(2n+1, \mathbb{R})$ and $\tau': \mathfrak{s}(V, b)_{\overline{1}} \mapsto \mathfrak{sp}(2m, \mathbb{R})$ by
\begin{equation}
\tau(w) = w^{2}_{|_{V_{\overline{0}}}}, \qquad \tau'(w) = w^{2}_{|_{V_{\overline{1}}}}.
\end{equation}
More precisely, for all $w = \begin{pmatrix} 0 & X_{1} \\ \J^{-1}_{m, m}X^{t}_{1} & 0 \end{pmatrix}$, we have:
\begin{equation}
\tau\begin{pmatrix} 0 & X_{1} \\ \J^{-1}_{m, m}X^{t}_{1} & 0 \end{pmatrix} = X_{1}\J^{-1}_{m, m}X^{t}_{1}, \qquad \tau\begin{pmatrix} 0 & X_{1} \\ \J^{-1}_{m, m}X^{t}_{1} & 0 \end{pmatrix} = \J^{-1}_{m, m}X^{t}_{1}X_{1}.
\end{equation}
\noindent We get an injection of $\mathfrak{h}$ into $\mathfrak{h}'$ given by:
\begin{equation}
\mathfrak{h} \ni \sum\limits_{k=1}^n h_{k}\H_{k} \mapsto \sum\limits_{k=1}^n h_{k}\H_{k} \in \mathfrak{h}'.
\end{equation}
\noindent We denote by $\tilde{\mathfrak{h}}$ the image of $\mathfrak{h}$ into $\mathfrak{h}'$. We define $\mathfrak{z}'$ and $\Z'$ by
\begin{equation}
\mathfrak{z}' = \{x \in \mathfrak{g}', [x, y] = 0, (\forall y \in \tilde{\mathfrak{h}})\},
\end{equation}
\begin{equation}
\Z' = \{\Ad(g)y = y, (\forall y' \in \tilde{\mathfrak{h}})\}.
\end{equation}
We define $\Phi(\mathfrak{g}', \mathfrak{z}') = \{\alpha \in \Phi(\mathfrak{g}', \mathfrak{h}'), \alpha_{|_{\tilde{\mathfrak{h}}}} \neq 0\}$ and $\Phi(\mathfrak{z}', \mathfrak{h}') = \Phi(\mathfrak{g}', \mathfrak{h}') \setminus \Phi(\mathfrak{g}', \mathfrak{z}')$, and let:
\begin{equation}
\pi_{\mathfrak{g}'/\mathfrak{z}'} = \prod\limits_{\alpha \in \Phi^{+}(\mathfrak{g}', \mathfrak{z}')} \alpha \qquad \pi_{\mathfrak{g}'/\mathfrak{z}'} = \prod\limits_{\alpha \in \Phi^{+}(\mathfrak{z}', \mathfrak{h}')} \alpha.
\end{equation}

\noindent For every $\phi \in \S(W)$, we define the function $f_{\phi}$ on $\tau(\mathfrak{h}^{\reg}_{1})$ by 
\begin{equation}
f_{\phi}(\tau(w)) := C_{\mathfrak{h}_{1}} \pi_{\mathfrak{g}'/\mathfrak{z}'}(\tau'(w)) \displaystyle\int_{\G'/\Z'} \phi(g'w)d(g'\Z').
\end{equation}
where $C_{\mathfrak{h}_{1}}$ is a constant of modulus 1 defined in \cite[Lemma~8,~page~17]{TOM4}.

\noindent As in the previous section, we denote by $\H^{\sharp}$ the two fold cover of $\widetilde{\H}$ such that the linear forms $\frac{\alpha}{2}$ are analytic integral for every roots $\alpha \in \Phi(\mathfrak{g}, \mathfrak{h})$ and let $\H^{\sharp}_{0}$ be its connected identity component. Let $\xi^{\sharp}_{\frac{\alpha}{2}}: \H^{\sharp}_{0} \to \S^{1}$ be the multiplicative character having the linear form $\frac{\alpha}{2}$ as derivative and let $c^{\sharp}_{-}: \mathfrak{h} \to \H^{\sharp}_{0}$ be the extension of $c_{-}$ on $ \H^{\sharp}_{0}$ (section \ref{Section5.1})

\begin{prop}

Fix $\Pi \in \mathscr{R}(\tilde{\G}, \omega^{\infty})$ of highest weight $\nu$ and let $\mu = \nu + \rho$. For every $\phi \in \S(W)$, the following formula holds:
\begin{equation}
T(\overline{\Theta_{\Pi}})(\phi) = \displaystyle\int_{\mathfrak{h}}c_{-}(x)^{\mu'}\ch(x)^{2m-2n-1} \displaystyle\int_{\mathfrak{h} \cap \tau(W)} e^{i B(x, y)}f_{\phi}(y)dydx
\label{Proposition11}
\end{equation}
where $\ch(x) = |\det(x-1)|^{\frac{1}{2}}$ and $\mu' \in \mathfrak{h}^{*}$ is defined by
\begin{equation}
\mu'_{j} = \mu_{n+1-j} \qquad (1 \leq j \leq n).
\end{equation}

\end{prop}

\begin{proof}

For a proof of this proposition, we refer to \cite[Corollary~38,~page~47]{TOM4}.

\end{proof}

\begin{rema}

\begin{enumerate}
\item By \cite{VERGNE2}, the weights $\nu$ and $\nu'$ of the representations $\Pi$ and $\Pi'$ respectively are given by:
\begin{equation}
\nu = \sum\limits_{a=1}^k \nu_{a}e_{a} \mapsto \nu' = -\sum\limits_{a=1}^m \frac{2n+1}{2}e_{a} - \sum\limits_{a = 1}^k \nu_{n+1-a}e_{a},
\end{equation}
where $0 \leq k \leq n$ and $\nu_{1} \geq \ldots \geq \nu_{k} > 0$. By fixing $k = m$ and considering a decreasing sequence $\nu_{1} \geq \ldots \geq \nu_{m} \geq 0$ with at most $n$ non zero $\nu_{i}$, we get
\begin{equation}
\nu = \sum\limits_{a=1}^n \nu_{a}e_{a} \mapsto \lambda' = -\sum\limits_{a=1}^m \frac{2n+1}{2}e_{a} - \sum\limits_{a = 1}^m \nu_{m+1-a}e_{a} = \sum\limits_{a = 1}^m \tau_{a}e_{a},
\end{equation}
with $\tau_{a} = -\frac{2n+1}{2} - \nu_{m+1-a}$, where $(\tau_{a})_{a \in [|1, m|]}$ is a decreasing sequence of negative numbers.

\noindent The linear form $\rho$ is given by:
\begin{eqnarray*}
\rho & = & \frac{1}{2} \sum\limits_{\alpha \in \Phi^{+}(\mathfrak{g}, \mathfrak{h})} \alpha = \frac{1}{2} \sum\limits_{1 \leq i < j \leq n} (e_{i} - e_{j}) + \frac{1}{2} \sum\limits_{1 \leq i < j \leq n} (e_{i} + e_{j}) + \frac{1}{2} \sum\limits_{k=1}^n e_{k} \\
       & = & \frac{1}{2} \sum\limits_{1 \leq i < j \leq n} 2e_{i} + \frac{1}{2} \sum\limits_{k=1}^n e_{k} = \sum\limits_{k=1}^n \left(n-k+\frac{1}{2}\right)e_{k}
\end{eqnarray*}
Then $\mu = \sum\limits_{k = 1}^n (\nu+\rho)_{a}e_{a}$ with $(\nu+\rho)_{a} = \nu_{a} + n - a + \frac{1}{2}$ and 
\begin{equation}
\mu'_{a} = \mu_{n+1-a} = \nu_{n+1-a} + a -\frac{1}{2}.
\end{equation}
\item For every $x, y \in \mathfrak{g}$ or $\mathfrak{g}'$, we denote by $B$ the bilinear form defined by
\begin{equation}
B(x, y) = \Re\tr(xy)
\label{FormeB1}
\end{equation}
The form $B$ is $\G$ (resp. $\G'$)-invariant and non-degenerate on $\mathfrak{g}$ and $\mathfrak{g}'$. More precisely, for all $x = \sum\limits_{k=1}^n x_{k}\H_{k, k}$ (resp. $x' = \sum\limits_{k=1}^{m} x'_{k}\H_{k, k}$) and $y = \sum\limits_{k=1}^n y_{k}\H_{k, k}$ (resp. $y' = \sum\limits_{k=1}^{m} y'_{k}\H_{k, k}$), the form $B$ is given by
\begin{equation}
B(x, y) = \sum\limits_{k=1}^n -\pi x_{k} y_{k} \qquad \left(\text{resp. } B(x', y') = \sum\limits_{k=1}^m -\pi x'_{k} y'_{k}\right).
\end{equation}
See \cite[page~49-50]{TOM4}.
\end{enumerate}

\end{rema}

\begin{theo}

For all regular element $y \in \tilde{\mathfrak{h}}$ and $y' \in \mathfrak{h}'^{\reg}$, we get:
\begin{equation}
\pi_{\mathfrak{g}'/\mathfrak{z}'}(y) \displaystyle\int_{\G'/\Z'} e^{iB(y', g'y)}d(g'\Z') = (-i)^{\frac{1}{2}\dim(\mathfrak{g}'/\mathfrak{z}')}(-1)^{n(\mathfrak{h'})} \sum\limits_{\sigma \in \mathscr{W}(K', \mathfrak{h}')/\mathscr{W}(K', \mathfrak{h}')^{\tilde{\mathfrak{h}}}} \cfrac{e^{iB(y', \sigma y)}}{\pi_{\mathfrak{g}'/\mathfrak{z}'}(\sigma^{-1}y')}
\label{Theoreme1}
\end{equation}
where $n(\tilde{\mathfrak{h}})$ is the number of non-compacts positives roots which do not vanish on $\tilde{\mathfrak{h}}'$.

\end{theo}

\noindent It is clear that for all $\psi \in \mathscr{C}^{\infty}_{c}(\G'.\mathfrak{h}'^{\reg})$, 
\begin{equation}
\phi: W \ni w \mapsto \displaystyle\int_{\mathfrak{g}'} \chi_{x'}(w) \psi(x')dx' \in \mathbb{C}
\end{equation}
is a Schwartz function on $W$.
Using the Weyl integration formula, we get for all $y \in \tau(\mathfrak{h}_{1})$
\begin{eqnarray*}
& & \pi_{\mathfrak{g}'/\mathfrak{z}'}(y) \displaystyle\int_{\G'/\Z'} \displaystyle\int_{\mathfrak{g}'} e^{i B(y', g'y)} \psi(y')dy'd(g'\Z') \\
& = & \cfrac{1}{|\mathscr{W}(\K', \mathfrak{h}')|} \displaystyle\int_{\mathfrak{h}'} \left( \pi_{\mathfrak{g}'/\mathfrak{z}'}(y) \displaystyle\int_{\G'/\Z'} e^{i B(y', g'y)} d(g'\Z') \right) \psi^{\tilde{\G}'/\tilde{\H}'}(y') |\pi_{\mathfrak{g}'/\mathfrak{h}'}(y')|^{2}dy' 
\end{eqnarray*}
where $\psi^{\tilde{\G}'/\tilde{\H}'} : \mathfrak{h}'^{\reg} \to \mathbb{C}$ is given by:
\begin{equation}
\psi^{\tilde{\G'}/\tilde{\H'}}(y') = \displaystyle\int_{\tilde{\G'}/\tilde{\H'}} \psi(\Ad(\tilde{g}')(y')) d(\tilde{g'}\tilde{\H'}).
\label{psiG'H'}
\end{equation}
Using the same method than before, we get the following equality:
\begin{equation}
f_{\phi}(y) = C \displaystyle\int_{\mathfrak{h}'}e^{i B(y', y)} \pi_{\mathfrak{z}'/\mathfrak{h}'}(y') \overline{\pi_{\mathfrak{g}'/\mathfrak{h}'}(y')} \psi^{\tilde{\G}'/\tilde{\H}'}(y')dy'.
\end{equation}
We denote by $\mathscr{A}_{\psi}$ the function defined on $\mathfrak{h}'^{\reg}$ by
\begin{equation}
\mathscr{A}_{\psi}(y') = \pi_{\mathfrak{g}'/\mathfrak{h}'}(y') \psi^{\tilde{\G}'/\tilde{\H}'}(y') \qquad (y' \in \mathfrak{h}'^{\reg}).
\end{equation}
Then, we get:
\begin{equation}
f_{\phi} = \mathscr{F}_{\mathfrak{h}'}(C\pi_{\mathfrak{z}'/\mathfrak{h}'} \mathscr{A}_{\psi})_{|_{\tilde{\mathfrak{h}}}}.
\end{equation}

\noindent Before giving the main theorem, we prove a result concerning the right hand-side of \eqref{Proposition11}. More particularly, we are interested in the support of the following distribution: 
\begin{equation}
\displaystyle\int_{\mathfrak{h}} c_{-}(x)^{-\mu'} \ch(x)^{2m-2n-1}e^{i B(x, y)} dx.
\label{DistributionSupport111908}
\end{equation}

\begin{prop}

Let $\beta_{1}$ and $\beta_{2}$ the two integers defined by:
\begin{equation}
\beta_{1} = \begin{cases} \max\left\{k \geq 1, \mu'_{k} \leq -(m-n+\frac{3}{2})\right\} & \text{ if } \mu'_{1} \leq -(m-n+\frac{3}{2}) \\ 0 & \text{ otherwise } \end{cases}
\end{equation}
\begin{equation}
\beta_{2} = \begin{cases} \min\left\{ k \geq 1, \mu'_{k} \geq m-n+\frac{3}{2}\right\} - 1 & \text{ if } \mu'_{n} \geq m - n + \frac{3}{2} \\ n & \text{ otherwise } \end{cases}
\end{equation}
Then,
\begin{equation}
\supp \displaystyle\int_{\mathfrak{h}} c_{-}(x)^{-\mu'} \ch(x)^{2m-2n-1}e^{i B(x, y)} dx = \left\{\sum\limits_{k=1}^{\beta_{1}} h_{k}\H_{k} + \sum\limits_{k=\beta_{2}+1}^{n} y_{k}\H_{k}, h_{i} > 0, y_{i} < 0 \right\}.
\end{equation}

\end{prop}

\begin{proof}

According to \cite[Lemma~33,~page~43]{TOM4}, we get:
\begin{equation}
\displaystyle\int_{\mathfrak{h}} c_{-}(x)^{-\mu'} \ch(x)^{2m-2n-1}e^{iB(x, y)} dx = \prod\limits_{k=1}^n \displaystyle\int_{\mathbb{R}} (1+ix_{k})^{-a_{k}}(1-ix_{k})^{-b_{k}} e^{-i\pi x_{k}y_{k}}dx_{k},
\end{equation}
where $a_{k} = \mu'_{k} - m + n + \frac{1}{2}$ and $b_{k} = -\mu'_{k} - m + n +\frac{1}{2}$. In particular, $a_{k} \geq 1$ (resp. $b_{k} \geq 1$) implies $b_{k} \leq 0$ (resp. $a_{k} \leq 0$).

\noindent Using \cite[Appendix~C]{TOM4}, we get the following equality
\begin{equation}
\displaystyle\int_{\mathbb{R}} (1+ix_{k})^{-a_{k}}(1-ix_{k})^{-b_{k}} e^{-i\pi x_{k}y_{k}}dx_{k} = P_{a_{k}, b_{k}}(\pi y_{k})e^{-|\pi y_{k}|} + Q_{a_{k}, b_{k}}\left(\cfrac{\partial}{\partial y_{k}}\right)\delta_{0}(\pi y_{k}),
\end{equation}
where $P_{a_{k}, b_{k}}$ and $Q_{a_{k}, b_{k}}$ are the polynomials defined in \cite[Appendix~C]{TOM4}. In particular, using \cite[Appendix~C]{TOM4}, the support of the polynomial $P_{a_{k}, b_{k}}$ are
\begin{equation}
\supp P_{a_{k}, b_{k}} = \begin{cases} ]-\infty, 0] & \text{ if } a_{k} \geq 1 \text{ and } a_{k} \leq 1\\ [0, +\infty[ & \text{ if } a_{k} \leq 0 \text{ and } b_{k} \geq 1 \\ \{\emptyset\} & \text{ otherwise } \end{cases} = \begin{cases} ]-\infty, 0] & \text{ if } \mu'_{k} \geq m-n+\frac{3}{2} \\ [0, +\infty[ & \text{ if } \mu'_{k} \leq -(m-n+\frac{3}{2}) \\ \{\emptyset\} & \text{ otherwise } \end{cases}.
\end{equation}
Then, 
\begin{equation}
\supp \displaystyle\int_{\mathbb{R}} (1+ix_{k})^{-a_{k}}(1-ix_{k})^{-b_{k}} e^{-i\pi x_{k}y_{k}} dx_{k} = \begin{cases} [0, +\infty[ & \text{ if } k \leq \beta_{1} \\ \{0\} & \text{ if } k \in [|\beta_{1}+1, \beta_{2}|] \\ ]-\infty, 0] & \text{ if } k \geq \beta_{2}+1 \end{cases}
\end{equation}
and finally
\begin{equation}
\supp \displaystyle\int_{\mathfrak{h}} c_{-}(x)^{-\mu'} \ch(x)^{2m-2n-1}e^{iB(x, y)} dx = \left\{\sum\limits_{k = 1}^{\beta_{1}} h_{k}\H_{k} + \sum \limits_{k = \beta_{2}+1}^{n} y_{k}\H_{k}, h_{k} > 0, y_{k} < 0\right\} \subseteq \tau(\mathfrak{h}_{1}).
\end{equation}

\end{proof}

\noindent From now on, we denote by $\supp(\beta_{1}, \beta_{2})$ the support of the distribution given in equation \eqref{DistributionSupport111908}. We denote by $\tilde{\mathfrak{h}}^{\perp}$ the orthogonal complement of $\tilde{\mathfrak{h}}$ with respect to the bilinear form $B$, and by $\pr: \mathfrak{h} \to \tilde{\mathfrak{h}}$ the associated projection, i.e.
\begin{equation}
\pr(y_{1} + y_{2}) = y_{1} \qquad (y_{1} \in \tilde{\mathfrak{h}}, y_{2} \in \tilde{\mathfrak{h}}^{\perp}).
\label{label projection 1}
\end{equation}
Finally,, let us consider the subgroup $\mathscr{W}(\G, \mathfrak{h}, \beta)$ of $\mathscr{W}(\G, \mathfrak{h})$ defined by:
\begin{equation}
\mathscr{W}(\G, \mathfrak{h}, \beta) = \{\eta \in \mathscr{W}(\G, \mathfrak{h}), \supp(\beta_{1}, \beta_{2}) \subseteq \eta \tau(\mathfrak{h}_{1})\}.
\end{equation}

\begin{theo}
The character $\Theta_{\Pi'}$ of $\Pi' \in \mathscr{R}(\tilde{\G}', \omega^{\infty})$ is given, up to a constant, by the following formula:
\begin{equation}
\Theta_{\Pi'}(h) = \sum\limits_{\eta \in \mathscr{W}(\G, \mathfrak{h}, \beta)} \sum\limits_{\sigma \in \mathscr{W}(\K', \mathfrak{h}')} \cfrac{\sgn(\eta) \pr(\sigma h)^{-\eta^{-1}\mu'}}{\prod\limits_{\underset{\alpha_{|_{\tilde{\mathfrak{h}}} \neq 0}}{\alpha \in \Phi^{+}(\mathfrak{g}', \mathfrak{h}')}}((\sigma h)^{\frac{\alpha}{2}} - (\sigma h)^{-\frac{\alpha}{2}})}.
\label{Formule finale caractere 111}
\end{equation}

\end{theo}

\begin{proof}

The proof of this theorem is similar to what we did in section \ref{O(2n,R)} for $\O(2n, \mathbb{R})$.

\end{proof}

\section{The dual pair $(\G = \U(n, \mathbb{H}), \G' = \O^{*}(m, \mathbb{H}))$, $n \leq m$}

\label{SectionU(n,H)}

Let $\mathbb{H}$ be the fields of quaternions, i.e. $\mathbb{H} = \mathbb{R} + i\mathbb{R} + j\mathbb{R} + ij \mathbb{R}$ with $ij = -ji$. Let $z = a + ib + jc + ijd = (a + ib) + j(c - id) \in \mathbb{H}$. We have the following morphisms:
\begin{eqnarray*}
\Psi_{1}: \mathbb{H} & \mapsto & \M(2, \mathbb{C}) \\
(a + ib) + j(c - id) & \mapsto & \begin{pmatrix} a + ib & -c + id \\ c + id & a - ib \end{pmatrix}
\end{eqnarray*}
and $\Psi^{2}: \M(2, \mathbb{C}) \to \M(4, \mathbb{R})$ defined in equation \eqref{EmbeddingComplex}. In particular, we associate to a quaternionic number $z = (a + ib) + j(c - id)$ a $4 \times 4$-real matrix $\Psi^{2} \circ \Psi_{1}(z)$ given by:
\begin{equation}
\Psi^{2} \circ \Psi_{1}(z) = \begin{pmatrix} a & -c & -b & -d \\ c & a & -d & -b \\ b & d & a & -c \\ d & b & c & a \end{pmatrix}.
\end{equation}
More generally, we denote by $\Psi_{d}$ the corresponding morphism from $\M(d, \mathbb{H})$ into $\M(2d, \mathbb{C})$ and then we get a morphism $\Psi^{d} \circ \Psi_{d}: \M(d, \mathbb{H}) \to \M(4d, \mathbb{R})$. We view $\mathbb{H}^{k}$ as a left $\mathbb{H}$-vector space
\begin{equation}
z.v = v \bar{z} \qquad (z \in \mathbb{H}, v \in \mathbb{H}^{k}),
\end{equation}
where $\bar{z} = a - ib - jc - ijd$ is the non-trivial conjugation on $\mathbb{H}$. 

\noindent Let $V_{\overline{0}}$ be a left $n$-dimensional vector space over $\mathbb{H}$ endowed with a positive definite hermitian form $b_{0}$. Let $\mathscr{B} = \{v_{1}, \ldots, v_{n}\}$ be a basis of $V_{\overline{0}}$ such that the matrix of $b_{0}$ with respect to $\mathscr{B}$ is $\Mat(b_{0}, \mathscr{B}) = \Id_{n}$ and we denote by $\U(V_{\overline{0}}, b_{0})$ the group of isometries of $b_{0}$, i.e.
\begin{equation}
\U(V_{\overline{0}}, b_{0}) = \left\{g  \in \GL(n, \mathbb{H}), b_{0}(g(u), g(v)) = b_{0}(u, v), (\forall u, v \in V_{\overline{0}})\right\}.
\end{equation}
Writing the endomorphisms $g$ in the basis $\mathscr{B}$, the group $\U(V_{\overline{0}}, b_{0})$ can be realized as:
\begin{equation}
\left\{g \in \GL(n, \mathbb{H}), g^{*}g = \Id_{n}\right\},
\label{U(n,H)}
\end{equation}
where $g^{*} = \bar{g}^{t}$. From now on, we denote by $\U(n, \mathbb{H})$ the group defined in \eqref{U(n,H)}.

\begin{rema}

As mentioned in \cite[Chapter~1,~Section~17]{KNA}, the quaternionic unitary group $\U(n, \mathbb{H})$ is isomorphic to $\Sp(2n, \mathbb{C}) \cap \U(2n, \mathbb{C})$. In particular, the complexification of the Lie algebra of $\U(n, \mathbb{H})$ is isomorphic to $\mathfrak{sp}(2n, \mathbb{C})$.

\end{rema} 

\noindent We denote by $\mathfrak{g}_{0} = \mathfrak{u}(n, \mathbb{H})$ the Lie Algebra of $\U(n, \mathbb{H})$ and by $\mathfrak{g}$ its complexification. Let $\mathfrak{h}_{0}$ the subalgebra of $\mathfrak{g}_{0}$ given by:
\begin{equation}
\mathfrak{h}_{0} = \bigoplus\limits_{k=1}^{n} \mathbb{R} i\E_{k, k}.
\end{equation}
The corresponding subspace in $\M(2n, \mathbb{C})$ is given by:
\begin{equation}
\Psi_{n}(\mathfrak{h}_{0}) = \bigoplus_{k=1}^n \mathbb{R} i(\E_{k, k} - \E_{n+k, n+k}) = \bigoplus_{k=1}^n \mathbb{R} i\H_{k}.
\end{equation}
We denote by $\mathfrak{h}$ (resp. $\Psi_{n}(\mathfrak{h})$) the complexification of $\mathfrak{h}_{0}$ (resp. $\Psi_{n}(\mathfrak{h}_{0})$).
As recalled in \cite{KNA}, the roots are given by
\begin{equation}
\Phi(\Psi_{n}(\mathfrak{g}), \Psi_{n}(\mathfrak{h})) = \left\{\pm e_{i} \pm e_{j}, 1 \leq i < j \leq n\right\} \cup \left\{\pm 2e_{k}, 1 \leq k \leq n\right\},
\end{equation}
where the linear form $e_{a}$ is given by
\begin{equation}
e_{a}\left(\sum\limits_{k=1}^n ih_{k}\H_{k}\right) = h_{a}.
\end{equation}
Without considering the embedding of $\mathfrak{u}(n, \mathbb{H})$ on $\M(2n, \mathbb{C})$, one can define the forms $\widetilde{e_{a}}$ as
\begin{equation}
\widetilde{e_{a}}\left(\sum\limits_{k=1}^n ih_{k}\E_{k, k}\right) = h_{a},
\end{equation}
and get the same root system. Again, we define $\pi_{\mathfrak{g}/\mathfrak{h}} = \prod\limits_{\alpha \in \Phi^{+}(\mathfrak{g}, \mathfrak{h})} \alpha$.

\begin{lemme}

For all $h = \sum\limits_{k = 1}^n ih_{k}E_{k, k}$, we have:
\begin{equation}
\pi_{\mathfrak{g}/\mathfrak{h}}(h) = \prod\limits_{1 \leq i < j \leq n} (h^{2}_{i} - h^{2}_{j}) \prod\limits_{k = 1}^n 2h_{k}.
\end{equation}

\end{lemme}

\noindent Now, let $V_{\overline{1}}$ be a quaternionic vector space of dimension $m$ endowed with a non-degenerate skew-hermitian form $b_{1}$. Let $\mathscr{B}' = \{w_{1}, \ldots, w_{m}\}$ be a basis of $V_{\overline{1}}$ such that $\Mat(b_{1}, \mathscr{B}') = i\Id_{m}$. We denote by $\O^{*}(V_{\overline{1}}, b_{1})$ the group of isometries of the form $b_{1}$, i.e.
\begin{equation}
\O^{*}(V_{\overline{1}}, b_{1}) = \left\{g \in \GL(V_{\overline{1}}), b_{1}(g(u), g(v)) = b_{1}(u, v), \thinspace \forall u, v \in V_{\overline{1}}\right\}.
\end{equation}
Writing the endomorphisms in the basis $\mathscr{B}'$, we get the following realization
\begin{equation}
\O^{*}(V_{\overline{1}}, b_{1}) = \left\{g \in \GL(m, \mathbb{H}), g^{*}i\Id_{m}g = i\Id_{m}\right\}.
\label{OrthoQuater}
\end{equation}
We denote by $\O^{*}(m, \mathbb{H})$ the subgroup of $\GL(m, \mathbb{H})$ defined in equation \eqref{OrthoQuater}.

\begin{rema}

Using the embedding $\Psi_{m}$, the group $\O^{*}(m, \mathbb{H})$ can be seen as a subgroup of $\GL(2m, \mathbb{C})$. More precisely, one can check that $\O^{*}(m, \mathbb{H})$ is isomorphic to
\begin{equation}
\left\{g \in \U(m, m, \mathbb{C}), g^{t}M_{m, m}g = M_{m, m}\right\},
\label{O*}
\end{equation}
where $M_{m, m} = \begin{pmatrix} 0 & \Id_{m} \\ \Id_{m} & 0 \end{pmatrix}$. Until the end, we will denote by $\O^{*}(2m, \mathbb{C})$ the subgroup defined in equation \eqref{O*}.

\end{rema}

\noindent Let $\mathfrak{g}'_{0} = \mathfrak{o}^{*}(m, \mathbb{H})$ be the Lie algebra of $\O^{*}(m, \mathbb{H})$ and let $\mathfrak{h}'_{0}$ be the subalgebra of $\mathfrak{g}'_{0}$ given by
\begin{equation}
\mathfrak{h}'_{0} = \bigoplus\limits_{k = 1}^m \mathbb{R} i\E_{k, k}.
\end{equation}

\noindent The complexification of $\mathfrak{g}'_{0}$, denoted by $\mathfrak{g}'$, is isomorphic to $\mathfrak{o}(2m, \mathbb{C})$. In particular, as recalled in \cite{KNA}, the roots of $\mathfrak{g}'$ are given by
\begin{equation}
\mathfrak{g}' = \left\{\pm e_{i} \pm e_{j}, 1 \leq i \neq j \leq m\right\}
\end{equation}
where the linear form $e_{a}, 1 \leq a \leq m$ is given by 
\begin{equation}
e_{a}\left(\sum\limits_{k=1}^m ih_{k}\E_{k, k}\right) = h_{a}.
\end{equation}
As before, let $\pi_{\mathfrak{g}'/\mathfrak{h}'}$ be the product of all positive roots.

\begin{lemme}

For all $h' = \sum\limits_{k=1}^m ih_{k}\E_{k, k}$, we have:
\begin{equation}
\pi_{\mathfrak{g}'/\mathfrak{h}'}(h') = \prod\limits_{1 \leq i < j \leq m} (-h^{2}_{i} + h^{2}_{j}).
\end{equation}

\end{lemme}

\begin{rema}

The maximal compact subgroup of $\O^{*}(m, \mathbb{H})$ is $\K = \U(m, \mathbb{H})$. Let $\mathfrak{k}_{0}$ be the Lie algebra of $\K$ and $\mathfrak{k}$ its complexification. According to the previous section, the compact roots are given by
\begin{equation}
\Phi(\mathfrak{k}, \mathfrak{h}) = \left\{\pm(e_{i} - e_{j}), 1 \leq i < j \leq m\right\}.
\end{equation}

\end{rema}

\noindent Let $V = V_{\overline{0}} \oplus V_{\overline{1}}$, $b = b_{0} \oplus b_{1}$ as defined in Appendix  \ref{AppendixA} and let $(\S, \mathfrak{s}(V, b) = \mathfrak{s}(V, b)_{\overline{0}} \oplus \mathfrak{s}(V, b)_{\overline{1}})$ be the corresponding Lie supergroup.

\begin{lemme}

An element $X = \begin{pmatrix} 0 & X_{1} \\ X_{2} & 0 \end{pmatrix}$ is in $\mathfrak{s}(V, b)_{\overline{1}}$ if and only if $X_{2} = -iX^{*}_{1}$.

\end{lemme}




\noindent We consider the decompositions of $V_{\overline{0}}$ and $V_{\overline{1}}$ given by
\begin{equation}
V_{\overline{0}} = \bigoplus_{i = 1}^n V^{i}_{\overline{0}} =  \bigoplus_{i = 1}^n \mathbb{H}v_{i} \qquad V_{\overline{1}} = \bigoplus_{i = 1}^m V^{i}_{\overline{1}} =  \bigoplus_{i = 1}^m \mathbb{H}w_{i}.
\end{equation}
We define the odd endomorphisms $u_{j} \in \mathfrak{s}(V, b)_{\overline{1}}, 1 \leq j \leq n$, by
\begin{equation}
u_{j}(v_{k}) = \delta_{j, k}e^{-\frac{i\pi}{4}}w_{k}, \qquad u_{j}(w_{k}) = \delta_{j, k}e^{-\frac{i\pi}{4}}v_{k},
\end{equation}
and let $\mathfrak{h}_{1}$ be the subspace of $\mathfrak{s}(V, b)_{\overline{1}}$ given by
\begin{equation}
\mathfrak{h}_{1} = \sum\limits_{j=1}^n \mathbb{R}u_{j}.
\end{equation}
We define the moment maps $\tau: \mathfrak{s}(V, b)_{\overline{1}} \mapsto \mathfrak{g}$ and $\tau': \mathfrak{s}(V, b)_{\overline{1}} \mapsto \mathfrak{g}'$ by
\begin{equation}
\tau(w) = w^{2}|_{V_{\overline{0}}} \qquad \tau'(w) = w^{2}|_{V_{\overline{1}}}, \qquad w \in \mathfrak{s}(V, b)_{\overline{1}}.
\end{equation}
In particular, $\tau(\mathfrak{h}_{1}) \subseteq \mathfrak{h}$ and $\tau'(\mathfrak{h}_{1}) \subseteq \mathfrak{h}'$, and we consider the following embedding of $\mathfrak{h}$ into $\mathfrak{h}'$ given by
\begin{equation}
\mathfrak{h} \ni \sum\limits_{k=1}^n ih_{k}\E_{k, k} \mapsto \sum\limits_{k=1}^n ih_{k}\E_{k, k} \in \mathfrak{h}'.
\label{Derniere210784}
\end{equation}
We denote by $\tilde{\mathfrak{h}}$ the image of $\mathfrak{h}$ into $\mathfrak{h}'$ given in \eqref{Derniere210784}, and let $\mathfrak{z}'$ be the subalgebra of $\mathfrak{g}'$ given by
\begin{equation}
\mathfrak{z}' = C_{\mathfrak{g}'}\tilde{\mathfrak{h}} = \left\{x \in \mathfrak{g}', [x, y] = 0, (\forall y \in \tilde{\mathfrak{h}})\right\}.
\end{equation}

\begin{lemme}

Set $\pi_{\mathfrak{g}'/\mathfrak{z}'} = \prod\limits_{\underset{\alpha_{|_{\tilde{\mathfrak{h}}} \neq 0}}{\alpha \in \Phi^{+}(\mathfrak{g}', \mathfrak{h}')}} \alpha$. For all $h' = \sum\limits_{k = 1}^m ih_{k}\E_{k, k}$, we have:
\begin{equation}
\pi_{\mathfrak{g}'/\mathfrak{z}'}(h') = \prod\limits_{1 \leq i < j \leq n} (h^{2}_{i} - h^{2}_{j})\prod\limits_{i=1}^n \prod\limits_{j=n+1}^m (h^{2}_{i} - h^{2}_{j}).
\end{equation}

\end{lemme}

\noindent For every $\phi \in \S(W)$, we define the function $f_{\phi}$ on $\tau(\mathfrak{h}^{\reg}_{1})$ by 
\begin{equation}
f_{\phi}(\tau(w)) := C_{\mathfrak{h}_{\overline{1}}} \pi_{\mathfrak{g}'/\mathfrak{z}'}(\tau'(w)) \displaystyle\int_{\G'/\Z'} \phi(g'w)d(g'\Z')
\end{equation}
where $C_{\mathfrak{h}_{1}}$ is a constant of modulus 1 defined in \cite[Lemma~8,~page~17]{TOM4}.

\noindent As in the previous section, we denote by $\H^{\sharp}$ the two fold cover of $\tilde{\H}$ such that the linear forms $\frac{\alpha}{2}$ are analytic integral for every roots $\alpha \in \Phi(\mathfrak{g}, \mathfrak{h})$. Let $\xi^{\sharp}_{\frac{\alpha}{2}}: \H^{\sharp}_{0} \to \S^{1}$ the multiplicative character having the linear form $\frac{\alpha}{2}$ as derivative and let $c^{\sharp}_{-}: \mathfrak{h} \to \H^{\sharp}_{0}$ be the map defined in section \ref{Section5.1}.

\begin{prop}

Fix $\Pi \in \mathscr{R}(\tilde{\G}, \omega^{\infty})$ of highest weight $\nu$ and let $\mu = \nu + \rho$. For every $\phi \in \S(W)$, the following formula holds:
\begin{equation}
T(\overline{\Theta_{\Pi}})(\phi) = \displaystyle\int_{\mathfrak{h}}c_{-}(x)^{\mu'}\ch(x)^{m-n-1} \displaystyle\int_{\mathfrak{h} \cap \tau(W)} e^{i B(x, y)}f_{\phi}(y)dydx
\label{Proposition111111}
\end{equation}
where $\ch(x) = |\det_{\mathbb{R}}(x-1)|^{\frac{1}{2}} = |\det(\Psi^{2n} \circ \Psi_{n}(x) - \Id_{4n})|$ and $\mu' \in \mathfrak{h}^{*}$ is defined by
\begin{equation}
\mu'_{j} = \mu_{n+1-j} \qquad (1 \leq j \leq n).
\end{equation}

\end{prop}

\begin{proof}

For a proof of this proposition, we refer to \cite[Corollary~38,~page~47]{TOM4}.

\end{proof}

\begin{rema}

\begin{enumerate}
\item By \cite{VERGNE2}, the weights $\nu$ and $\nu'$ of the representations $\Pi$ and $\Pi'$ respectively given by:
\begin{equation}
\nu = \sum\limits_{a=1}^n \nu_{a}e_{a} \mapsto \nu' = -\sum\limits_{a=1}^m ne_{a} - \sum\limits_{a = 1}^n \nu_{n+1-a}e_{a},
\end{equation}
where $\nu_{1} \geq \ldots \geq \nu_{n} \geq 0$. By considering a decreasing sequence $\nu_{1} \geq \ldots \geq \nu_{m} \geq 0$ with at most $n$ non zero $\nu_{i}$, we get
\begin{equation}
\nu = \sum\limits_{a=1}^n \nu_{a}e_{a} \mapsto \nu' = -\sum\limits_{a=1}^{m} ne_{a} - \sum\limits_{a = 1}^m \nu_{m+1-a}e_{a} = \sum\limits_{a = 1}^m \tau_{a}e_{a},
\end{equation}
with $\tau_{a} = -n - \nu_{m+1-a}$, where $(\tau_{a})_{a \in [|1, m|]}$ is a decreasing sequence of negative numbers.

\noindent The linear form $\rho$ is given by:
\begin{eqnarray*}
\rho & = & \frac{1}{2} \sum\limits_{\alpha \in \Phi^{+}(\mathfrak{g}, \mathfrak{h})} \alpha = \frac{1}{2} \sum\limits_{1 \leq i < j \leq n} (e_{i} - e_{j}) + \frac{1}{2} \sum\limits_{1 \leq i < j \leq n} (e_{i} + e_{j}) + \frac{1}{2} \sum\limits_{k=1}^n 2e_{k} \\
       & = & \frac{1}{2} \sum\limits_{1 \leq i < j \leq n} 2e_{i} + \sum\limits_{k=1}^n e_{k} = \sum\limits_{k=1}^n (n-k+1)e_{k}
\end{eqnarray*}
Then $\mu = \sum\limits_{k = 1}^n (\nu+\rho)_{a}e_{a}$ with $(\nu+\rho)_{a} = \nu_{a} + n - a + 1$ and 
\begin{equation}
\mu'_{a} = \mu_{n+1-a} = \nu_{n+1-a} + a.
\end{equation}
\item For every $x, y \in \mathfrak{g}$ or $\mathfrak{g}'$, we denote by $B$ the bilinear form defined by
\begin{equation}
B(x, y) = \Re\tr(xy)
\label{FormeB3}
\end{equation}
The form $B$ is $\G$ (resp. $\G'$)-invariant and non-degenerate on $\mathfrak{g}$ and $\mathfrak{g}'$. More precisely, for all $x = \sum\limits_{k=1}^n ix_{k}\E_{k, k}$ (resp. $x' = \sum\limits_{k=1}^{m} ix'_{k}\E_{k, k}$) and $y = \sum\limits_{k=1}^n iy_{k}\E_{k, k}$ (resp. $y' = \sum\limits_{k=1}^{m} iy'_{k}\E_{k, k}$), the form $B$ is given by
\begin{equation}
B(x, y) = \sum\limits_{k=1}^n -2\pi x_{k} y_{k} \qquad \left(\text{resp. } B(x', y') = \sum\limits_{k=1}^m -2\pi x'_{k} y'_{k}\right).
\end{equation}
\end{enumerate}       
\end{rema}

\begin{theo}

For all regular element $y \in \tilde{\mathfrak{h}}$ and $y' \in \mathfrak{h}'^{\reg}$, we get:
\begin{equation}
\pi_{\mathfrak{g}'/\mathfrak{z}'}(y) \displaystyle\int_{\G'/\Z'} e^{iB(y', g'y)}d(g'\Z') = (-i)^{\frac{1}{2}\dim(\mathfrak{g}'/\mathfrak{z}')}(-1)^{n(\mathfrak{h'})} \sum\limits_{\sigma \in \mathscr{W}(\K', \mathfrak{h}')/\mathscr{W}(\K', \mathfrak{h}')^{\tilde{\mathfrak{h}}}} \cfrac{e^{iB(y', \sigma y)}}{\pi_{\mathfrak{g}'/\mathfrak{z}'}(\sigma^{-1}y')}
\label{Theoreme1}
\end{equation}
where $n(\tilde{\mathfrak{h}}')$ is the number of non-compacts positives roots which do not vanish on $\tilde{\mathfrak{h}}$.

\end{theo}

\noindent It is clear that for all $\psi \in \mathscr{C}^{\infty}_{c}(\G'.\mathfrak{h}'^{\reg})$, 
\begin{equation}
\phi: W \ni w \mapsto \displaystyle\int_{\mathfrak{g}'} \chi_{x'}(w) \psi(x')dx' \in \mathbb{C}
\end{equation}
is a Schwartz function on $W$. Using the Weyl integration formula, we get for all $y \in \tau(\mathfrak{h}_{1})$
\begin{eqnarray*}
& & \pi_{\mathfrak{g}'/\mathfrak{z}'}(y) \displaystyle\int_{\G'/\Z'} \displaystyle\int_{\mathfrak{g}'} e^{i B(y', g'y)} \psi(y')dy'd(g'\Z') \\
& = & \cfrac{1}{|\mathscr{W}(\K', \mathfrak{h}')|} \displaystyle\int_{\mathfrak{h}'} \left( \pi_{\mathfrak{g}'/\mathfrak{z}'}(y) \displaystyle\int_{G'/Z'} e^{i B(y', g'y)} d(g'\Z') \right) \psi^{\tilde{\G}'/\tilde{\H}'}(y') |\pi_{\mathfrak{g}'/\mathfrak{h}'}(y')|^{2}dy' 
\end{eqnarray*}
where $\psi^{\tilde{\G}'/\tilde{\H}'} : \mathfrak{h}'^{\reg} \to \mathbb{C}$ is given by:
\begin{equation}
\psi^{\tilde{\G'}/\tilde{\H'}}(y') = \displaystyle\int_{\tilde{\G'}/\tilde{\H'}} \psi(\Ad(\tilde{g}')(y')) d(\tilde{g'}\tilde{\H'}).
\label{psiG'H'}
\end{equation}
Using the same method than before, we get the following equality:
\begin{equation}
f_{\phi}(y) = C \displaystyle\int_{\mathfrak{h}'}e^{i B(y', y)} \pi_{\mathfrak{z}'/\mathfrak{h}'}(y') \overline{\pi_{\mathfrak{g}'/\mathfrak{h}'}(y')} \psi^{\tilde{\G}'/\tilde{\H}'}(y')dy'.
\end{equation}
We denote by $\mathscr{A}_{\psi}$ the function on $\mathfrak{h}'^{\reg}$ defined by
\begin{equation}
\mathscr{A}_{\psi}(y') = \pi_{\mathfrak{g}'/\mathfrak{h}'}(y') \psi^{\tilde{\G}'/\tilde{\H}'}(y') \qquad (y' \in \mathfrak{h}'^{\reg}).
\end{equation}
Then, 
\begin{equation}
f_{\phi} = \mathscr{F}_{\mathfrak{h}'}(C\pi_{\mathfrak{z}'/\mathfrak{h}'} \mathscr{A}_{\psi})_{|_{\tilde{\mathfrak{h}}}}.
\label{Fourier1234}
\end{equation}

\noindent Before giving the main theorem, we prove a result concerning the right hand-side of \eqref{Proposition111111}. More particularly, we are interested in the support of the following distribution:
\begin{equation}
\displaystyle\int_{\mathfrak{h}} c_{-}(x)^{-\mu'} \ch(x)^{m-n-1}e^{i B(x, y)} dx.
\label{DistributionSupport1234}
\end{equation}

\begin{prop}

Let $\beta_{1}$ and $\beta_{2}$ be the two integers defined by:
\begin{equation}
\beta_{1} = \begin{cases} \max\left\{k \geq 1, \mu'_{k} \leq -(m-n+\frac{3}{2})\right\} & \text{ if } \mu'_{1} \leq -(m-n+\frac{3}{2}) \\ 0 & \text{ otherwise } \end{cases}
\end{equation}
\begin{equation}
\beta_{2} = \begin{cases} \min\left\{ k \geq 1, \mu'_{k} \geq m-n+\frac{3}{2}\right\} - 1 & \text{ if } \mu'_{n} \geq m - n + \frac{3}{2} \\ n & \text{ otherwise } \end{cases}
\end{equation}
Then,
\begin{equation}
\supp \displaystyle\int_{\mathfrak{h}} c_{-}(x)^{-\mu'} \ch(x)^{m-n-1}e^{i B(x, y)} dx = \left\{\sum\limits_{k=1}^{\beta_{1}} ih_{k}\H_{k} + \sum\limits_{k=\beta_{2}+1}^{n} iy_{k}\H_{k}, \thinspace h_{k} > 0, y_{k} < 0 \right\}.
\end{equation}

\end{prop}

\begin{proof}

According to \cite[Lemma~33,~page~43]{TOM4}, we get:
\begin{equation}
\displaystyle\int_{\mathfrak{h}} c_{-}(x)^{-\mu'} \ch(x)^{m-n-1}e^{iB(x, y)} dx = \prod\limits_{k=1}^n \displaystyle\int_{\mathbb{R}} (1+ix_{k})^{-a_{k}}(1-ix_{k})^{-b_{k}} e^{-2i\pi x_{k}y_{k}}dx_{k},
\end{equation}
where $a_{k} = \mu'_{k} - m + n + 1$ and $b_{k} = -\mu'_{k} - m + n +1$. In particular, $a_{k} \geq 1$ (resp. $b_{k} \geq 1$) implies $b_{k} \leq 0$ (resp. $a_{k} \leq 0$).

\noindent Using \cite[Appendix~C]{TOM4}, we get the following equality
\begin{equation}
\displaystyle\int_{\mathbb{R}} (1+ix_{k})^{-a_{k}}(1-ix_{k})^{-b_{k}} dx_{k} = P_{a_{k}, b_{k}}(2\pi y_{k})e^{-|2\pi y_{k}|} + Q_{a_{k}, b_{k}}\left(\cfrac{\partial}{\partial y_{k}}\right)\delta_{0}(2\pi y_{k}),
\end{equation}
where $P_{a_{k}, b_{k}}$ and $Q_{a_{k}, b_{k}}$ are the polynomials defined in \cite[Appendix~C]{TOM4}. In particular, using \cite[Appendix~C]{TOM4}, the support of the polynomial $P_{a_{k}, b_{k}}$ are
\begin{equation}
\supp P_{a_{k}, b_{k}} = \begin{cases} ]-\infty, 0] & \text{ if } a_{k} \geq 1 \text{ and } b_{k} \leq 0\\ [0, +\infty[ & \text{ if } b_{k} \leq 0 \text{ and } a_{k} \geq 0 \\ \{\emptyset\} & \text{ otherwise } \end{cases} = \begin{cases} ]-\infty, 0] & \text{ if } \mu'_{k} \geq m-n \\ [0, +\infty[ & \text{ if } \mu'_{k} \leq -(m-n) \\ \{\emptyset\} & \text{ otherwise } \end{cases}.
\end{equation}
Then, we obtain:
\begin{equation}
\supp \displaystyle\int_{\mathbb{R}} (1+ix_{k})^{-a_{k}}(1-ix_{k})^{-b_{k}} dx_{k} = \begin{cases} [0, +\infty[ & \text{ if } k \leq \beta_{1} \\ \{0\} & \text{ if } k \in [|\beta_{1}+1, \beta_{2}|] \\ ]-\infty, 0] & \text{ if } k \geq \beta_{2}+1 \end{cases}
\end{equation}
and finally
\begin{equation}
\supp \displaystyle\int_{\mathfrak{h}} c_{-}(x)^{-\mu'} \ch(x)^{m-n-1}e^{iB(x, y)} dx = \left\{\sum\limits_{k = 1}^{\beta_{1}} ih_{k}\H_{k} + \sum \limits_{k = \beta_{2}+1}^{n} iy_{k}\H_{k}, h_{k} > 0, y_{k} < 0\right\} \subseteq \tau(\mathfrak{h}_{1}).
\end{equation}

\end{proof}

From now on, we denote by $\supp(\beta_{1}, \beta_{2})$ the support of the distribution given in equation \eqref{DistributionSupport1234}. We denote by $\tilde{\mathfrak{h}}^{\perp}$ the orthogonal complement of $\tilde{\mathfrak{h}}$ with respect to the bilinear form $B$, and by $\pr: \mathfrak{h} \to \tilde{\mathfrak{h}}$ the associated projection, i.e.
\begin{equation}
\pr(y_{1} + y_{2}) = y_{1} \qquad (y_{1} \in \tilde{\mathfrak{h}}, y_{2} \in \tilde{\mathfrak{h}}^{\perp}).
\label{Projection1234}
\end{equation}
Finally, let us consider the subgroup $\mathscr{W}(\G, \mathfrak{h}, \beta)$ of $\mathscr{W}(\G, \mathfrak{h})$ defined by:
\begin{equation}
\mathscr{W}(\G, \mathfrak{h}, \beta) = \left\{\eta \in \mathscr{W}(\G, \mathfrak{h}), \supp(\beta_{1}, \beta_{2}) \subseteq \eta \tau(\mathfrak{h}_{1})\right\}
\end{equation}

\begin{theo}

The character $\Theta_{\Pi'}$ of $\Pi' \in \mathscr{R}(\tilde{\G}', \omega^{\infty})$ is given, up to a constant, by the following formula:
\begin{equation}
\Theta_{\Pi'}(h) = \sum\limits_{\eta \in \mathscr{W}(\G, \mathfrak{h}, \beta)} \sum\limits_{\sigma \in \mathscr{W}(\K', \mathfrak{h}')} \cfrac{\sgn(\eta) \pr(\sigma h)^{-\eta^{-1}\mu'}}{\prod\limits_{\underset{\alpha_{|_{\tilde{\mathfrak{h}}} \neq 0}}{\alpha \in \Phi^{+}(\mathfrak{g}', \mathfrak{h}')}}((\sigma h)^{\frac{\alpha}{2}} - (\sigma h)^{-\frac{\alpha}{2}})}.
\label{Formule finale caractere 19081990}
\end{equation}

\end{theo}

\begin{proof}

By \eqref{Proposition111111} and \eqref{DistributionSupport1234}, we get
\begin{equation}
T(\overline{\Theta_{\Pi}})(\phi) = \displaystyle\int_{\mathfrak{h}} c_{-}(x)^{-\mu'}\ch(x)^{m-n-1} \displaystyle\int_{\mathscr{W}(\G, \mathfrak{h})\tau(\mathfrak{h}_{1})} e^{i B(x, y)} f_{\phi}(y)dy dx 
\end{equation}
Then,
\begin{eqnarray}
& & \displaystyle\int_{\mathfrak{h}} c_{-}(x)^{-\mu'}\ch(x)^{m-n-1} \displaystyle\int_{\mathscr{W}(\G, \mathfrak{h})\tau(\mathfrak{h}_{1})} e^{i B(x, y)} f_{\phi}(y)dy dx \nonumber \\
& = & \sum\limits_{\eta_{0} \in \mathscr{W}(G, \mathfrak{h},\beta)} \displaystyle\int_{\mathfrak{h}} c_{-}(x)^{-\mu'}\ch(x)^{m-n-1} \displaystyle\int_{\tau(\mathfrak{h}_{1})} e^{i B(x, \eta_{0}(y))} f_{\phi}(\eta_{0}(y)) dydx \nonumber \\
& = & \sum\limits_{\eta_{0} \in \mathscr{W}(\G, \mathfrak{h}, \beta)} \sgn(\eta_{0})\displaystyle\int_{\mathfrak{h}} c_{-}(x)^{-\eta^{-1}_{0}\mu'}\ch(x)^{m-n-1} \displaystyle\int_{\tau(\mathfrak{h}_{1})} e^{i B(x, y)} f_{\phi}(y) dydx \label{int1234}
\end{eqnarray}
Again, playing with the support of $f_{\phi}$, we get:
\begin{equation*}
\displaystyle\int_{\tau(\mathfrak{h}_{1})} e^{i B(x, y)} f_{\phi}(y) dy = \displaystyle\int_{\mathfrak{h}} e^{i B(x, y)} f_{\phi}(y) dy
\end{equation*}
We identify $\mathfrak{h}$ with its image $\tilde{\mathfrak{h}}$ in $\mathfrak{h}'$. Therefore, the equation \eqref{int1234} can be written as:
\begin{equation}
\sum\limits_{\eta_{0} \in \mathscr{W}(\G, \mathfrak{h}, \beta)} \sgn(\eta_{0})\displaystyle\int_{\tilde{\mathfrak{h}}} c_{-}(x)^{-\eta^{-1}_{0}\mu'}\ch(x)^{m-n-1} \displaystyle\int_{\tilde{\mathfrak{h}}} e^{i B(x, y)} f_{\phi}(y) dydx
\label{int22}
\end{equation}
By \eqref{Fourier1234}, this can be written as:
\begin{eqnarray}
& & \sum\limits_{\eta_{0} \in \mathscr{W}(G, \mathfrak{h}, \beta)} \sgn(\eta_{0})\displaystyle\int_{\mathfrak{h}} c_{-}(x)^{-\eta^{-1}_{0}\mu'}\ch(x)^{m-n-1} \displaystyle\int_{\tau(\mathfrak{h}_{\bar{1}, m})} e^{i B(x, y)} \mathscr{F}_{\mathfrak{h}'}(\pi_{\mathfrak{z}'/\mathfrak{h}'} \mathscr{A}_{\psi})(y) dydx \nonumber \\
& = & \sum\limits_{\eta_{0} \in \mathscr{W}(G, \mathfrak{h}, \beta)} \sgn(\eta_{0})\displaystyle\int_{\tilde{\mathfrak{h}}} c_{-}(x)^{-\eta^{-1}_{0}\mu'}\ch(x)^{m-n-1} \displaystyle\int_{\tilde{\mathfrak{h}}} e^{i B(x, y)} \left. \mathscr{F}_{\mathfrak{h}'}(\pi_{\mathfrak{z}'/\mathfrak{h}'} \mathscr{A}_{\psi})\right|_{\tilde{\mathfrak{h}}}(y) dydx \nonumber \\
& = & \sum\limits_{\eta_{0} \in \mathscr{W}(\G, \mathfrak{h}, \beta)} \sgn(\eta_{0})\displaystyle\int_{\tilde{\mathfrak{h}}} c_{-}(x)^{-\eta^{-1}_{0}\mu'}\ch(x)^{m-n-1} \mathscr{F}^{-1}_{\mathfrak{h}'} \left. \mathscr{F}_{\mathfrak{h}'}(\pi_{\mathfrak{z}'/\mathfrak{h}'} \mathscr{A}_{\psi})\right|_{\tilde{\mathfrak{h}}}(y) dydx \nonumber \\
& = & \sum\limits_{\eta_{0} \in \mathscr{W}(\G, \mathfrak{h}, \beta)} \sgn(\eta_{0})\displaystyle\int_{\tilde{\mathfrak{h}}} c_{-}(x')^{-\eta^{-1}_{0}\mu'}\ch(x')^{m-n-1} \displaystyle\int_{\tilde{\mathfrak{h}}^{\perp}} (\pi_{\mathfrak{z}'/\mathfrak{h}'} \mathscr{A}_{\psi})(x' + y') dy'dx' \nonumber \\
& = & \sum\limits_{\eta_{0} \in \mathscr{W}(\G, \mathfrak{h}, \beta)} \sgn(\eta_{0})\displaystyle\int_{\tilde{\mathfrak{h}}} c_{-}(x')^{-\eta^{-1}_{0}\mu'}\ch(x')^{m-n-1} \displaystyle\int_{\tilde{\mathfrak{h}}^{\perp}} \pi_{\mathfrak{z}'/\mathfrak{h}'}(y') \mathscr{A}_{\psi}(x' + y') dy'dx' \label{int333}
\end{eqnarray}
We now choose a particular function $\psi \in \mathscr{C}^{\infty}_{c}(\G'.\mathfrak{h}'^{\reg})$. We denote by $\G'^{c} \subseteq \G'$ the domain of the Cayley transform in $\G'$. Fix $\Psi \in \mathscr{C}^{\infty}_{c}(\widetilde{\G}')$ such that $\supp(\Psi) \in \widetilde{\G'^{c}} \subseteq \widetilde{\G}'$ and let $j_{\mathfrak{g}'}$ be the Jacobian of $c: \mathfrak{g}'^{c} \to \G'^{c}$. We recall that $\tilde{c}_{-}$ is defined by $\tilde{c}_{-}(x') = \tilde{c}(x') \tilde{c}(0)^{-1}$ and let $\psi$ be the function given by
\begin{equation}
\psi(x') = \Theta(\tilde{c}(x'))j_{\mathfrak{g}'}(x')\Psi(\tilde{c}(x')) \qquad (x' \in \mathfrak{g}'^{c}).
\label{choixpsi1}
\end{equation}
With such a function $\psi$, the integral \eqref{int333} becomes:
\begin{equation}
\sum\limits_{\eta_{0} \in \mathscr{W}(\G, \mathfrak{h}, \beta)} \sgn(\eta_{0})\displaystyle\int_{\mathfrak{h}'} c_{-}(\pr(x))^{-\eta^{-1}_{0}\mu'}\ch(\pr_{m}(x))^{m-n-1}\pi_{\mathfrak{z}'/\mathfrak{h}'}(x) \pi_{\mathfrak{g}'/\mathfrak{h}'}(x) \Theta(\tilde{c}(x)) j_{\mathfrak{g}'}(x) \Psi^{\tilde{\G}'/\tilde{\H}'}(x) dx,
\label{int44}
\end{equation}
where $\pr$ is the projection given in the equation \eqref{Projection1234}.

\noindent To get the character $\Theta_{\Pi'}$ of $\Pi'$, we would like to write the integral defined in \eqref{int44} as an integral over $\widetilde{\H}'$. For all $h' = \sum\limits_{k=1}^{m} ih_{k}\E_{k, k}$, we have:
\begin{equation}
\det_{\mathbb{R}}(h'-1) = \det\left(\Psi^{2m} \circ \Psi_{m}(h') - \Id_{4m}\right) = \prod\limits_{k=1}^m (1+h^{2}_{k})^{2}.
\end{equation}
Similarly, we obtain:
\begin{eqnarray*}
& & \prod\limits_{\alpha \in \Phi^{+}(\mathfrak{g}, \mathfrak{h})}(\xi^{\sharp}_{\frac{\alpha}{2}}(c^{\sharp}_{-}(h')) - \xi^{\sharp}_{-\frac{\alpha}{2}}(c^{\sharp}_{-}(h'))) = \prod\limits_{1 \leq i < j \leq m} \left(c_{-}(h')^{\frac{(e_{i} - e_{j})}{2}} - c_{-}(h')^{\frac{(-e_{i} + e_{j})}{2}}\right)  \prod\limits_{1 \leq i < j \leq m} \left(c_{-}(h')^{\frac{(e_{i} + e_{j})}{2}} - c_{-}(h')^{\frac{(-e_{i} - e_{j})}{2}}\right) \\
& = & \prod\limits_{1 \leq i < j \leq m} \left(c_{-}(ih_{i})^{\frac{1}{2}}c_{-}(ih_{j})^{-\frac{1}{2}} - c_{-}(ih_{i})^{-\frac{1}{2}}c_{-}(ih_{j})^{\frac{1}{2}})(c_{-}(ih_{i})^{\frac{1}{2}}c_{-}(ih_{j})^{\frac{1}{2}} - c_{-}(ih_{i})^{-\frac{1}{2}}c_{-}(ih_{j})^{-\frac{1}{2}}\right) \\
& = & \prod\limits_{1 \leq i < j \leq m} \left(c(ih_{i}) + c(ih_{i})^{-1} - c(ih_{j}) - c(ih_{j})^{-1}\right) = \prod\limits_{1 \leq i < j \leq m} \cfrac{ih_{i}+1}{ih_{i}-1} + \cfrac{ih_{i}-1}{ih_{i}+1} - \cfrac{ih_{j}+1}{ih_{j}-1} - \cfrac{ih_{j}-1}{ih_{j}+1} \\
& = & \prod\limits_{1 \leq i < j \leq m} \left(\cfrac{2(h^{2}_{i} - 1)}{1+h^{2}_{i}} - \cfrac{2(h^{2}_{j} - 1)}{1+h^{2}_{j}}\right) = \prod\limits_{1 \leq i < j \leq m} \cfrac{4(h^{2}_{i} - 1)(1+h^{2}_{j}) - 4(h^{2}_{j} - 1)(1+h^{2}_{i})}{(1+h^{2}_{i})(1+h^{2}_{j})} \\
& = & \prod\limits_{1 \leq i < j \leq m} \cfrac{4(h^{2}_{i} - h^{2}_{j})}{(1+h^{2}_{i})(1+h^{2}_{j})} = \cfrac{\prod\limits_{1 \leq i < j \leq m} 4(h^{2}_{i} - h^{2}_{j})}{\prod\limits_{k=1}^m (1 + h^{2}_{k})^{m-1}} = c \cfrac{\pi_{\mathfrak{g}'/\mathfrak{h}'}(h')}{\prod\limits_{k=1}^m (1 + h^{2}_{k})^{m-1}}.
\end{eqnarray*}

\noindent In particular, we get:
\begin{eqnarray*}
\prod\limits_{\underset{\alpha_{|_{\tilde{\mathfrak{h}}} \neq 0}}{\alpha \in \Phi^{+}(\mathfrak{g}, \mathfrak{h})}}(\xi_{\frac{\alpha}{2}}(c^{\sharp}_{-}(h')) - \xi_{-\frac{\alpha}{2}}(c^{\sharp}_{-}(h'))) & = & \cfrac{\prod\limits_{1 \leq i < j \leq n}4(h^{2}_{i} - h^{2}_{j})}{\prod\limits_{k = 1}^n (1 + h^{2}_{k})^{n-1}} \prod\limits_{i=1}^n \prod\limits_{j = n+1}^{m} \cfrac{4(h^{2}_{i} - h^{2}_{j})}{(1+h^{2}_{i})(1+h^{2}_{j})} \\
         & = & \cfrac{\prod\limits_{1 \leq i < j \leq n} 4(h^{2}_{i} - h^{2}_{j}) \prod\limits_{i=1}^n \prod\limits_{j=n+1}^m 4(h^{2}_{i} - h^{2}_{j})}{\prod\limits_{k=1}^{n}(1+h^{2}_{k})^{m-1} \prod\limits_{k=n+1}^m (1+h^{2}_{k})^{n}} \\
         & = & C' \cfrac{\pi_{\mathfrak{g}'/\mathfrak{z}'}(h')}{\prod\limits_{k=1}^{n}(1+h^{2}_{k})^{m-1} \prod\limits_{k=n+1}^m (1+h^{2}_{k})^{n}}
\end{eqnarray*}

\noindent For all $h' = \sum\limits_{k=1}^m ih_{k}\E_{k, k}$, we get:
\begin{equation}
\ch(h') = \prod\limits_{k=1}^m (1+h^{2}_{k}), \qquad \ch(\pr(h')) = \prod\limits_{k=1}^n (1+h^{2}_{k}),
\end{equation}
\begin{equation}
j_{\mathfrak{g}'}(h') = \prod\limits_{k=1}^m (1+h^{2}_{k})^{-(2m-1)} \qquad j_{\mathfrak{h}'}(h') = \prod\limits_{k=1}^m (1+h^{2}_{k})^{-1}.
\end{equation}

\noindent Then,

\begin{eqnarray*}
& & \ch(\pr(h'))^{m-n-1} \pi_{\mathfrak{z}'/\mathfrak{h}'}(h') \pi_{\mathfrak{g}'/\mathfrak{h}'}(h') \Theta(\tilde{c}(h')) \\
& = & \prod\limits_{k=1}^n (1+h^{2}_{k})^{m-n-1} \prod\limits_{1 \leq i < j \leq m}(-h^{2}_{i} + h^{2}_{j}) \prod\limits_{n+1 \leq i < j \leq m} (-h^{2}_{i} + h^{2}_{j}) \prod\limits_{k=1}^m(1+h^{2}_{k})^{n} \\
& = & \prod\limits_{k=1}^n (1+h^{2}_{k})^{m-1} \prod\limits_{k=n+1}^{m} (1+h^{2}_{k})^{n} \prod\limits_{1 \leq i < j \leq m}(-h^{2}_{i} + h^{2}_{j}) \prod\limits_{n+1 \leq i < j \leq m} (-h^{2}_{i} + h^{2}_{j}) \\
& = & \prod\limits_{k=1}^n (1+h^{2}_{k})^{m-1} \prod\limits_{k=n+1}^{m} (1+h^{2}_{k})^{n} \cfrac{\left|\prod\limits_{\alpha \in \Phi^{+}(\mathfrak{g}, \mathfrak{h})}(\xi_{\frac{\alpha}{2}}(c^{\sharp}_{-}(h')) - \xi_{-\frac{\alpha}{2}}(c^{\sharp}_{-}(h'))) \right|^{2}}{\prod\limits_{\underset{\alpha_{|_{\tilde{\mathfrak{h}}} \neq 0}}{\alpha \in \Phi^{+}(\mathfrak{g}, \mathfrak{h})}}(\xi_{\frac{\alpha}{2}}(c^{\sharp}_{-}(h')) - \xi_{-\frac{\alpha}{2}}(c^{\sharp}_{-}(h')))} \prod\limits_{k=1}^n (1+h^{2}_{k})^{m-1} \prod\limits_{k=n+1}^m (1+h^{2}_{k})^{2m-2+n} \\
& = & \cfrac{\left|\prod\limits_{\alpha \in \Phi^{+}(\mathfrak{g}, \mathfrak{h})}(\xi_{\frac{\alpha}{2}}(c^{\sharp}_{-}(h')) - \xi_{-\frac{\alpha}{2}}(c^{\sharp}_{-}(h'))) \right|^{2}}{\prod\limits_{\underset{\alpha_{|_{\tilde{\mathfrak{h}}} \neq 0}}{\alpha \in \Phi^{+}(\mathfrak{g}, \mathfrak{h})}}(\xi_{\frac{\alpha}{2}}(c^{\sharp}_{-}(h')) - \xi_{-\frac{\alpha}{2}}(c^{\sharp}_{-}(h')))} \prod\limits_{k=1}^n (1+h^{2}_{k})^{2m-2} \prod\limits_{k=n+1}^m (1+h^{2}_{k})^{2m-2} \\
& = & \cfrac{\left|\prod\limits_{\alpha \in \Phi^{+}(\mathfrak{g}, \mathfrak{h})}(\xi_{\frac{\alpha}{2}}(c^{\sharp}_{-}(h')) - \xi_{-\frac{\alpha}{2}}(c^{\sharp}_{-}(h'))) \right|^{2}}{\prod\limits_{\underset{\alpha_{|_{\tilde{\mathfrak{h}}} \neq 0}}{\alpha \in \Phi^{+}(\mathfrak{g}, \mathfrak{h})}}(\xi_{\frac{\alpha}{2}}(c^{\sharp}_{-}(h')) - \xi_{-\frac{\alpha}{2}}(c^{\sharp}_{-}(h')))} (j^{-1}_{\mathfrak{g}'}j_{\mathfrak{h}'})(h')
\end{eqnarray*}

and the equation \eqref{int44} becomes:
\begin{eqnarray*}
& & T(\overline{\Theta_{\Pi}})(\phi) \\
&=&  C'_{1} \sum\limits_{\eta_{0} \in \mathscr{W}(G, \mathfrak{h}, \beta)} \sgn(\eta_{0})\displaystyle\int_{\mathfrak{h}'} c_{-}(\pr(x))^{-\eta^{-1}_{0}\mu'}\ch(\pr(x))^{m-n-1}\pi_{\mathfrak{z}'/\mathfrak{h}'}(x) \pi_{\mathfrak{g}'/\mathfrak{h}'}(x) \Theta(\tilde{c}(x)) j_{\mathfrak{g}'}(x) \Psi^{\tilde{\G}'/\tilde{\H}'}(\tilde{c}_{-}(x)) dx \\
& = & C''_{1} \sum\limits_{\eta_{0} \in \mathscr{W}(\G, \mathfrak{h}, \beta)} \sgn(\eta_{0}) \displaystyle\int_{\mathfrak{h}'} c_{-}(\pr(x))^{-\eta^{-1}_{0}\mu'} \cfrac{\left| \prod\limits_{\alpha \in \Phi^{+}(\mathfrak{g}', \mathfrak{h}')} (c_{-}(x)^{\frac{\alpha}{2}} - c_{-}(x)^{-\frac{\alpha}{2}})\right|^{2}}{\prod\limits_{\underset{\alpha_{|_{\tilde{\mathfrak{h}}}} \neq 0}{\alpha \in \Phi^{+}(\mathfrak{g}', \mathfrak{h}')}} (c_{-}(x)^{\frac{\alpha}{2}} - c_{-}(x)^{-\frac{\alpha}{2}})} j_{\mathfrak{h}}(x) \Psi^{\tilde{\G}'/\tilde{\H}'}(\tilde{c}_{-}(x)) dx
\end{eqnarray*}

\noindent Finally, 
\begin{equation}
T(\overline{\Theta_{\Pi}})(\phi) = \displaystyle\int_{\widetilde{\H}'} C''_{1} \sum\limits_{\eta_{0} \in \mathscr{W}(\G, \mathfrak{h}, \beta)} \cfrac{\sgn(\eta_{0}) \pr(h)^{-\eta^{-1}_{0}\mu'}}{\prod\limits_{\underset{\alpha_{|_{\tilde{\mathfrak{h}}}} \neq 0}{\alpha \in \Phi^{+}(\mathfrak{g}', \mathfrak{h}')}} (h^{\frac{\alpha}{2}} - h^{-\frac{\alpha}{2}})} \left|\prod\limits_{\alpha \in \Phi^{+}(\mathfrak{g}', \mathfrak{h}')} (h^{\frac{\alpha}{2}} - h^{-\frac{\alpha}{2}})\right|^{2} \Psi^{\tilde{\G}'/\tilde{\H}'}(h) dh
\label{intH'1}
\end{equation}
Using the invariance of $\left|\prod\limits_{\alpha \in \Phi^{+}(\mathfrak{g}', \mathfrak{h}')} (h^{\frac{\alpha}{2}} - h^{-\frac{\alpha}{2}})\right|^{2} \Psi^{\tilde{\G}'/\tilde{\H}'}(h)$ under the action of $\mathscr{W}(\K', \mathfrak{h}')$, we get:
\begin{equation*}
\left|\prod\limits_{\alpha \in \Phi^{+}(\mathfrak{g}', \mathfrak{h}')}(h^{\frac{\alpha}{2}} - h^{-\frac{\alpha}{2}})\right|^{2} \Psi^{\tilde{\G}'/\tilde{\H}'}(h) = \cfrac{1}{|\mathscr{W}(\K', \mathfrak{h}')|} \sum\limits_{\sigma \in \mathscr{W}(\K', \mathfrak{h}')} \left|\prod\limits_{\alpha \in \Phi^{+}(\mathfrak{g}', \mathfrak{h}')}((\sigma h)^{\frac{\alpha}{2}} - (\sigma h)^{-\frac{\alpha}{2}})\right|^{2} \Psi^{\tilde{\G}'/\tilde{\H}'}(\sigma h)    
\end{equation*}   
and
\begin{eqnarray*}
& &T(\overline{\Theta_{\Pi}})(\phi) \\
& = &  \cfrac{C''_{1}}{|\mathscr{W}(\K', \mathfrak{h}')|}\displaystyle\int_{\widetilde{\H}'} \sum\limits_{\eta_{0} \in \mathscr{W}(\G, \mathfrak{h}, \beta)} \sum\limits_{\sigma \in \mathscr{W}(\K', \mathfrak{h}')} \cfrac{\sgn(\eta_{0}) \pr(h)^{-\eta^{-1}_{0}\mu'}}{\prod\limits_{\underset{\alpha_{|_{\tilde{\mathfrak{h}}}} \neq 0}{\alpha \in \Phi^{+}(\mathfrak{g}', \mathfrak{h}')}} (h^{\frac{\alpha}{2}} - h^{-\frac{\alpha}{2}})} \left|\prod\limits_{\alpha \in \Phi^{+}(\mathfrak{g}', \mathfrak{h}')} ((\sigma h)^{\frac{\alpha}{2}} - (\sigma h)^{-\frac{\alpha}{2}})\right|^{2} \Psi^{\tilde{\G}'/\tilde{\H}'}(\sigma h) dh \\
                  & = & \cfrac{C''_{1}}{|\mathscr{W}(\K', \mathfrak{h}')|} \displaystyle\int_{\tilde{\H}'} \sum\limits_{\eta_{0} \in \mathscr{W}(\G, \mathfrak{h}, \beta)} \sum\limits_{\sigma \in \mathscr{W}(\K', \mathfrak{h}')} \cfrac{\sgn(\eta_{0}) \pr(\sigma^{-1}h)^{-\eta^{-1}_{0}\mu'}}{\prod\limits_{\underset{\alpha_{|_{\tilde{\mathfrak{h}}}} \neq 0}{\alpha \in \Phi^{+}(\mathfrak{g}', \mathfrak{h}')}} ((\sigma^{-1}h)^{\frac{\alpha}{2}} - (\sigma^{-1}h^{-\frac{\alpha}{2}})} \left|\prod\limits_{\alpha \in \Phi^{+}(\mathfrak{g}', \mathfrak{h}')} (h^{\frac{\alpha}{2}} - h^{-\frac{\alpha}{2}})\right|^{2} \Psi^{\tilde{\G}'/\tilde{\H}'}(h) dh 
\end{eqnarray*}
Finally, we obtain that the character $\Theta_{\Pi'}$ is given, up to a constant, by the formula:
\begin{equation*}
\sum\limits_{\eta \in \mathscr{W}(\G, \mathfrak{h}, \beta)} \sum\limits_{\sigma \in \mathscr{W}(\K', \mathfrak{h}')} \cfrac{\sgn(\eta) \pr(\sigma h)^{-\eta^{-1}\mu'}}{\prod\limits_{\underset{\alpha_{|_{\tilde{\mathfrak{h}}} \neq 0}}{\alpha \in \Phi^{+}(\mathfrak{g}', \mathfrak{h}')}}((\sigma h)^{\frac{\alpha}{2}} - (\sigma h)^{-\frac{\alpha}{2}})}.
\end{equation*}

\end{proof}

\appendix

\section{Dual pairs viewed as Lie supergroups}
\label{AppendixA}

\noindent For readers who are not familiar with the "correspondence" between irreducible reductive dual pairs and some Lie supergroups, we recall the main ideas in this section. This correspondence appears to be very useful when one wants to work with dual pairs. All of this section is based on the article \cite{TOM5}. 

\noindent By a Lie supergroup, here we mean a Harish-Chandra pair. We recall the definition here (more details can be found in \cite{VAR}).

\begin{defn}

\noindent A Harish-Chandra pair is a triple $(\G, \mathfrak{g} = \mathfrak{g}_{\overline{0}} \oplus \mathfrak{g}_{\overline{1}}, \Ad)$ where:
\begin{itemize}
\item $\G$ is a Lie group with the Lie algebra $\mathfrak{g}_{\overline{0}}$,
\item $\mathfrak{g}$ is a Lie superalgebra,
\item $\Ad: \G \to \GL(\mathfrak{g})$ is a representation of $\G$ such that the differential is equal to the action of $\mathfrak{g}_{\overline{0}}$ on $\mathfrak{g}$ via the superbracket.
\end{itemize}

\end{defn}

\noindent Let $(\G, \G')$ be an irreducible reductive dual pair of type $I$ in $\Sp(W)$. According to \cite[Lecture~5,~Theorem~5.3]{LI}, there exists a division algebra $\mathbb{D}$ over $\mathbb{R}$ with an involution $\iota$, a hermitian space $(V, (\cdot,\cdot))$ and a skew-hermitian space $(V', (\cdot,\cdot)')$ such that $W = V \otimes_{\mathbb{D}} V'$, and $\G$ (resp. $\G'$) can be identified with the isometry group of the form $(V, (\cdot,\cdot)')$ (resp. $(V', (\cdot,\cdot)')$). Moreover, the symplectic form is given by the following formula
\begin{equation}
\left<u_{1} \otimes u_{2}, w_{1} \otimes w_{2}\right> = \tr_{\mathbb{D}/\mathbb{R}}((u_{1}, w_{1}).(u_{2}, w_{2})').
\end{equation}

\noindent Denote by $E_{\overline{0}} = V$ and $E_{\overline{1}} = V'$ and let $E = E_{\overline{0}} \oplus E_{\overline{1}}$. We denote by $\End(E)$ the set of endomorphisms of $E$ with $\mathbb{Z}_{2}$-graduation given by:
\begin{equation}
\End(V)_{\alpha} = \{X \in \End(V), X(V_{\beta}) \subseteq V_{\alpha + \beta}, \beta \in \mathbb{Z}_{2}\} \qquad (\alpha \in \mathbb{Z}_{2}).
\end{equation}
We consider the even bilinear form $\tau= (\cdot,\cdot) \oplus (\cdot,\cdot)'$ defined on $E$ by
\begin{equation}
\tau(u_{0} + u_{1}, v_{0} + v_{1}) = (u_{0}, v_{0}) + (u_{1}, v_{1})'.
\end{equation}
Let $s \in \End(E)_{\overline{0}}$ be defined by
\begin{equation}
s(u_{0} + u_{1}) = u_{0} - u_{1} \qquad u_{0} \in E_{\overline{0}}, u_{1} \in E_{\overline{1}}.
\end{equation}
The following equality holds
\begin{equation}
\tau(u, v) = \iota(\tau(v, su)) \qquad (u, v \in E).
\end{equation}
We consider
\begin{equation}
\mathfrak{g}(E, \tau)_{\overline{0}} = \left\{X \in \End(E)_{\overline{0}}, \tau(Xu, v) = \tau(u, -Xv), u, v \in E\right\},
\end{equation}
\begin{equation}
\mathfrak{g}(E, \tau)_{\overline{1}} = \left\{X \in \End(E)_{\overline{1}}, \tau(Xu, v) = \tau(u, sXv), u, v \in E\right\},
\end{equation}
\begin{equation}
\GL(E, \tau)_{\overline{0}} = \{g \in \End(E)_{\overline{0}} \cap \GL(E), \tau(g(u), g(v)) = \tau(u, v), u, v \in E\}.
\end{equation}

\noindent As shown in \cite{TOM5}, $(\S, \mathfrak{s}, \Ad) = (\GL(E, \tau)_{\overline{0}}, \mathfrak{g}(E, \tau) = \mathfrak{g}(E, \tau)_{\overline{0}} \oplus \mathfrak{g}(E, \tau)_{\overline{1}}, \Ad)$ is a Lie supergroup (where $\Ad$ is the natural action by conjugation). Moreover, the space $\Hom(V_{\overline{0}}, V_{\overline{1}})$ is identified with the symplectic vector space $W$ and the symplectic form is given by:
\begin{equation}
\left<x, y\right> = \tr_{\mathbb{D}/\mathbb{R}}\{sx, y\} \qquad x, y \in \mathfrak{g}(E, \tau)_{\overline{1}},
\end{equation}
where $\{\cdot,\cdot\}$ is the superbracket on $\mathfrak{g}(E, \tau)$.

\noindent In this context, the moment maps $\tau: W \to \mathfrak{g}$ and $\tau': W \to \mathfrak{g}'$ are defined, for $X = \begin{pmatrix} 0 & B \\ B^{*} & 0 \end{pmatrix}$, by $\tau(X) = X^{2}_{|_{V_{\overline{0}}}}$ and $\tau'(X) = X^{2}_{|_{V_{\overline{1}}}}$. In particular, we have:
\begin{equation}
\tau \begin{pmatrix} 0 & B \\ B^{*} & 0 \end{pmatrix}  = BB^{*}, \qquad \tau' \begin{pmatrix} 0 & B \\ B^{*} & 0 \end{pmatrix} = B^{*}B.
\end{equation}

\section{Some comments concerning Theorem \ref{TheoremU(p,q)}}

\subsection{How can we get the constant $C$?}

To determine the constant appearing in Theorem \ref{TheoremU(p,q)}, we use an idea developed by Harish-Chandra in \cite[Chapter~III,~Section~41,~page~97]{HAR1}. Fix $a \in \mathscr{C}^{\infty}_{c}(\widetilde{\G}')$ a $\widetilde{\K}'-$right invariant function such that:
\begin{equation*}
\displaystyle\int_{\widetilde{\G}'} a(g') dg' = 1
\end{equation*}
To simplify the notations, we will denote by $\mathscr{W}$ the compact Weyl group $\mathscr{W}(\K', \mathfrak{h}')$ of $\K'$. Fix 
\begin{equation}
\Delta_{0}(h) = \prod\limits_{\alpha \in \Phi^{+}(\mathfrak{k}', \mathfrak{h}')} (h^{\frac{\alpha}{2}} - h^{-\frac{\alpha}{2}}) \qquad \Delta_{+}(h) = \prod\limits_{\alpha \in \Phi^{+}(\mathfrak{g}', \mathfrak{h}') \setminus \Phi^{+}(\mathfrak{k}', \mathfrak{h}')} (h^{\frac{\alpha}{2}} - h^{-\frac{\alpha}{2}}).
\end{equation}
In particular,
\begin{equation}
\Delta(h) = \Delta_{0}(h) \Delta_{+}(h) = \prod\limits_{\alpha > 0} (h^{\frac{\alpha}{2}} - h^{-\frac{\alpha}{2}})
\end{equation}

\noindent For every function $\beta \in \mathscr{C}^{\infty}_{c}(\widetilde{\H}'^{\reg})$, we denote by $f_{\beta}$ the function of $\mathscr{C}^{\infty}_{c}(\widetilde{\G}'_{\widetilde{\H}'})$ defined as
\begin{equation*}
f_{\beta}(h^{x}) = a(x) \Delta_{0}(h)^{-1} \sum\limits_{\sigma \in \mathscr{W}} \sgn(\sigma) \beta(h^{\sigma}) \qquad (h \in \widetilde{\H}'^{\reg}, x \in \widetilde{\G}')
\end{equation*}
(here, according to \cite{HAR1}, $\widetilde{\G}'_{\widetilde{\H}'}$ is the set of the $\widetilde{\G}'$-orbits of $\widetilde{\H}'^{\reg}$ and $h^{x} = xhx^{-1}$). Similarly, we define the function $g_{\beta}$ of $\mathscr{C}^{\infty}(\widetilde{\K}')$ as
\begin{equation*}
g_{\beta}(h^{k}) = \Delta_{+}(h)\Delta_{0}(h)^{-1} \sum\limits_{\sigma \in \mathscr{W}} \sgn(\sigma)\beta(h^{\sigma}) \qquad k \in \widetilde{\K}'.
\end{equation*}
We denote by $f = \frac{1}{2} \dim (\G'/\H')$, $f_{0} = \frac{1}{2} \dim(\K'/\H')$ and $e = f - f_{0}$. From now on, we will identify a representation $\Pi_{\gamma} \in \widehat{\widetilde{\K}'}$ with its highest weight $\gamma$. According to \cite[page~97]{HAR1}, we have:
\begin{equation*}
\Theta_{\Pi'}(f_{\beta}) = (-1)^q \sum\limits_{\gamma \in \widehat{\widetilde{\K}'}} m(\gamma) \displaystyle\int_{\widetilde{\K}'} g_{\beta}(k) \Theta_{\gamma}(k) dk
\end{equation*}
where $m(\gamma)$ is the multiplicity of $\gamma$ in the decomposition of $\Theta_{\Pi'}$. The last equality can be rewritten as follows: 
\begin{eqnarray*}
\Theta_{\Pi'}(f_{\beta}) & = & (-1)^q \sum\limits_{\gamma \in \widehat{\widetilde{\K}'}} m(\gamma) \displaystyle\int_{\widetilde{\K}'} g_{\beta}(k) \Theta_{\gamma}(k) dk \\
    & = & (-1)^q \sum\limits_{\gamma \in \widehat{\widetilde{\K}'}} \cfrac{m(\gamma)}{|\mathscr{W}(\K', \mathfrak{h}')|} \displaystyle\int_{\widetilde{\H}'} g_{\beta}(h) \Theta_{\gamma}(h) \left|\Delta_{0}(h)\right|^{2} dh \\
    & = & (-1)^q \sum\limits_{\gamma \in \widehat{\widetilde{\K}'}} \sum\limits_{\sigma \in \mathscr{W}} \sgn(\sigma) \cfrac{m(\gamma)}{|\mathscr{W}(\K', \mathfrak{h}')|} \displaystyle\int_{\widetilde{\H}'} \Delta_{+}(h)\Delta_{0}(h)^{-1}\beta(h^{\sigma}) \Theta_{\gamma}(h) \left|\Delta_{0}(h)\right|^{2} dh \\
    & = & (-1)^m \sum\limits_{\gamma \in \widehat{\widetilde{\K}'}} \sum\limits_{\sigma \in \mathscr{W}} \sgn(\sigma) \cfrac{m(\gamma)}{|\mathscr{W}(\K', \mathfrak{h}')|} \displaystyle\int_{\widetilde{\H}'} \Delta_{+}(h)\beta(h^{\sigma}) \Theta_{\gamma}(h) \Delta_{0}(h) dh
\end{eqnarray*}
The equation \eqref{FormuleFinaleU(p,q)} is of the form
\begin{equation*}
\Theta_{\Pi'}(h) = \cfrac{C P(h)}{\Delta(h)},
\end{equation*}
with $P(h) = \sum\limits_{\eta \in \mathscr{W}(\G, \mathfrak{h}, m)/\mathscr{W}(\G, \mathfrak{h})_{m}} \sum\limits_{\sigma \in \mathscr{W}(\Z'(m), \mathfrak{h}')} \sum\limits_{\tau \in \mathscr{W}(\K', \mathfrak{h}')} \sgn(\eta\sigma\tau) \pr_{m}(\tau h)^{-\eta^{-1}\mu'}(\tau h)^{\sigma  \rho_{\mathfrak{z}'(m)}}$.

\noindent On the other hand, we have
\begin{eqnarray*}
 \displaystyle\int_{\widetilde{\G}'} f_{\beta}(g) \Theta_{\Pi'}(g)dg & = & \cfrac{1}{|\mathscr{W}(\K', \mathfrak{h}')|} \displaystyle\int_{\widetilde{\H}'} f_{\beta}(h) \Theta_{\Pi'}(h)|\Delta(h)|^{2}dg \\
                          & = & \cfrac{C}{|\mathscr{W}(\K', \mathfrak{h}')|} \displaystyle\int_{\widetilde{\H}'} f_{\beta}(h) \cfrac{P(h)}{\Delta(h)}|\Delta(h)|^{2}dg \\
                          & = & \cfrac{C}{|\mathscr{W}(\K', \mathfrak{h}')|} \sum\limits_{\sigma \in \mathscr{W}} \sgn(\sigma)\displaystyle\int_{\widetilde{\H}'} \Delta(h)^{-1} \beta(h^{\sigma}) P(h) \cfrac{|\Delta(h)|^{2}}{\Delta(h)} dh \\
                          & = & \cfrac{C(-1)^{m}}{|\mathscr{W}(\K', \mathfrak{h}')|} \sum\limits_{\sigma \in \mathscr{W}} \sgn(\sigma)\displaystyle\int_{\widetilde{\H}'} P(h) \beta(h^{\sigma}) dh.
\end{eqnarray*}
So, we get the following equality
\begin{equation}
C(-1)^{m} \sum\limits_{\sigma \in \mathscr{W}} \sgn(\sigma)\displaystyle\int_{\widetilde{\H}'} P(h) \beta(h^{\sigma}) dh = (-1)^m \sum\limits_{\gamma \in \widehat{\widetilde{\K}'}} \sum\limits_{\sigma \in \mathscr{W}} \sgn(\gamma) m(\gamma) \displaystyle\int_{\widetilde{\H}'} \Delta_{+}(h)\beta(h^{\sigma}) \Theta_{\gamma}(h) \Delta_{0}(h) dh.
\label{eqref1}
\end{equation}
This equality is valid for every function $\beta \in \mathscr{C}^{\infty}_{c}(\widetilde{\H}'^{\reg})$, in particular for the function $\beta$ given by $\beta(h) = \cfrac{h^{-(\lambda + \rho)}}{\Delta_{+}(h)}$, where $\lambda$ is the highest weight of the representation $\Pi'$. For this particular function $\beta$, we get:
\begin{eqnarray}
\sum\limits_{\gamma \in \hat{\widetilde{\K}'}} \sum\limits_{\sigma \in \mathscr{W}} \sgn(\sigma) m(\gamma) \displaystyle\int_{\widetilde{\H}'} \Delta_{+}(h)\beta(h^{\sigma}) \Theta_{\gamma}(h) \Delta_{0}(h) dh & = & \sum\limits_{\gamma \in \hat{\widetilde{\K}'}} \sum\limits_{\sigma \in \mathscr{W}} \sgn(\sigma) m(\gamma) \displaystyle\int_{\widetilde{\H}'} h^{-(\lambda + \rho)} \Theta_{\gamma}(h) \Delta_{0}(h) dh \nonumber \\
     & = & \sum\limits_{\gamma \in \hat{\widetilde{\K}'}} m(\gamma) \displaystyle\int_{\widetilde{\H}'} \left(\sum\limits_{\sigma \in \mathscr{W}} \sgn(\sigma)h^{-(\lambda + \rho)} \right) \Theta_{\gamma}(h) \Delta_{0}(h) dh \nonumber \\
     & = & \sum\limits_{\gamma \in \hat{\widetilde{\K}'}} m(\gamma) \displaystyle\int_{\widetilde{\H}'} \overline{\Theta_{\lambda}(h)}\Theta_{\gamma}(h) \left|\Delta_{0}(h)\right|^{2} dh  \label{eqref2}
     \end{eqnarray}
Using \cite[Chapter~IV,~Section~2,~Corollary~4.10]{KNA}, we obtain:
\begin{equation*}
\displaystyle\int_{\widetilde{\K}'} \overline{\Theta_{\lambda}(k)} \Theta_{\gamma}(k) dk = \begin{cases} 1 & \text{ if }  \lambda = \eta \\ 0 & \text{ otherwise } \end{cases}
\end{equation*}
By the Weyl integration formula, we have:
\begin{equation}
\displaystyle\int_{\widetilde{\H}'} \overline{\Theta_{\lambda}(h)} \Theta_{\gamma}(h)\left|\Delta_{0}(h)\right|^{2} dk = \begin{cases} |\mathscr{W}| & \text{ if }  \lambda = \eta \\ 0 & \text{ otherwise } \end{cases}
\label{eqref3}
\end{equation}
According to \cite[Theorem~9,~page~555]{HOWE5}, there exists a unique $\widetilde{\K}'$-type with multiplicity $1$ and of highest weight $\lambda$. Then,
\begin{equation*}
|\mathscr{W}| = C \displaystyle\int_{\widetilde{\H}'} P(h) \left( \sum\limits_{\sigma \in \mathscr{W}} \sgn(\sigma) h^{-\sigma(\lambda + \rho)}\right)\cfrac{dh}{\Delta_{+}(h)}
\end{equation*}
Finally, we obtain the following formula for the constant $C$:
\begin{equation}
C^{-1} = \cfrac{1}{|\mathscr{W}|} \displaystyle\int_{\widetilde{\H}'} \left( \sum\limits_{\sigma \in \mathscr{W}} \sgn(\sigma) h^{-\sigma(\lambda + \rho)}\right)\cfrac{P(h)}{\Delta_{+}(h)} dh
\label{normalizingC}
\end{equation}

\subsection{A remark concerning the interval $[|m_{\min}, m_{\max}|]$ for $\G = \U(1, \mathbb{C})$}

In section 5, more particularly in Theorem \ref{TheoremU(p,q)}, we get an explicit formula for the character of the representation $\Pi'$ on the maximal compact Cartan subgroup of $\U(p, q, \mathbb{C})$. In particular, we consider an embedding of $\mathfrak{h}$ in $\mathfrak{h}'$, denoted by $\mathfrak{h}'(m)$, where $m$ can be chosen arbitrarily in $[|m_{\min}, m_{\max}|]$, and the surprising fact is that the character formula \eqref{FormuleFinaleU(p,q)} does not depend of the choice of $m$. In this section, for the dual pair $(\G= \U(1, \mathbb{C}), \G' = \U(p, q, \mathbb{C})), 1 \leq p \leq q$, we compute explicitly the formula \eqref{FormuleFinaleU(p,q)} for two different values of $m$.  For this particular dual pair, we get:
\begin{equation*}
m_{\min} = \begin{cases} 1 & \text{ if } \lambda_1 \leq -q \\ 0 & \text{ if } \lambda_1 > -q \end{cases}
\end{equation*}
\begin{equation*}
m_{\max} = \begin{cases} 0 & \text{ if } \lambda_1 \geq p \\ 1 & \text{ if } \lambda_1 < p \end{cases}
\end{equation*} 
Then, the possible values for the parameter $m$ are:
\begin{equation}
m = \begin{cases}
1 & \text{ if } \lambda_1 \leq -q \\ 0 \text{ or } 1 & \text{ if } -q < \lambda_1 < p \\ 0 & \text{ if } \lambda_1 \geq p \end{cases}.
\end{equation}
Using \eqref{FormuleFinaleU(p,q)}, we get that the character $\Theta_{\Pi'}$ is given by:
\begin{equation*}
\Theta_{\Pi'}(h) = C \sum\limits_{\sigma \in \mathscr{W}(\K', \mathfrak{h}')} \cfrac{\pr_{m}(\sigma h)^{-\mu'}}{\prod\limits_{\underset{\alpha_{|_{\mathfrak{h}'(m)} \neq 0}}{\alpha > 0}}((\sigma h)^{\frac{\alpha}{2}} - (\sigma h)^{-\frac{\alpha}{2}})}
\end{equation*}
where $C$ is a constant. We start with $m = 1$. For all $h = \diag(h_{1}, \ldots, h_{p+q})$, we have $\pr_{1}(h) = h_{1}$. Moreover,
\begin{equation*}
\left\{\alpha \in \Phi^+(\mathfrak{g}', \mathfrak{h}') \thinspace ; \thinspace \alpha_{|_{\mathfrak{h}'(1)}} \neq 0 \right\} = \left\{ e_{1} - e_{j} \thinspace ; \thinspace 2 \leq j \leq p+q\right\}.
\end{equation*} 
So,
\begin{eqnarray*}
\prod\limits_{\underset{\alpha_{|_{\mathfrak{h}'(1)}} \neq 0}{\alpha > 0}} (h^{\frac{\alpha}{2}} - h^{-\frac{\alpha}{2}}) & =  & \prod\limits_{k = 2}^{p+q} (h^{\frac{1}{2}}_{1}h^{-\frac{1}{2}}_{k} - h^{-\frac{1}{2}}_{1}h^{\frac{1}{2}}_{k}) =  \prod\limits_{k = 2}^{p+q} h^{-\frac{1}{2}}_{1}h^{-\frac{1}{2}}_{k}(h_{1} - h_{k}) \\
& = & h^{-\frac{p+q-1}{2}}_{1} \prod\limits_{k = 2}^{p+q} h^{-\frac{1}{2}}_{k} \prod\limits_{k=2}^{p+q} (h_{1} - h_{k}) = h^{-\frac{p+q-2}{2}}_{1} \prod\limits_{k = 2}^{p+q} h^{-\frac{1}{2}}_{k} \prod\limits_{k=2}^{p+q} (h_{1} - h_{k})
\end{eqnarray*}
\noindent According to \eqref{FormuleFinaleU(p,q)},
\begin{equation*}
\Theta_{\Pi'}(h) = C \sum\limits_{\sigma \in \mathscr{W}(\K', \mathfrak{h}')} \cfrac{h^{\frac{p+q}{2}-\mu' - 1}_{\sigma(1)} \prod\limits_{k=1}^{p+q} h^{\frac{1}{2}}_{k}}{\prod\limits_{k=2}^{p+q} (h_{\sigma(1)} - h_{\sigma(k)})}.
\end{equation*}

\begin{lemme}

The following equality holds
\begin{equation}
\sum\limits_{\sigma \in \mathscr{W}(\K', \mathfrak{h}')} \cfrac{h^{\frac{p+q}{2}-\mu' - 1}_{\sigma(1)}}{\prod\limits_{k=2}^{p+q} (t_{\alpha(1)} - t_{\alpha(k)})} = (p-1)!q! \sum\limits_{b = 1}^{p} \cfrac{h^{\frac{p+q}{2}-\mu' - 1}_{b}}{\prod\limits_{\underset{a \neq b}{a = 1}}^b (h_{b} - h_{a})}
\end{equation}

\end{lemme}

\begin{proof}

First,
\begin{equation*}
\sum\limits_{\sigma \in \mathscr{W}(\K', \mathfrak{h}')} \cfrac{h^{\frac{p+q}{2}-\mu' - 1}_{\sigma(1)}}{\prod\limits_{k=2}^{p+q} (h_{\sigma(1)} - h_{\sigma(k)})} = \sum\limits_{k = 1}^{p} \sum\limits_{\underset{\sigma(1) = k}{\alpha \in \mathscr{W}(\K', \mathfrak{h}')}} \cfrac{h^{\frac{p+q}{2}-\mu' - 1}_{k}}{\prod\limits_{j=2}^{p+q} (h_{k} - h_{\sigma(j)})}
\end{equation*}
Then, for all $k \in [|1, p|]$ and $\sigma \in \mathscr{W}(\K', \mathfrak{h}')$ such that $\sigma(1) = k$, we have $\sigma\left(\left\{2, \ldots, p+q\right\}\right) = \left\{1, \ldots, p+q\right\} \setminus \left\{k\right\}$, and
\begin{equation*}
\prod\limits_{j=2}^{p+q} (h_{k} - h_{\sigma(j)}) = \prod\limits_{\underset{a \neq k}{a=1}}^{p+q} (h_{k} - h_{a}).
\end{equation*}
Finally,
\begin{equation*}
 \sum\limits_{k = 1}^{p} \sum\limits_{\underset{\sigma(1) = k}{\sigma \in \mathscr{W}(\K', \mathfrak{h}')}} \cfrac{h^{\frac{p+q}{2}-\mu' - 1}_{k}}{\prod\limits_{j=2}^{p+q} (h_{k} - h_{\sigma(j)})} = (p-1)!q! \sum\limits_{b = 1}^{p} \cfrac{h^{\frac{p+q}{2}-\mu' - 1}_{b}}{\prod\limits_{\underset{a \neq b}{a = 1}}^{p+q} (h_{b} - h_{a})}.
 \end{equation*}
 
\end{proof}

\noindent Up to a constant, the character $\Theta_{\Pi'}(h)$ is given by the following formula:
\begin{equation}
\Theta_{\Pi'}(h) = \prod\limits_{k=1}^{p+q} h^{\frac{1}{2}}_{k}\sum\limits_{b = 1}^{p} \cfrac{h^{\frac{p+q}{2}-\mu' - 1}_{b}}{\prod\limits_{\underset{a \neq b}{a=1}}^{p+q} (h_{b} - h_{a})}.
\label{car1}
\end{equation}

\noindent Now, we treat the case $m = 0$. We have $\pr_{0}(h) = h_{p+q}$ and
\begin{equation*}
\left\{\alpha \in \Phi^{+}(\mathfrak{g}', \mathfrak{h}') \thinspace ; \thinspace \alpha_{|_{\mathfrak{h}'(0)} \neq 0}\right\} = \left\{e_{k} - e_{p+q} \thinspace ; \thinspace 1 \leq k \leq p+q-1\right\}.
\end{equation*}
Then,
\begin{eqnarray*}
\prod\limits_{\underset{\alpha_{|_{\mathfrak{h}'(0)}} \neq 0}{\alpha > 0}} (h^{\frac{\alpha}{2}} - h^{-\frac{\alpha}{2}}) & =  & \prod\limits_{k = 1}^{p+q-1} (h^{\frac{1}{2}}_{k}h^{-\frac{1}{2}}_{p+q} - h^{-\frac{1}{2}}_{k}h^{\frac{1}{2}}_{p+q}) =  \prod\limits_{k = 1}^{p+q-1} h^{-\frac{1}{2}}_{p+q}h^{-\frac{1}{2}}_{k}(h_{k} - h_{p+q}) \\
& = & h^{-\frac{p+q-1}{2}}_{p+q} \prod\limits_{k = 1}^{p+q-1} h^{-\frac{1}{2}}_{k} \prod\limits_{k=1}^{p+q-1} (h_{k} - h_{p+q}) = h^{-\frac{p+q-2}{2}}_{p+q} \prod\limits_{k = 1}^{p+q-1} h^{-\frac{1}{2}}_{k} \prod\limits_{k=1}^{p+q-1} (h_{k} - h_{p+q})
\end{eqnarray*}
i.e.
\begin{equation}
\Theta_{\Pi'}(h) = C \sum\limits_{\alpha \in \mathscr{W}(\K', \mathfrak{h}')} \cfrac{h^{\frac{p+q}{2}-\mu' - 1}_{\sigma(p+q)} \prod\limits_{k=1}^{p+q} h^{\frac{1}{2}}_{j}}{\prod\limits_{k=1}^{p+q-1} (h_{\sigma(k)} - h_{\sigma(p+q)})}.
\label{car2}
\end{equation}

\begin{lemme}
\begin{equation*}
\sum\limits_{\sigma \in \mathscr{W}(\K', \mathfrak{h}')} \cfrac{h^{\frac{p+q}{2}-\mu' - 1}_{\sigma(p+q)} }{\prod\limits_{k=1}^{p+q-1} (t_{\sigma(k)} - t_{\sigma(p+q)})} = p!(q-1)! \sum\limits_{b=p+1}^{p+q} \cfrac{h^{\frac{p+q}{2}-\mu' - 1}_{b}}{\prod\limits_{\underset{a \neq b}{a=1}}^{p+q} (h_{b} - h_{a})}.
\end{equation*}
\end{lemme}
\noindent Finally, the character $\Theta_{\Pi'}$ of $\Pi'$ is given, up to a constant, by the formula:
\begin{equation}
 \Theta_{\Pi'}(h) = C p!(q-1)! \prod\limits_{k=1}^{p+q} h^{\frac{1}{2}}_{j} \sum\limits_{b=p+1}^{p+q} \cfrac{h^{\frac{p+q}{2}-\mu' - 1}_{b}}{\prod\limits_{\underset{a \neq b}{a=1}}^{p+q} (h_{b} - h_{a})}.
 \label{car2}
 \end{equation}
 
 \noindent Now, we prove that the character formula given in Theorem \ref{TheoremU(p,q)} is independent of the choice of $m \in [|0, 1|]$.
 
 \begin{prop}
 
For all $\lambda_{1} \in [|-q+1, p-1|]$, we have:
 \begin{equation}
\sum\limits_{b=p+1}^{p+q} \cfrac{h^{\frac{p+q}{2}-\mu' - 1}_{b}}{\prod\limits_{\underset{a \neq b}{a=1}}^{p+q} (h_{b} - h_{a})} = - \sum\limits_{b = 1}^{p} \cfrac{h^{\frac{p+q}{2}-\mu' - 1}_{b}}{\prod\limits_{\underset{a \neq b}{a=1}}^{p+q} (h_{b} - h_{a})}.
\end{equation}

\end{prop}

\begin{rema}

The previous proposition can be reformulated as follows: for all $k \in [|0, p+q-2|]$, the following equality holds: 
 \begin{equation}
\sum\limits_{b=p+1}^{p+q} \cfrac{h^{k}_{b}}{\prod\limits_{a \neq b} (h_{b} - h_{a})} = - \sum\limits_{b = 1}^{p} \cfrac{h^{k}_{b}}{\prod\limits_{a \neq b} (h_{b} - h_{a})}
\label{Egalite U(1, C)}
\end{equation}

\end{rema}

\begin{proof}

We first multiply both sides of \eqref{Egalite U(1, C)} by $\prod\limits_{1 \leq i < j \leq p+q} (h_{i} - h_{j})$. Then, \eqref{Egalite U(1, C)} can be written as
\begin{equation*}
\sum\limits_{b=p+1}^{p+q} (-1)^{b-1} h^{k}_{b}\prod\limits_{\underset{b \neq a, b \neq c}{1 \leq a < c \leq p+q}} (h_{a} - h_{c}) = - \sum\limits_{b = 1}^{p} (-1)^{b-1}h^{k}_{b}\prod\limits_{\underset{b \neq a, b \neq c}{1 \leq a < c \leq p+q}} (h_{a} - h_{c}),
\end{equation*}
which is equivalent to
\begin{equation*}
\sum\limits_{b=1}^{p+q} (-1)^{b-1} h^{k}_{b}\prod\limits_{\underset{b \neq a, b \neq c}{1 \leq a < c \leq p+q}} (h_{a} - h_{c}) = 0
\end{equation*}
For each $k \in [|0, p+q-2|]$, we denote by $M^{k}$ the matrix in $\M(p+q, \mathbb{R})$ defined by the following equalities:
\begin{itemize}
\item $M^{k}_{i, j} = h^{j-1}_{i}$ if $j \leq k$
\item $M^{k}_{i, k+1} = M_{i, k+2} = t^{k}_{i}$
\item $M^{k}_{i, j} = t^{j-2}_{i}$ if $j > k+2$
\end{itemize}
i.e.
\begin{equation*}
\begin{pmatrix} 1 & h_{1} & \ldots & h^{k-1}_{1} & h^{k}_{1} & h^{k}_{1} & h^{k+1}_{1} & \ldots & h^{p+q-2}_{1} \\ \vdots & & & & \vdots &  & & & \vdots \\ 1 & h_{p+q} & \ldots & h^{k-1}_{p+q} & h^{k}_{p+q} & h^{k}_{p+q} & h^{k+1}_{p+q} & \ldots & h^{p+q-2}_{p+q}
\end{pmatrix}.
\end{equation*}
 
\noindent The $(k+1)th$ et $(k+2)th$ columns are equals, so for every $k \in [|0, p+q-2|]$, the determinant of $M^{k}$ is zero. Developing the determinant of $M^{k}$ with respect to the $(k+1)th$-column, we conclude:
\begin{equation*}
0 = \det(M_{k}) = (-1)^{k} \sum\limits_{b=1}^{p+q} (-1)^{b-1} h^{k}_{b}\prod\limits_{\underset{b \neq a, b \neq c}{1 \leq a < c \leq p+q}} (h_{a} - h_{c}).
\end{equation*}

\end{proof}

\nocite{TOM1, TOM2, TOM3, TOM4, TOM5, TOM6, TOM7, TOM8, TOM9, TOM10, TOM11, HOW1, BOU, VER, ROS, WEIL, DUFLO, VAR, LI, HARISH1, HARISH2}

\bibliographystyle{plain}

\end{document}